\documentclass[11pt]{article}
\usepackage[a4paper,vmargin={3.5cm,3.5cm},hmargin={2.5cm,2.5cm}]{geometry}
\usepackage[T1]{fontenc}
\usepackage[utf8]{inputenc}
\usepackage{ stmaryrd }
\usepackage{mathtools}
\usepackage{graphicx}
\usepackage{framed}
\usepackage[normalem]{ulem}
\usepackage{amsmath,amsthm}
\usepackage{physics}
\usepackage{tikz}
\usepackage{mathdots}
\usepackage{yhmath}
\usepackage{cancel}
\usepackage{color}
\usepackage{siunitx}
\usepackage{array}
\usepackage{multirow}
\DeclareMathOperator{\lex}{lex}
\DeclareMathOperator*{\argmax}{argmax}

\DeclareMathOperator*{\argmaxxlex}{\overset{[2],\lex}{\argmax}}
\DeclareMathOperator*{\argmaxlex}{\overset{\lex}{\argmax}}
\DeclareMathOperator*{\maxlex}{\overset{\lex}{\max}}
\DeclareMathOperator*{\maxxlex}{\overset{[2],\lex}{\max}}

\usepackage{amsthm}
\usepackage{amssymb}
\usepackage{gensymb}
\usepackage{tabularx}
\usepackage{extarrows}
\usepackage{booktabs}
\usetikzlibrary{fadings}
\usetikzlibrary{patterns}
\usetikzlibrary{shadows.blur}
\usetikzlibrary{shapes}
\usepackage{bbm}
\usepackage{mathtools}
\usepackage{forest}
\usepackage{thmtools}
\usepackage{thm-restate}
\usepackage{enumitem}
\usepackage{hyperref}

\usepackage{cleveref}

\declaretheorem[name=Theorem,numberwithin=section]{theorem}
\newtheorem{coro}{Corollary}[section]
\newtheorem{prop}{Proposition}[section]
\newtheorem{lemma}[prop]{Lemma}

\newtheorem{definition}[prop]{Definition}
\newtheorem{conj}{Conjecture}
\forestset{multiple directions/.style={for tree={#1}, phantom, for relative level=1{no edge, delay={!c.content/.pgfmath=content("!u")}, before computing xy={l=0,s=0}}},
    multiple directions/.default={},
    grow subtree/.style={for tree={grow=#1}}, 
    grow' subtree/.style={for tree={grow'=#1}}}
\usepackage{tikz}
\usepackage{capt-of}
\usepackage{tikz-cd}
\newcommand{\eps}{\varepsilon}
\usepackage[maxbibnames=99]{biblatex}
\bibliography{References.bib}
\DeclareMathOperator{\perfE}{perf_E}
\DeclareMathOperator{\perf}{perf}
\DeclareMathOperator{\perfV}{perf_V}
\begin{document}
\title{Optimal matching under size priority}
\author{Nathanaël Enriquez, Mike Liu, Laurent Ménard and Vianney Perchet}

\theoremstyle{remark}
\newtheorem{remark}{Remark}

\vfill

\maketitle

\vfill

\begin{abstract}
Past studies on the local limit of maximal weight matchings in edge-weighted large random graphs rely fundamentally on the assumption that the weights are atomless, which ensures that the maximal weight matching is unique. This excludes de facto maximal size matchings that correspond to equal edge-weights.
In this work, we overcome this difficulty by assigning i.i.d.~atomless weights to edges and choosing the maximal size matching that maximises the weight. We call these doubly constrained matchings \emph{optimal matchings}.
The natural generalisation of optimal matchings for infinite unimodular random graphs are unimodular matchings of maximal density at the root that maximise the expected weight at the root when it is matched. 

For unimodular Bienaymé-Galton-Watson (UBGW) trees and for a broad class of weight distributions, we show existence and uniqueness in law of such matchings. We also prove that if a sequence of finite random weighted graphs converges locally to an UBGW tree with i.i.d.~weights, then there exists a sequence of matchings on the finite graphs that converges locally to the optimal matching on the limiting tree.
Finally, we identify a regime, depending only on the offspring distribution of the limiting tree, in which correlations between edge states in the optimal matching decay exponentially with their graph distance. In this regime, we strengthen the previous convergence to the convergence of optimal matchings of the finite graphs.
As a by-product, we can explicitly compute the asymptotic densities of edges that belong to all maximal-density matchings, and of edges that belong to none.
\end{abstract}

\vfill


\paragraph{Mathematics Subjects Classification:} 05C70, 05C82, 60C05, 60K35.

\vfill

\newpage

\tableofcontents

\newpage

\section{Introduction and main results}

A matching of a graph $G$ is a collection of edges of $G$ with no common vertices. A matching $M$ of finite graph $G$ is said to be \textit{maximal} if it maximizes the cardinality of $M$ among the set of matchings of $G$.
The study of maximal matchings on sparse graphs dates back to the foundational paper by Karp and Sipser~\cite{Karp} where they provide the asymptotic size of maximal matchings on sparse Erdős–Rényi graphs $G(n,\frac{c}{n})$ as $n \rightarrow +\infty$ for fixed $c>0$. More recently, Bordenave, Lelarge and Salez~\cite{bordenave2012matchings} established an asymptotic formula for the cardinality of maximal matchings on sparse graphs that converge locally to a locally finite tree $\mathbb T$. These results do not study the geometry of maximal matchings for a good reason: There are indeed typically an exponential (in the size of the graph) number of maximal matchings. This suggests that there are many possible geometries for the local limits of such matchings.

In a previous article by the same authors~\cite{enriquez2024optimalunimodularmatching} the geometry of matchings maximising the sum of its edge weights with no consideration for the size is studied in depth. Namely, the maximal weight matching converges locally to a unimodular matching whose distribution is explicitly described. The method, originally devised by Aldous~\cite{aldous2000zeta2,Aldous1992AsymptoticsIT}, relies on using a message passing algorithm that finds such matchings on trees.

All these previous works rely fundamentally on the assumption that the weights are atomless, which ensures that the maximal weight matching is unique. This excludes de facto the study of maximal matchings.
In this work, we overcome this difficulty by assigning i.i.d.~atomless weights to edges and choosing the maximal matching that maximises the weight. We call these doubly constrained matchings \emph{optimal} matchings.

\subsection{Optimal unimodular matchings}
Informally (see Section~\ref{sec:def} for precise definitions), a random matching on the unimodular tree $\mathbb T$ rooted at a distinguished oriented Vertex $o \in V(\mathbb{T})$ is unimodular if its law is invariant by re-rooting. If we assign  i.i.d. weights $w=(w(e))_{e\in E(\mathbb T)}$ with law $\omega$ to the edges of $\mathbb T$, we say that a unimodular matching $\mathbb M$ of the weighted tree is optimal if the quantity 
\begin{equation}
    \label{eq:perfE}
    \perfE(\mathbb{T},o,w,\mathbb{M})=\left( \mathbb{P}\left( \exists v \sim o, \{o,v\} \in \mathbb{M}\right),\mathbb{E}\left[ \sum_{v\sim o}w(o,v)\mathbbm{1}_{\{o,v\} \in \mathbb{M}}\right]\right)
\end{equation}
is maximal in \emph{lexicographical order} among all unimodular matchings of $(\mathbb T, o, w)$. Note that for finite graphs, when $o$ is chosen uniformly among the vertices, this performance is maximal when the matching is maximal in size and maximal in total weight under this constraint.

Before stating our main result, we need to define more precisely Unimodular Bienaymé Galton Watson random trees (UBGW). Let $\pi$ be a distribution on non-negative integers with finite expectation and generating function $\phi$. We denote by $\hat{\pi}$ the size-biased version of $\pi$, whose generating function is $\hat \phi$, defined by $\hat \phi(x) = \frac{\phi'(x)}{\phi'(1)}$. We say that a rooted tree $(\mathbb{T},o)$ is a UBGW with offspring distribution $\pi$ if the law of the tree can be obtained as follows. Draw a sequence of independent vertex-rooted independent Bienaymé-Galton-Watson trees with offspring distribution $\hat{\pi}$ (that is, each vertex of these trees have an independent number of children distributed according to $\hat \pi$). Independently draw a random number $N$ with distribution $\pi$.  The tree $(\mathbb{T},o)$ is then obtained by connecting the root of each of the $N$ first Bienaymé-Galton-Watson trees to an additional vertex that we call $o$.

Our first theorem states the existence and uniqueness in law of unimodular matchings on $\mathbb{T}$ that are optimal in the sense of performance~\eqref{eq:perfE}:

\begin{theorem}\label{maintheorem}
    Let $\pi$ be a  reproduction law and $\omega$ a distribution on $\mathbb{R}$ such that:
    \begin{itemize}
        \item both $\pi$ and $\omega$ have finite expectation,
        \item the law $\omega$ is atomless and admits a density on some open interval,
        \item denoting by $\phi$ the generating function of $\pi$ and setting $\hat{\phi}= \frac{\phi'}{\phi'(1)}$, the function
        \[t \mapsto \hat{\phi}(1-\hat{\phi}(1-t))\]
        has a finite number of fixed points.
    \end{itemize}
    Then there exists a unique (in law) unimodular optimal matching $(\mathbb{T},o,\mathbb{M}_{\mathrm{opt}})$ whose marginal tree $(\mathbb{T},{o})$ is a Bienaymé-Galton-Watson tree with reproduction law $\pi$ and i.i.d.~weights of law $\omega$. Furthermore the distribution of $(\mathbb{T},{o},\mathbb{M}_{\mathrm{opt}}(\mathbb{T}))$ is described in Proposition~\ref{prop:Zconstruction}.
\end{theorem}

\begin{remark} \label{rem:exceptions}
    The third assumption about $\hat \phi$ induces notable exclusions. Indeed, if one assumes that $\hat \phi$ is analytic at $1$ (this includes in particular laws with exponential moments), then the function $\hat\phi ( 1- \ \hat{\phi} (1-x))$ has infinitely many fixed points iff it is equal to $x$ on the whole interval $[0,1]$. The family of geometric laws
    \[
    \hat \phi (x) = \frac{px}{1-(1-p)x}
    \]
    falls into that exclusion. The family given by $\hat \phi (x) = 1-(1-x^k)^{1/k}$ for $k \in \mathbb N^*$ is another example. All such exceptions are also critical for assumption \eqref{eq:subcriticalcond} introduced below. We will say a bit more on these special cases in Remark~\ref{rem:BLS}.
\end{remark}

Our next main result states that if a sequence of random weighted random graphs converges in the Benjamini-Schramm sense to a unimodular Bienaymé-Galton-Watson tree, then this optimal matching is the limit of a sequence of matchings on the finite graphs. This includes classical models of random graphs such as Erd\H{o}s-R\'enyi random graphs, regular random graphs or configuration models that all converge to unimodular versions of Bienaymé-Galton-Watson random trees. See for instance the books by van der Hofstad~\cite{vanderHofstad1,vanderHofstadBook2}, the article by Benjamini, Lyons and Schramm~\cite{BENJAMINI_LYONS_SCHRAMM_2015}, or the monograph by Aldous and Lyons~\cite{aldous2018processes}, and references therein.

\begin{theorem}\label{maintheorem2}
    Let $G_n=(V_n,E_n)$ be a sequence of random graphs and choose a root $\overset{\rightarrow}{o}_n$ uniformly among the vertices.
    Suppose that the sequence of rooted weighted random graphs $(G_n,o_n)$ converges locally to an (vertex-rooted) Unimodular Bienaymé-Galton-Watson (UBGW) tree $(\mathbb{T},o)$ with reproduction law $\pi$. Assume further that the ratio $\frac{2|E_n|}{|V_n|}$ converges in probability to the expectation of $\pi$ and that the hypothesis of Theorem~\ref{maintheorem} holds on the UBGW tree. Now assign i.i.d weights $(w_n(e))_{e \in E(G_n)}$ of distribution $\omega$ on the edges of $G_n$. We have
    \begin{enumerate}
        \item There exists a sequence $(M_n)_n$ of matchings of $G_n$ such that $(G_n,o_n,M_n)$ converges locally to $(\mathbb{T},{o},\mathbb{M}_{\mathrm{opt}})$ given by Theorem~\ref{maintheorem}.
        In particular, we have that:

         \[
        \lim_{n \to \infty} 
        \left(  \mathbb{P}\left( \exists v \sim o_n, \{o_n,v\} \in M_n\right),\mathbb{E}\left[ \sum_{v\sim o_n}w_n(o_n,v)\mathbbm{1}_{\{o_n,v\} \in M_n}\right]\right) = \perfE(\mathbb{T},o,\mathbb{M}_{\mathrm{opt}}),
        \]
        
        \item Reciprocally, for any sequence $(M_n)_n$ of matchings of $G_n$, take $(a,b)$ to be any subsequential limit of
        \[ \left(  \mathbb{P}\left( \exists v \sim o_n, \{o_n,v\} \in M_n\right),\mathbb{E}\left[ \sum_{v\sim o_n}w_n(o_n,v)\mathbbm{1}_{\{o_n,v\} \in M_n}\right]\right) ,\]
        then we must have that 
        \[ (a,b)\overset{\mathrm{lex}}{\leq} \perfE(\mathbb{T},o,\mathbb{M}_{\mathrm{opt}}).
        \]
    \end{enumerate}

\end{theorem}

\begin{remark}
  Let us stress that, in the first statement of Theorem~\ref{maintheorem2}, the matching $M_n$ is not necessarily a maximal matching of $G_n$, but only asymptotically maximal. An interesting and delicate question is then the existence of a sequence of maximal matchings of the graphs $G_n$ that also have asymptotically optimal performance. This question will be completely answered in a regime that we call subcritical in the next section. We conjecture that it holds in general.
\end{remark}

\begin{remark} \label{rem:BLS}
    Bordenave, Lelarge and Salez proved in~\cite{salez2011cavity} that, if $G_n$ is a sequence of \textbf{deterministic} graphs that converges locally in distribution when rooted uniformly, then there is almost sure convergence of the normalised size of the maximum size matching. 
    In the case of configuration models, which converge almost surely to UBGW trees under uniform rooting (see e.g. \cite{BordenaveLN}),
    the asymptotic size of maximal matchings on $G_n$ converges almost surely to \[2-\max_{u \in [0,1]}F_\pi(u)\]
    with
    \[ F_\pi: x \mapsto \phi(1-x)+\phi(1-\hat{\phi}(1-x))+x\hat{\phi}(1-x)\phi'(1). \]
    The uniqueness of $M_{\mathrm{opt}} (\mathbb T)$ established in Theorem~\ref{maintheorem} ensures that this quantity is equal to $\mathbb P \left( \exists v \sim o, \{o,v\} \in \mathbb{M}_{\mathrm{opt}} \right)$.
    The first statement of the Theorem contains a weaker convergence in expectation when the graphs $G_n$ are \textbf{random} and converge in distribution. 
    Note that this result also applies for the exceptions of Remark~\ref{rem:exceptions}, and we can observe that, in these cases, maximal matchings are in fact asymptotically perfect matchings.
\end{remark}

\begin{remark} \label{rem:uniformmatch}
A classical way to study maximum size matchings on a finite graph is to introduce a temperature parameter $\beta$ and a Gibbs distribution on matchings of the form $\mu(M) \propto e^{-\beta |E(M)| }$ (see~\cite{bordenave2012matchings}). 
In the zero temperature limit, namely letting $\beta \to +\infty$, we recover a uniformly sampled maximum size matching. 
    We stress that our theorems are not about such measures. In fact, we will show in Proposition~\ref{prop:pasuniforme} that, in some strong sense, uniform maximum size matchings on UBGW trees are fundamentally different from maximum weight maximum size matchings on UBGW trees.
\end{remark}

    \subsection{The subcritical regime} 
    
    Our next statement gives a positive answer to the convergence question in a regime that we call \textit{local} or \textit{subcritical} for reasons that we will explain after the theorem. For a rooted graph $(G,o)$ we will write $B_H(G,o)$ to be the subgraph of $G$ induced by vertices with graph distance at most $H$ from $o$.
     
\begin{theorem}\label{th:reciproquegamarnik}
    Under the same notation and hypothesis as Theorem~\ref{maintheorem2}, assume furthermore that for any random variable $X$ taking values in $[0,1]$, we have that
    \begin{equation}\label{eq:subcriticalcond}
        \mathbb{E}\left[ \hat{\phi}'(1-X)\right] \hat{\phi}'\left(1-\mathbb{E}\left[\hat{\phi}(1-X)\right]\right) < 1.
    \end{equation}

    Also assume that $(G_n,o_n)$ converges locally in total variation to $(\mathbb{T},o)$, in the sense that for every $\varepsilon>0,H>0$, there exists a rank $N_{\eps,H}$ such that for any $n \geq N_{\eps,H}$, there is a coupling $(\tilde{G}_n,\tilde{o}_n,\tilde{\mathbb{T}},\tilde{o})$  between $(G_n,o_n)$ and $(\mathbb{T},o)$ such that 
    \begin{equation}\label{eq:TVlocal}
        \mathbb{P}\left(B_H(\tilde{G}_n,\tilde{o}_n)=B_H(\tilde{\mathbb{T}},\tilde{o})\right) > 1-\varepsilon.
    \end{equation}

    For any $G_n$, define $M_{\mathrm{opt}}(G_n)$ to be the maximal weight maximum size matching on $G_n$, then we have that 
    \[ (G_n,o_n,M_{\mathrm{opt}}(G_n)) \longrightarrow (\mathbb{T},o,\mathbb{M}_{\mathrm{opt}})  \]
    where the convergence is understood in the sense of local convergence in total variation.
\end{theorem}
\begin{remark}
    Note that condition~\eqref{eq:subcriticalcond} does not depend on the weight distribution and depends only on the reproduction law of the tree. We do not know if this condition is sharp and it is likely that the true condition depends on both $\pi$ and $\omega$.
\end{remark}

We call a reproduction satisfying condition~\eqref{eq:subcriticalcond} \emph{subcritical}.
For Erd\H{o}s-R\'enyi random graphs $G(n,\tfrac{c}{n})$, the sub-criticality condition is equivalent to $c<e$. This $e$-threshold is the object of several phase transitions for the Erd\H{o}s-R\'enyi model; see, for instance, the works of Karp-Sipser~\cite{Karp} for a geometric phenomenon or Coste-Salez~\cite{CS} for a spectral phenomenon. More related to our work, Gamarnik-Nowicki-Swirszcz~\cite{gamarnik2003maximum} proved that, for i.i.d.~exponential weights, independent sets can be approximated by local functions when $c<e$ and maximum weight matchings for all $c$. Theorem~\ref{th:reciproquegamarnik} is in fact a consequence of a generalisation of this locality property for optimal matchings and Bienaymé-Galton-Watson trees that we state now. See also the recent paper by Lam and Sen~\cite{correlationarnab} where they establish similar results for exponential weights and for graphs with degrees bounded by $3$, which remarkably include non tree-like graphs such as the hexagonal lattice. 
\begin{theorem} \label{th:decay}
    Let $(\mathbb{T},\overset{\rightarrow}{o})$ be an edge-rooted unimodular Bienaymé-Galton-Watson tree with reproduction law $\pi$.
    Set $\hat{\phi}$ to be the generating function of $\hat{\pi}$. Assume that $\hat{\phi}$ verifies condition~\eqref{eq:subcriticalcond}. Then there exists a sequence of determinist  functions $F_H$ such that:
    \begin{itemize}
        \item Each function $F_H$ is $H-$local
        \item Set \[ \rho_{\pi}:=\sup_{X \text{ r.v. in } [0,1]}\mathbb{E}\left[ \hat{\phi}'(1-X)\right] \hat{\phi}'\left(1-\mathbb{E}\left[\hat{\phi}(1-X)\right]\right)<1. \] Then there exists $C>0$ such that, for every $H >0$,
        \[\mathbb{P}\left((\mathbb{T},o,  \mathbb{M}_{\mathrm{opt}})\simeq F_H(\mathbb{T},o)\right)>1-C\rho_{\pi}^H.\]
    \end{itemize}
\end{theorem}

For a sequence of finite weighted graphs $G_n$, this gives an annealed exponential decay for the correlation between the states of two edges relatively to their distance. 
In particular, this implies asymptotic decorrelation of the state of two uniformly picked edges in $G_n$:

\begin{coro}\label{coro:indepasympt}
    Let $(\mathbb{T},o)$ be as in Theorem~\ref{th:decay} and let $(G_n,o_n)_{n\geq 0}$ be a sequence uniformly rooted finite graphs such that $(G_n,o_n)$ converges locally to $(\mathbb{T},o)$. Assume in addition that if $o_n'$ is a uniform edge chosen independently from $o_n$, then $(G_n,o_n')$ converges locally to an independent copy of $(\mathbb{T},o)$. 
    Then $(G_n,M_{\mathrm{opt}}(G_n),o_n)$ and $(G_n,M_{\mathrm{opt}}(G_n),o_n')$
    converge to two independent copies 
    of $(\mathbb{T},o,\mathbb{M}_{\mathrm{opt}})$.
\end{coro}

\begin{remark}
    The condition is satisfied for configuration models (see subsection~\ref{subsec:UBGW} for a definition). It is also not empty: One can consider for $c<1$, $G_n$ to be $n$ copies of the same Bienaymé-Galton-Watson tree with a Poisson reproduction law of parameter $c$. It is clear that $\pi$ satisfies condition~\eqref{eq:subcriticalcond}, $(G_n,o_n)$ converges locally to $(\mathbb{T},o)$ but two uniformly chosen vertices will see the same copy of the tree.
\end{remark}

\subsection{On mandatory and blocking edges}

We present now an application of our framework to the geometry  of the set of maximal matchings under a condition slightly weaker than~\eqref{eq:subcriticalcond}. Under this condition, there is a local convergence in total variation of the set of edges belonging to {\it all} the maximal matchings, that we call {\it mandatory edges}. Similarly, there is a local convergence of the set of edges belonging to {\it no} maximal matching, that we call {\it blocking edges}. Furthermore, the local limit will be explicitly described in a later section.

\begin{theorem}\label{thm:mandatory}
    Consider a sequence of \textbf{unweighted} random finite graphs $(G_n,o_n)$ that locally converge in total variation sense to a unimodular unweighted Bienaymé-Galton-Watson tree with reproduction law $\pi$.
    Assume that the generating function $\phi$ of $\pi$ verifies that 
    \begin{equation}
    \exists! \, t \in (0,1): \,
    t = \hat{\phi}(1-\hat{\phi}(1-t)).
    \label{eq:uniquefixedpoint} \tag{\ref{eq:subcriticalcond}'}
    \end{equation}
     Define $\mathcal{M}_{n,\max}$ as the set of maximum size matchings in $(G_n)$. Then the sequence
    \[ \left(G_n,o_n, \left( \mathbbm{1}_{e \in \bigcap_{M \in \mathcal{M}_{n,\max}}M}  \right)_{e\in E(\mathbb T)} , \left( \mathbbm{1}_{e \in \bigcap_{M \in \mathcal{M}_{n,\max}} M^{\complement}} \right)_{e\in E(\mathbb T)}\right)   \]
    converges locally in the total variation sense towards some limit $ (\mathbb T, o, (\mathrm{m}(e))_{e\in E(\mathbb T)}, (\mathrm{b}(e))_{e\in E(\mathbb T}))$. Furthermore, this limit can be approximated by local functions in the sense that there exists a sequence of explicit local functions $A_H$ defined on rooted graphs such that
    \[
    \left( \mathbb T, o, (\mathrm{m}(e))_{e\in E(\mathbb T)}, (\mathrm{b}(e))_{e\in E(\mathbb T)} \right) = \lim_{H \rightarrow \infty} A_H(\mathbb{T},o).
    \]
\end{theorem}
\begin{remark}
    Condition \eqref{eq:subcriticalcond} implies Condition \eqref{eq:uniquefixedpoint}. Indeed, by considering in condition \eqref{eq:subcriticalcond} constant random variables  $X$, we get that the derivative of $\hat{\phi}(1-\hat{\phi}(1-t))$ is smaller than 1 on the whole interval $[0,1]$.
\end{remark}
We leave the explicit form of the functions $A_H$ to Section~\ref{sec:subcritical}. This allows us to give a law of large numbers on the number of mandatory and blocking edges in $(G_n)$. 
\begin{coro} \label{coro:ERmb}
    Under the same condition and notation of Theorem~\ref{thm:mandatory}, we have that
    \begin{align}
        &\frac{1}{|E(G_n)|}\sum_{e \in E(G_n)} \mathbbm{1}_{e \in \bigcap_{M \in M_{n,\max}}M} \overset{\mathbb{P}}{\longrightarrow} \gamma^2, \\
        &\frac{1}{|E(G_n)|} \sum_{e \in E(G_n)} \mathbbm{1}_{e \in \bigcap_{M \in M_{n,\max}}M^{\complement}} \overset{\mathbb{P}}{\longrightarrow} (1-\gamma)^2.
    \end{align}
    where the convergence holds in probability and $\gamma$ is the unique fixed point of $t \mapsto \hat{\phi}(1-t)$.
\end{coro}

We conjecture that an extension of this result holds in expectation for generic offspring distributions, including laws that do not verify condition~\ref{eq:subcriticalcond} or condition~\ref{eq:uniquefixedpoint}. See Conjecture~\ref{conj:ER}, Conjecture~\ref{conj:desirables} and Conjecture~\ref{conj:desirablesgeom}.

\subsection{A glimpse at our approach: Belief propagation and renormalisation}
Maximum weight matchings were studied in our previous article \cite{enriquez2024optimalunimodularmatching} with an approach based on belief propagation. Informally, directed message variables $\mathbf{Z}(u,v)_{(u,v) \in \overset{\rightarrow}{E}(\mathbb{T})}$ on the weighted tree $(\mathbb{T},o,w)$ satisfying
\[ \mathbf{Z}(u,v)=\max\left(0,\max_{\substack{u' \sim v \\ u' \neq u}} \left( w(v,u')-\mathbf{Z}(v,u') \right)\right)  \] 
were defined and studied. In particular, it was shown in Theorem~1 of \cite{enriquez2024optimalunimodularmatching} that there is a unique stationary distribution for such messages and that they correspond to the unique (in distribution) unimodular matching $(\mathbb{T},o,\mathbb{M}_{\max})$ on the tree that maximizes
\[
\mathbb E \left[ w(o) \mathbbm{1}_{o \in \mathbb M_{\mathrm{max}}} \right].
\]
The matching $\mathbb M_{\mathrm{max}}$ is constructed by the messages via the decision rule
\[
\{u,v\} \in \mathbb{M}_{\max} \Leftrightarrow \mathbf{Z}(u,v) + \mathbf{Z}(v,u) < w(u,v).
\]
See Subsection~\ref{subsec:premierpapier} for a more detailed presentation.

The main idea behind this paper is to introduce a renormalisation procedure on these messages when the weights are of the form $1+\eps w$. This defines a tree, messages and matching that we call $(\mathbb{T},o,w,\mathbb{M}_{\max}^{(\eps)})$ As $\eps \rightarrow 0$, we find that the corresponding messages $\mathbf{Z}^{\eps}$ can be split into a macroscopic component $x(u,v)$ and a microscopic component $Z(u,v)$, such that the messages $(x,Z)$ verify a two dimensional equivalent to the previous recursion:
\[(x,Z)(u,v)=\maxlex \left( (0,0), \maxlex_{ \substack{u' \sim v \\ u' \neq u}} \left( (1,w(v,u'))-(x(v,u'),Z(v,u')) \right) \right),   \]
where $\mathrm{lex}$ stands for the lexicographical order. The decision rule becomes
\begin{equation*}
    (u,v) \in \mathbb{M}_{\mathrm{opt}} \Leftrightarrow (x,Z)(u,v)+(x,Z)(v,u) \overset{\mathrm{\lex}}{<} (1,w(u,v)).
\end{equation*}
A detailed heuristic of this renormalisation procedure is given in Section~\ref{sec:heuristic}.
We will show in Section~\ref{sec:marginali} that under general conditions, it is possible to determine the explicit distribution of the macroscopic component and that there is a unique solution to the corresponding recursive distributional equation.

\bigskip

A by-product of our approach is that the maximum weight matching on the tree converges in law to the optimal matching:

\begin{theorem}\label{thm:convergenceZ}
    Let $(\mathbb{T},o,w)$ be a Unimodular Bienaymé-Galton-Watson tree with reproduction law $\pi$ and i.i.d. weights of law $\omega$.
    Assume that $\pi$ and $\omega$ have finite expectation, and that the generating function $\phi$ of $\pi$ is such that the function
    \[x \in [0,1] \mapsto \phi(1-x) + \phi(1-\hat\phi(1-x)) + \phi'(1) x \hat\phi(1-x)\]
    has at most two global maxima.
    For any $\eps>0$, denote by $(\mathbb{T},o,w, \mathbb{M}_{\max}^{(\eps)})$ the unique (in law) unimodular matching maximising
    \[  \mathbb{E}\left[ \sum_{v\sim o}(1+ \varepsilon w(o,v)) \mathbbm{1}_{ \{o,v\} \in \mathbb{M}_{\max}^{(\eps)} }\right],\]
    and by $(\mathbb{T},o,w,\mathbb{M}_{\mathrm{opt}})$ the unique (in law) optimal unimodular matching 
     given by Theorem~\ref{maintheorem}. We have the following weak local convergence
    \[(\mathbb{T},o,w, \mathbb{M}_{\max}^{(\eps)}) \underset{\eps \rightarrow 0}{\longrightarrow} (\mathbb{T},o,w, \mathbb{M}_{\mathrm{opt}}).\]
\end{theorem}

\bigskip

Combining Theorem~\ref{th:reciproquegamarnik} and Theorem~\ref{thm:convergenceZ}, we get the following corollary:

\begin{coro}\label{coro:convergenceepsilonn}
    Let $(G_n,w_n)$ be a sequence of weighted finite graphs.
    Under assumptions and notations of Theorem~\ref{th:reciproquegamarnik} and Theorem~\ref{thm:convergenceZ}, define $M_{\max}^{(\eps)}(G_n)$ to be the matching on $G_n$ maximising the sum of $1+w_n$ on its edges, then
    \[  (G_n,o_n, M_{\max}^{(\eps)}(G_n)  \]
    converges locally weakly as $n \to +\infty, \eps \to 0$ to
    \[ \left( \mathbb{T},o,\mathbb{M}_{\mathrm{opt}}(G_n) \right) . \]
\end{coro}
\bigskip

Our results are summarised in Figure~\ref{fig:diagram}

\begin{figure} \label{fig:diagram}
\begin{center}
\begin{tikzcd}[row sep=large, column sep=large]
\left(G_n,o_n,M^{(\eps)}_{\max}(G_n)\right) \arrow[rrr, "\text{\cite[Theorem 2]{enriquez2024optimalunimodularmatching}}","n \to +\infty"' ] \arrow[dd," \eps \to 0"'] \arrow[rrrdd, dashed, "{\begin{matrix} \text{Corollary~\ref{coro:convergenceepsilonn}} \\ (n,\varepsilon) \to (+\infty, 0) \end{matrix}}" description] &  &  & \left(\mathbb{T},o,\mathbb{M}_{\max}^{(\eps)} \right) \arrow[dd, "\text{Theorem~\ref{thm:convergenceZ}}"," \eps \to 0"'] \\
                                       &  &  &              \\
\left(G_n,o_n, M_{\mathrm{opt}}(G_n) \right) \arrow[rrr, dashed, "\text{Theorem~\ref{th:reciproquegamarnik}} ","n \to +\infty"']                          &  &  & \left(\mathbb{T},o,\mathbb{M}_{\mathrm{opt}}\right)          
\end{tikzcd}
\end{center}
\caption{Plain arrows are valid under general assumptions and dashed arrows are proved in the subcritical regime.}
\end{figure}
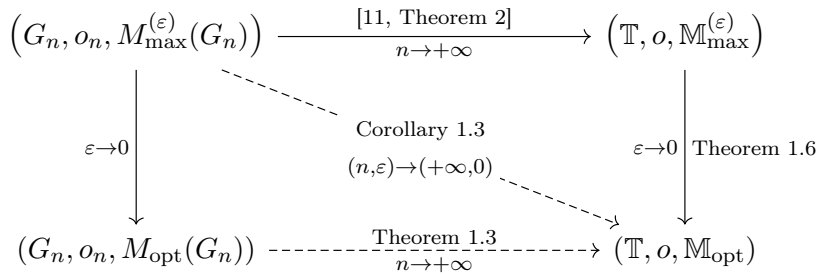

\subsection{Organisation of the paper}
Section~\ref{sec:def} presents the precise definitions and important notions used in the paper. In particular, we give definitions for unimodular graphs, local convergence, and optimal matchings.

As already mentioned, Section~\ref{sec:heuristic} gives detailed heuristics for our approach in the special case of UBGW trees. It also ends with an informal computation of the Karp-Sipser formula for the size of maximal matchings from the renormalised messages in the case of Poisson offspring distribution.

Section~\ref{sec:infinitematching} formally builds the joint distribution $\left(\mathbb{T},o,w,(x,Z)(u,v)_{(u,v) \in \overset{\rightarrow}{E}}\right)$ and shows that the matching defined by $\big\{\{u,v\}, (x,Z)(u,v)+(x,Z)(v,u)  \overset{\lex}{<} (1,w(u,v))\big\}$
defines the unique distribution of an optimal unimodular matching in $(\mathbb{T},\overset{\rightarrow}{o},w)$.
It concludes with the proofs of Theorem~\ref{maintheorem} and Theorem~\ref{maintheorem2} by combining the results of the previous sections.

Section~\ref{sec:marginali} carries out the computations necessary to determine the distribution of the first marginal of $(x,Z)$. 
We then prove that the distribution of the optimal matching characterises the distribution of the message variables. Combined with Theorem~\ref{maintheorem}, this allows one to prove the uniqueness in distribution of the messages (see Theorem~\ref{uniqueness}). The section concludes with the proof of Theorem~\ref{thm:convergenceZ}.

Section~\ref{sec:subcritical} studies the subcritical regime corresponding to assumption~\eqref{eq:subcriticalcond}. It contains the proofs of Theorem~\ref{th:reciproquegamarnik}, Theorem~\ref{th:decay} and Theorem~\ref{thm:mandatory}. The methods used to treat this regime are quite different than in the general case and this section can be read independently of Section~\ref{sec:infinitematching} and Section~\ref{sec:marginali}. The section concludes with the proof of Corollary~\ref{coro:convergenceepsilonn}.

Section~\ref{sec:final} discusses independent related questions and conjectures arising from this paper. It starts with a generalisation of our framework to minimal cost maximal gains matchings. We then show that as announced in Remark~\ref{rem:uniformmatch}, the optimal matchings we introduce are fundamentally different than uniformly sampled maximal size matchings. Finally, we discuss conjectures on mandatory and blocking edges related to Theorem~\ref{thm:mandatory} and Corollary~\ref{coro:ERmb}.

Finally, Section~\ref{sec:Appendix} contains the proof some technical lemmas.

\paragraph{Acknowledgements}
We thank Sébastien Martineau and James Martin for stimulating discussions during this work. We also thank Andrei Alpeev for discussions around Proposition~\ref{prop:russe}.

\section{Unimodular matched graphs}\label{sec:def}

This section is essentially drawn from \cite{enriquez2024optimalunimodularmatching} since we work in the same setting. 

We will work on rooted weighted graphs defined as follows.
\begin{definition}
     Let $G=(V,E)$ be a (locally finite) graph.
     We call \emph{vertex-rooted weighted graph} (resp. \emph{edge-rooted weighted graph}), the triplet $(G,o,w)$ , where $o \in V$ (resp. $o\in \overset{\rightarrow}{E}$) and $w$ is a function from $E$ to $\mathbb{R}$.
\end{definition}
\begin{remark}
We will use the notation $o$ for both vertex roots and directed edge roots. However, we will specify $\overset{\rightarrow}{o}$ or $o=(o_-,o_+)$ when distinguishing between vertex and edge roots will be important.
\end{remark}

It will be useful to add graph decorations, namely functions that map  $\overset{\rightarrow }{E}$ to $\mathbb{R}$.
\begin{definition} \label{def:decograph}
    Let $(G,o,w)$ be a vertex-rooted (resp.\ edge-rooted) weighted graph. Fix $I$ an integer and $(f_i)_{i \in \{1,...,I\}}$ some functions from $\overset{\rightarrow}{E}$ to $\overline{\mathbb{R}}$. We then say that
    $(G,o,w,(f_i)_{i \in \{1,...,I\}})$ is a \emph{decorated} vertex-rooted (resp.\ edge-rooted) weighted graph.  If $f_i$ is symmetrical then we will identify it with a map from $E$ to $\overline{\mathbb{R}}$.
\end{definition}

Since we are not interested in the labels of the vertices, we will work up to graph isomorphism.
We will say that $(G,o,w,(f_i)_{i \in \{1,...,I\}}) \simeq (G',o',w',(f_i')_{i \in \{1,...,I\}})$ if there exists some one-to-one function $g$ called graph isomorphism from $V$ to $V'$ such that if $g(o)=o'$ and for all $(u,v) \in V$, $w'(g(u),g(v))=w(u,v)$ and for all $ i \in \{1,...,I\}$, $f'_i(g(u),g(v))=f_i(u,v)$ .

\begin{definition}
Let $\mathcal{G}^{\star}$ be the space of locally finite decorated vertex-rooted weighted graphs up to isomorphism.
We will write $\mathcal{L}(\mathcal{G}^{\star})$ for the space of laws on this space. Similarly, we denote by  $\hat{\mathcal{G}}^{\star}$ and $\mathcal{L}(\hat{\mathcal{G}}^{\star})$ the corresponding edge-rooted space and laws.
\end{definition}
When we do not need to keep track of the weights, root, or some decorations of the graphs, we will denote elements of 
$\mathcal{G}^{\star}$ indifferently by $G$, $(G,o)$, $(G,o,w)$, $(G,o,w,(f_i)_{i \in \{1,...,I\}})$, ..., keeping only the quantities we are currently interested in.
\bigskip
\begin{remark}
    To be completely rigorous, the space $\mathcal{L}(\mathcal{G}^{\star})$ depends on the value $I$ and the topology that we put on the decorations $(f_i)_{i \in I}$. We will choose topology of uniform convergence for the decorations $f_i: \overset{\rightarrow}{E} \mapsto \overline{\mathbb{R}}$. Since $I$ is always uniformly bounded in this paper, we will simply take the product topology of $(f_i)_{i \in I}$.
\end{remark}
The topology of local convergence on $\mathcal{G}^{\star}$ and $\hat{\mathcal{G}^{\star}}$ was first introduced by Benjamini-Schramm~\cite{BenjaminiSchramm1} and by Aldous-Steele \cite{Aldous2004}, whose  precise definitions follow. We denote by $d_{gr}$ the graph distance between vertices of a graph. For any vertex-rooted decorated weighted graph $(G=(V,E),o,w,(f_i)_{i \in I})$ and $H>0$, the $H$-neighbourhood $B_H((G,o,w,(f_i)_{i \in \{1,...,I\}}))$ of $o$ in the graph is the vertex-rooted decorated weighted graph $(G_H,o,w_H,(f_{i,H})_{i \in \{1,...,I\}})$ with $G_H :=(V_H,E_H)$ such that
\begin{align*}
V_H &=\{v \in V, d_{gr}(o,v)\leq H\},\\
E_H &=\{ (u,v) \in E, (u,v)\in V_H^2 \},
\end{align*}
$w_H$ are the weights of edges in $E_H$ and, for all $i \in \{1,...,I\}$, the decoration $f_{i,H}$ is the decoration $f_{i}$ restricted to $\overset{\rightarrow}{E}_H$.

Let $(G,o,w,(f_{i})_{i \in \{1,...,I\}})$ and $(G',o',w',(f_i')_{i \in \{1,...,I\}}$ be two vertex-rooted decorated weighted graphs. Let $H \geq 0$ be the largest integer such that there exists a graph isomorphism $g$ from  $B_H(G,o)$ to $ B_H (G',o')$ such that $\| w'_H\circ g - w_H \|_{\infty} \leq \frac{1}{H}$ and $\| f'_{i,H} \circ g - f_{i,H} \|_{\infty} \leq \frac{1}{H}$ for all $i \in \{1, \ldots , I \}$. We set
\begin{equation} \label{eq:dloc}
   d_{\mathrm{loc}} \left( (G,o,w,(f_{i})_{i \in \{1,...,I\}}) , (G',o',w',(f_i')_{i \in \{1,...,I\}} \right) := \frac{1}{1+H}.
\end{equation}
The function $ d_{\mathrm{loc}}$ is a distance on $\mathcal G^\star$ and the space $(\mathcal{G}^{\star},d_\mathrm
{loc})$ is a Polish space. The topology induced by  $ d_{\mathrm{loc}}$ is called the local topology on $\mathcal{G}^{\star}$. Weak convergence on $\mathcal{L}(\mathcal{G}^{\star})$ for the local topology is called local weak convergence.
Similar definitions can be given for the edge-rooted versions.

\bigskip

We end this section with the definition of local and weakly local functions on the space of rooted graphs.

\begin{definition}\label{def:locfunc}
    Let $(X,d)$ be a metric space and $H\geq0$, we say that a map $F:\mathcal{G}^{\star} \mapsto (X,d)$ is $H-$local if $B_H(G,o) = B_H(G',o')$ implies $F(G,o)=F(G',o')$. We will also say that a function $F$ is local if there exists $H\geq0$ such that $F$ is $H-$local.
    
    We say that $F:\mathcal{G}^{\star}\mapsto (X,d)$ is weakly local if it is limit of local functions for the topology of simple convergence.
\end{definition}

\subsection{Vertex-rooted and edge-rooted unimodularity}

We now introduce the notion of unimodularity for (possibly decorated) random graphs (recall Definition~\ref{def:decograph}, in the following all graphs can be decorated).
To this end, we define the space of doubly rooted graphs up to isomorphism $\mathcal{G}^{\star\star}$ similarly as before, but with two distinguished roots. It will be simpler to give separately specific definitions for vertex-rooted and edge-rooted graphs.
 We refer to Aldous-Lyons \cite{aldous2018processes} for a comprehensive exposition on the topic.

\begin{definition}[Vertex-rooted unimodularity]
We say that a probability measure $\mu$ on decorated vertex-rooted graphs is unimodular if the following statement holds for every measurable $f : \mathcal{G}^{\star\star} \mapsto \mathbb{R}_+$:
     
    \[
    \int_{\mathcal G^{\star}} \left( \sum_{v \in V(G)} f(G,o,v)  \right) \, 
    \mathrm{d} \mu (G,o)
     = \int_{\mathcal G^{\star}} \left( \sum_{v \in V(G)} f(G,v,o)  \right) \, 
    \mathrm{d} \mu (G,o)   .\]
    The subspace of unimodular laws on vertex rooted-graph will be noted $\mathcal{L}_U(\mathcal{G}^{\star})$.
\end{definition}
The definition can be written alternatively as 
\[ \mathbb{E}_{(\mathbb{G},\mathbf{o}) \sim \mu  } \left[\sum_{v \in V(G)} f(\mathbb{G},\mathbf{o},v)\right]=\mathbb{E}_{(\mathbb{G},\mathbf{o}) \sim \mu }\left[\sum_{v \in V(G)} f(\mathbb{G},v,\mathbf{o})\right].    \]
For edge-rooted graphs, we will use the following definition: 
\begin{definition}[Edge-rooted unimodularity]
    Let $\hat{\mu}$ be a probability measure on decorated edge-rooted graphs.
    Let $(\mathbb{G},(\mathbf{o}_-,\mathbf{o}_+))$ be a random edge-weighted decorated graph with law $\hat \mu$. Let $\overset{\rightarrow}{e}_1$ be a uniformly picked directed edge of the form $(\mathbf{o}_+,\mathbf{v})$ for $\mathbf{v} \neq \mathbf{o}_-$. 
    We say that $\hat{\mu}$ is:
    \begin{itemize}
    \item stationary if $(\mathbb{G},(\mathbf{o}_-,\mathbf{o}_+))\overset{\mathcal{L}}{=}(\mathbb{G},\overset{\rightarrow}{e}_1)$
    \item revertible if $(\mathbb{G},(\mathbf{o}_-,\mathbf{o}_+))\overset{\mathcal{L}}{=}(\mathbb{G},(\mathbf{o}_+,\mathbf{o}_-))$ 
    \item unimodular if it is both revertible and stationary.
    \end{itemize}
    The subspace of unimodular laws on edge-rooted graphs will be noted $\mathcal{L}_U(\hat{\mathcal{G}}^{\star})$
\end{definition}

To simplify notation, we will say that a random rooted graph (as a random variable) is unimodular when its corresponding law is. 

\begin{remark}
    Fix $G$ a finite deterministic graph, and let $\mathbf o$ be a random vertex (resp.  $\overset{\rightarrow}{\mathbf o}$ an oriented edge). It is straightforward to check that $(G,\mathbf o)$ (resp. $(G,\overset{\rightarrow}{\mathbf o})$) is vertex-rooted unimodular (resp. edge-rooted unimodular) iff $\mathbf o$ is a uniform vertex (resp. $\overset{\rightarrow}{\mathbf o}$ is a uniform oriented edge). Hence, unimodular graphs can be viewed as generalisations of uniformly rooted graphs, which is one of the fundamental findings of \cite{BenjaminiSchramm1}. 
\end{remark}

We now present the classical transformation that maps a unimodular vertex-rooted graph to a unimodular edge-rooted graph. Heuristically, to transform a graph rooted at a uniform vertex into a graph rooted at a uniformly oriented edge, one has to pick an oriented edge starting at its root vertex, but this induces a bias by the degree of the root vertex. Indeed, a directed edge $(u,v)$ will be less likely to be the new root the greater the degree of $u$ is. This bias has to be taken into account to conserve unimodularity. This is done in the following transformations on the laws.

\begin{definition} \label{def:R}
    Take $\mu_v \in \mathcal{L}(\mathcal{G}^{\star})$ such that $0<m := \int_{\mathcal G^{\star}} \deg (o) \, \mathrm{d} \mu_v(G,o) < \infty$. Let $R(\mu_v) \in \mathcal{L}(\hat{\mathcal{G}}^{\star})$  be the unique measure such that,
    for every $f$ measurable from $\hat{\mathcal{G}}^{\star}$ to $\mathbb{R}_+$,
\[
\int_{\hat{\mathcal{G}}^{\star}}f(G,(o_-,o_+)) \, \mathrm{d} R(\mu_v)(G,(o_-,o_+)) =
\int_{\mathcal{G}*} 
 \frac{1}{m} 
\sum_{u \sim o} f(G,o,u) \, \mathrm{d} \mu_v(G,o).
\]
\end{definition}
\begin{remark}
The operator $R$ is the composition of two transformations. The first operator consists of choosing a uniformly oriented edge started at the root vertex, giving a measure $R_1(\mu_v) \in \mathcal{L}(\hat{\mathcal{G}}^{\star})$:
\[
\int_{\hat{\mathcal{G}}^{\star}} 
f(G,(o_-,o_+))\, \mathrm{d} R_1(\mu_v)(G,(o_-,o_+)))= 
\int_{\mathcal{G}*} 
 \frac{1}{\deg(o)} 
\sum_{u \sim o} f(\textbf{}G,o,u) \, \mathrm{d} \mu_v(G,o).
\]
The second step is then to cancel the bias by the degree of the vertex, giving $R(\mu_v) = R_2 \circ R_1 (\mu_v) \in \mathcal{L}(\hat{\mathcal{G}}^{\star})$:
    \[ \int_{\hat{\mathcal{G}}^{\star}}f(G,(o_-,o_+))\mathrm{d}R_2 \circ R_1 (\mu_v) (G,(o_-,o_+)) = \int_{\hat{\mathcal{G}}^{\star}} \frac{\deg(o_-)}{m} f(G,(o_-,o_+)) \, \mathrm{d}R_1 (\mu_v)(G,(o_-,o_+)).\]
    
Equivalently, to obtain an edge-rooted version of a vertex rooted random graph, one can consider the size-biased version of the original measure $\mathrm{d}\mu_v(G,o)$ (namely sample a graph with measure
$\frac{\deg(o)}{m}\mathrm{d}\mu_v(G,o)$) and then from this sampled graph, starting from the root, select an edge uniformly at random to get the edge measure $dR(\mu_v)$.
\end{remark}

\bigskip

The following proposition links vertex unimodularity and edge unimodularity.
\begin{prop}[Theorem 4.1 in \cite{aldous2018processes}]
Let $\mu_v \in \mathcal{L}(\mathcal{G}^{\star})$ be such that $0<\int_{\mathcal G^{\star}} \deg (o) \, \mathrm{d} \mu_v(G,o) < \infty$.
Then the measure $\mu_v$ is unimodular if and only if $R(\mu_v)$ is unimodular.
\end{prop}

\subsection{Unimodular Bienaymé-Galton-Watson trees}\label{subsec:UBGW}
In this subsection, we introduce Unimodular Bienaymé-Galton-Watson trees (UBGW) along with models of random graphs that converge locally in law to these trees. We will present both the vertex and the edge-rooted point of view, without weights. In either case, the weighted version with weight law $\omega$ corresponds to drawing $(w(e))_{e \in E}$ independently of law $\omega$.

Let $\pi$ be a probability measure on $\mathbb{Z}_+ = \{0,1, ... \}$ with finite expectation $m > 0$. Let $\hat{\pi}$ be the size-biased version of $\pi$, that is, $\forall k \geq 0$, $\hat{\pi} (k) =\frac{k}{m} \pi (k)$.

In the introduction, we defined the edge-rooted Unimodular Bienaymé Galton Watson Tree with reproduction law $\pi$. Let us recall the construction: take two independent copies of Bienaymé-Galton-Watson trees with offspring distribution  $\hat{\pi}$ with respective root vertex $o_-$ and $o_+$, and connect their roots by the oriented edge $(o_-,o_+)$. The resulting random tree $(\hat{\mathbb T}, (o_-,o_+))$ is an edge-rooted unimodular random graph. See Figure~\ref{fig:EUBGW} for an illustration.
\tikzset{every picture/.style={line width=0.75pt}} 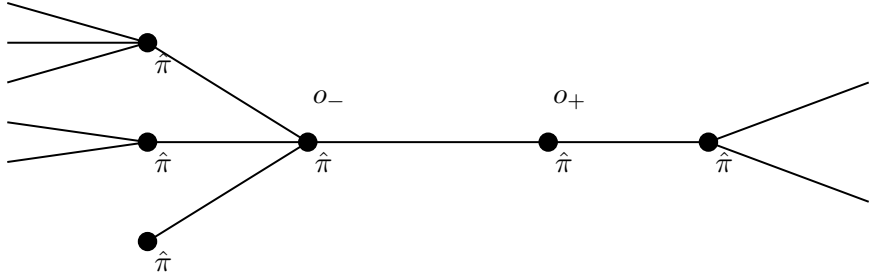
\begin{figure}[!t]
\centering
\begin{tikzpicture}[x=0.75pt,y=0.75pt,yscale=-1,xscale=1]
\draw    (250,140) -- (370,140) ;
\draw    (170,90) -- (250,140) ;
\draw    (170,140) -- (250,140) ;
\draw    (250,140) -- (170,190) ;
\draw    (370,140) -- (450,140) ;
\draw    (530,110) -- (450,140) ;
\draw    (530,170) -- (450,140) ;
\draw  [fill={rgb, 255:red, 0; green, 0; blue, 0 }  ,fill opacity=1 ] (245.65,140) .. controls (245.65,137.6) and (247.6,135.65) .. (250,135.65) .. controls (252.4,135.65) and (254.35,137.6) .. (254.35,140) .. controls (254.35,142.4) and (252.4,144.35) .. (250,144.35) .. controls (247.6,144.35) and (245.65,142.4) .. (245.65,140) -- cycle ;
\draw    (100,70) -- (170,90) ;
\draw    (100,110) -- (170,90) ;
\draw    (100,90) -- (170,90) ;
\draw    (100,130) -- (170,140) ;
\draw    (100,150) -- (170,140) ;
\draw  [fill={rgb, 255:red, 0; green, 0; blue, 0 }  ,fill opacity=1 ] (365.65,140) .. controls (365.65,137.6) and (367.6,135.65) .. (370,135.65) .. controls (372.4,135.65) and (374.35,137.6) .. (374.35,140) .. controls (374.35,142.4) and (372.4,144.35) .. (370,144.35) .. controls (367.6,144.35) and (365.65,142.4) .. (365.65,140) -- cycle ;
\draw  [fill={rgb, 255:red, 0; green, 0; blue, 0 }  ,fill opacity=1 ] (165.65,190) .. controls (165.65,187.6) and (167.6,185.65) .. (170,185.65) .. controls (172.4,185.65) and (174.35,187.6) .. (174.35,190) .. controls (174.35,192.4) and (172.4,194.35) .. (170,194.35) .. controls (167.6,194.35) and (165.65,192.4) .. (165.65,190) -- cycle ;
\draw  [fill={rgb, 255:red, 0; green, 0; blue, 0 }  ,fill opacity=1 ] (165.65,140) .. controls (165.65,137.6) and (167.6,135.65) .. (170,135.65) .. controls (172.4,135.65) and (174.35,137.6) .. (174.35,140) .. controls (174.35,142.4) and (172.4,144.35) .. (170,144.35) .. controls (167.6,144.35) and (165.65,142.4) .. (165.65,140) -- cycle ;
\draw  [fill={rgb, 255:red, 0; green, 0; blue, 0 }  ,fill opacity=1 ] (165.65,90) .. controls (165.65,87.6) and (167.6,85.65) .. (170,85.65) .. controls (172.4,85.65) and (174.35,87.6) .. (174.35,90) .. controls (174.35,92.4) and (172.4,94.35) .. (170,94.35) .. controls (167.6,94.35) and (165.65,92.4) .. (165.65,90) -- cycle ;
\draw  [fill={rgb, 255:red, 0; green, 0; blue, 0 }  ,fill opacity=1 ] (445.65,140) .. controls (445.65,137.6) and (447.6,135.65) .. (450,135.65) .. controls (452.4,135.65) and (454.35,137.6) .. (454.35,140) .. controls (454.35,142.4) and (452.4,144.35) .. (450,144.35) .. controls (447.6,144.35) and (445.65,142.4) .. (445.65,140) -- cycle ;

\draw (251,112.4) node [anchor=north west][inner sep=0.75pt]    {$o_{-}$};
\draw (371,112.4) node [anchor=north west][inner sep=0.75pt]    {$o_{+}$};
\draw (172,193.4) node [anchor=north west][inner sep=0.75pt]    {$\hat{\pi }$};
\draw (172,143.4) node [anchor=north west][inner sep=0.75pt]    {$\hat{\pi }$};
\draw (172,93.4) node [anchor=north west][inner sep=0.75pt]    {$\hat{\pi }$};
\draw (252,143.4) node [anchor=north west][inner sep=0.75pt]    {$\hat{\pi }$};
\draw (372,143.4) node [anchor=north west][inner sep=0.75pt]    {$\hat{\pi }$};
\draw (452,143.4) node [anchor=north west][inner sep=0.75pt]    {$\hat{\pi }$};
\end{tikzpicture}
\caption{A 2-Neighbourhood of an edge-rooted UBGW tree with the law of the number of children drawn on every vertex.} \label{fig:EUBGW}
\end{figure}

The vertex-rooted UBGW tree is the random tree $\mathbb T$ with the following law:
    \begin{itemize}
        \item The number of children of vertices of $\mathbb T$ are all independent.
        \item The number of children of the root $o$ is distributed according to $\pi$.
        \item Every non-root vertex has a number of children distributed according to $\hat \pi$, the sized biased version of $\pi$.
    \end{itemize}
The random tree $(\mathbb T, o)$ is a vertex-rooted unimodular random graph. See Figure~\ref{fig:VUBGW} for an illustration.
\tikzset{every picture/.style={line width=0.75pt}} 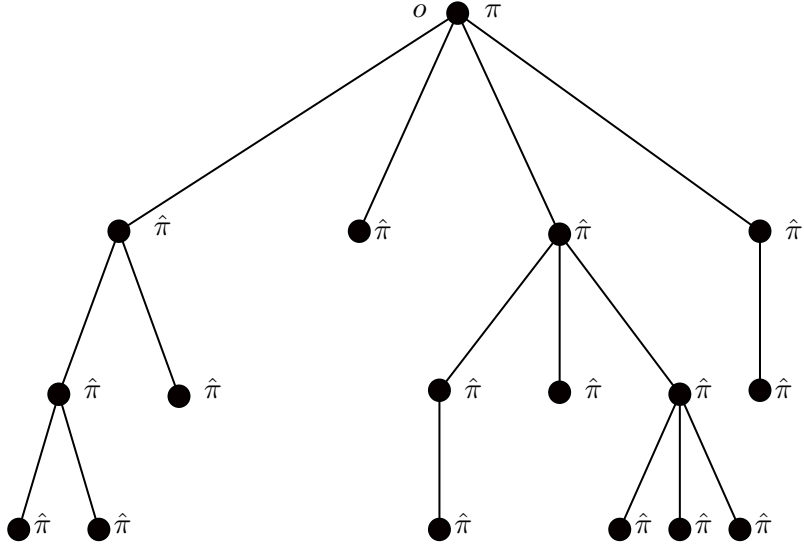
\begin{figure}[!ht]
\centering
\begin{tikzpicture}[x=0.75pt,y=0.75pt,yscale=-1,xscale=1]
\draw    (299.1,0.3) -- (130,110) ;
\draw    (299.1,0.3) -- (250,110) ;
\draw    (299.1,0.3) -- (350,111.5) ;
\draw    (299.1,0.3) -- (450,110) ;
\draw    (130,110) -- (100,192) ;
\draw    (130,110) -- (160,193) ;
\draw    (350,111.5) -- (290,190) ;
\draw    (350,111.5) -- (350,191) ;
\draw    (350,111.5) -- (410,192) ;
\draw    (450,190) -- (450,110) ;
\draw    (100,192) -- (80,260) ;
\draw    (100,192) -- (120,260) ;
\draw    (290,190) -- (290,260) ;
\draw    (410,192) -- (380,260) ;
\draw    (410,192) -- (410,260) ;
\draw    (410,192) -- (440,260) ;
\draw  [fill={rgb, 255:red, 7; green, 0; blue, 0 }  ,fill opacity=1 ] (124.77,110) .. controls (124.77,107.02) and (127.11,104.6) .. (130,104.6) .. controls (132.89,104.6) and (135.23,107.02) .. (135.23,110) .. controls (135.23,112.98) and (132.89,115.4) .. (130,115.4) .. controls (127.11,115.4) and (124.77,112.98) .. (124.77,110) -- cycle ;
\draw  [fill={rgb, 255:red, 7; green, 0; blue, 0 }  ,fill opacity=1 ] (293.88,0.3) .. controls (293.88,-2.68) and (296.21,-5.1) .. (299.1,-5.1) .. controls (301.99,-5.1) and (304.33,-2.68) .. (304.33,0.3) .. controls (304.33,3.28) and (301.99,5.7) .. (299.1,5.7) .. controls (296.21,5.7) and (293.88,3.28) .. (293.88,0.3) -- cycle ;
\draw  [fill={rgb, 255:red, 7; green, 0; blue, 0 }  ,fill opacity=1 ] (244.77,110) .. controls (244.77,107.02) and (247.11,104.6) .. (250,104.6) .. controls (252.89,104.6) and (255.23,107.02) .. (255.23,110) .. controls (255.23,112.98) and (252.89,115.4) .. (250,115.4) .. controls (247.11,115.4) and (244.77,112.98) .. (244.77,110) -- cycle ;
\draw  [fill={rgb, 255:red, 7; green, 0; blue, 0 }  ,fill opacity=1 ] (344.77,111.5) .. controls (344.77,108.52) and (347.11,106.1) .. (350,106.1) .. controls (352.89,106.1) and (355.23,108.52) .. (355.23,111.5) .. controls (355.23,114.48) and (352.89,116.9) .. (350,116.9) .. controls (347.11,116.9) and (344.77,114.48) .. (344.77,111.5) -- cycle ;
\draw  [fill={rgb, 255:red, 7; green, 0; blue, 0 }  ,fill opacity=1 ] (444.77,110) .. controls (444.77,107.02) and (447.11,104.6) .. (450,104.6) .. controls (452.89,104.6) and (455.23,107.02) .. (455.23,110) .. controls (455.23,112.98) and (452.89,115.4) .. (450,115.4) .. controls (447.11,115.4) and (444.77,112.98) .. (444.77,110) -- cycle ;
\draw  [fill={rgb, 255:red, 7; green, 0; blue, 0 }  ,fill opacity=1 ] (94.77,192) .. controls (94.77,189.02) and (97.11,186.6) .. (100,186.6) .. controls (102.89,186.6) and (105.23,189.02) .. (105.23,192) .. controls (105.23,194.98) and (102.89,197.4) .. (100,197.4) .. controls (97.11,197.4) and (94.77,194.98) .. (94.77,192) -- cycle ;
\draw  [fill={rgb, 255:red, 7; green, 0; blue, 0 }  ,fill opacity=1 ] (74.77,260) .. controls (74.77,257.02) and (77.11,254.6) .. (80,254.6) .. controls (82.89,254.6) and (85.23,257.02) .. (85.23,260) .. controls (85.23,262.98) and (82.89,265.4) .. (80,265.4) .. controls (77.11,265.4) and (74.77,262.98) .. (74.77,260) -- cycle ;
\draw  [fill={rgb, 255:red, 7; green, 0; blue, 0 }  ,fill opacity=1 ] (114.77,260) .. controls (114.77,257.02) and (117.11,254.6) .. (120,254.6) .. controls (122.89,254.6) and (125.23,257.02) .. (125.23,260) .. controls (125.23,262.98) and (122.89,265.4) .. (120,265.4) .. controls (117.11,265.4) and (114.77,262.98) .. (114.77,260) -- cycle ;
\draw  [fill={rgb, 255:red, 7; green, 0; blue, 0 }  ,fill opacity=1 ] (154.77,193) .. controls (154.77,190.02) and (157.11,187.6) .. (160,187.6) .. controls (162.89,187.6) and (165.23,190.02) .. (165.23,193) .. controls (165.23,195.98) and (162.89,198.4) .. (160,198.4) .. controls (157.11,198.4) and (154.77,195.98) .. (154.77,193) -- cycle ;
\draw  [fill={rgb, 255:red, 7; green, 0; blue, 0 }  ,fill opacity=1 ] (344.77,191) .. controls (344.77,188.02) and (347.11,185.6) .. (350,185.6) .. controls (352.89,185.6) and (355.23,188.02) .. (355.23,191) .. controls (355.23,193.98) and (352.89,196.4) .. (350,196.4) .. controls (347.11,196.4) and (344.77,193.98) .. (344.77,191) -- cycle ;
\draw  [fill={rgb, 255:red, 7; green, 0; blue, 0 }  ,fill opacity=1 ] (284.77,190) .. controls (284.77,187.02) and (287.11,184.6) .. (290,184.6) .. controls (292.89,184.6) and (295.23,187.02) .. (295.23,190) .. controls (295.23,192.98) and (292.89,195.4) .. (290,195.4) .. controls (287.11,195.4) and (284.77,192.98) .. (284.77,190) -- cycle ;
\draw  [fill={rgb, 255:red, 7; green, 0; blue, 0 }  ,fill opacity=1 ] (404.77,192) .. controls (404.77,189.02) and (407.11,186.6) .. (410,186.6) .. controls (412.89,186.6) and (415.23,189.02) .. (415.23,192) .. controls (415.23,194.98) and (412.89,197.4) .. (410,197.4) .. controls (407.11,197.4) and (404.77,194.98) .. (404.77,192) -- cycle ;
\draw  [fill={rgb, 255:red, 7; green, 0; blue, 0 }  ,fill opacity=1 ] (284.77,260) .. controls (284.77,257.02) and (287.11,254.6) .. (290,254.6) .. controls (292.89,254.6) and (295.23,257.02) .. (295.23,260) .. controls (295.23,262.98) and (292.89,265.4) .. (290,265.4) .. controls (287.11,265.4) and (284.77,262.98) .. (284.77,260) -- cycle ;
\draw  [fill={rgb, 255:red, 7; green, 0; blue, 0 }  ,fill opacity=1 ] (374.77,260) .. controls (374.77,257.02) and (377.11,254.6) .. (380,254.6) .. controls (382.89,254.6) and (385.23,257.02) .. (385.23,260) .. controls (385.23,262.98) and (382.89,265.4) .. (380,265.4) .. controls (377.11,265.4) and (374.77,262.98) .. (374.77,260) -- cycle ;
\draw  [fill={rgb, 255:red, 7; green, 0; blue, 0 }  ,fill opacity=1 ] (404.77,260) .. controls (404.77,257.02) and (407.11,254.6) .. (410,254.6) .. controls (412.89,254.6) and (415.23,257.02) .. (415.23,260) .. controls (415.23,262.98) and (412.89,265.4) .. (410,265.4) .. controls (407.11,265.4) and (404.77,262.98) .. (404.77,260) -- cycle ;
\draw  [fill={rgb, 255:red, 7; green, 0; blue, 0 }  ,fill opacity=1 ] (434.77,260) .. controls (434.77,257.02) and (437.11,254.6) .. (440,254.6) .. controls (442.89,254.6) and (445.23,257.02) .. (445.23,260) .. controls (445.23,262.98) and (442.89,265.4) .. (440,265.4) .. controls (437.11,265.4) and (434.77,262.98) .. (434.77,260) -- cycle ;
\draw  [fill={rgb, 255:red, 7; green, 0; blue, 0 }  ,fill opacity=1 ] (444.77,190) .. controls (444.77,187.02) and (447.11,184.6) .. (450,184.6) .. controls (452.89,184.6) and (455.23,187.02) .. (455.23,190) .. controls (455.23,192.98) and (452.89,195.4) .. (450,195.4) .. controls (447.11,195.4) and (444.77,192.98) .. (444.77,190) -- cycle ;
\draw (311,-5.3) node [anchor=north west][inner sep=0.75pt]    {$\pi $};
\draw (461,102.4) node [anchor=north west][inner sep=0.75pt]    {$\hat{\pi }$};
\draw (416,182.4) node [anchor=north west][inner sep=0.75pt]    {$\hat{\pi }$};
\draw (275,-5.6) node [anchor=north west][inner sep=0.75pt]    {$o$};
\draw (356,102.4) node [anchor=north west][inner sep=0.75pt]    {$\hat{\pi }$};
\draw (146,100.4) node [anchor=north west][inner sep=0.75pt]    {$\hat{\pi }$};
\draw (301,182.4) node [anchor=north west][inner sep=0.75pt]    {$\hat{\pi }$};
\draw (111,182.4) node [anchor=north west][inner sep=0.75pt]    {$\hat{\pi }$};
\draw (171,182.4) node [anchor=north west][inner sep=0.75pt]    {$\hat{\pi }$};
\draw (361,182.4) node [anchor=north west][inner sep=0.75pt]    {$\hat{\pi }$};
\draw (256,102.4) node [anchor=north west][inner sep=0.75pt]    {$\hat{\pi }$};
\draw (456,182.4) node [anchor=north west][inner sep=0.75pt]    {$\hat{\pi }$};
\draw (86,250.4) node [anchor=north west][inner sep=0.75pt]    {$\hat{\pi }$};
\draw (126,250.4) node [anchor=north west][inner sep=0.75pt]    {$\hat{\pi }$};
\draw (296,250.4) node [anchor=north west][inner sep=0.75pt]    {$\hat{\pi }$};
\draw (386,250.4) node [anchor=north west][inner sep=0.75pt]    {$\hat{\pi }$};
\draw (416,250.4) node [anchor=north west][inner sep=0.75pt]    {$\hat{\pi }$};
\draw (446,250.4) node [anchor=north west][inner sep=0.75pt]    {$\hat{\pi }$};
\end{tikzpicture}
\caption{A $3-$Neighbourhood of a vertex-rooted UBGW tree with the law of the number of children drawn on every vertex.} \label{fig:VUBGW}
\end{figure}

\bigskip

The most classical examples of random graphs converging to UBGW trees we consider are sparse Erdős–Rényi  and configuration models:
\begin{itemize}
    \item \textbf{Sparse Erdős–Rényi:} Introduced in the celebrated paper of Erdős and Rényi \cite{ErdosRenyi}, for $c>0$ and $N\geq 1$, the random graph $\mathcal G (N,\frac{c}{N})$ is defined on the vertex set $\{1,\ldots, N \}$ with independent edges between vertices with probability $\frac{c}{N}$. Once uniformly rooted, these graphs converge locally when $N$ goes to $\infty$, to a UBGW tree with reproduction law Poisson with parameter $c$.
    \item \textbf{Configuration model:} This model was introduced by Bollob\'as in 1980 \cite{B1980} and can be defined as follows.
Let $N\geq 1$ be an integer and let $d_1, \ldots, d_N \in \mathbb{Z}_+$ be such that $d_1 + \cdots +d_N$ is even. We interpret $d_i$ as a number of half-edges attached to vertex $i$. Then, the configuration model associated to the sequence $(d_i)_{1 \leq i \leq N}$ is the random multi-graph with vertex set $\{1, \ldots ,N\}$ obtained by a uniform matching of these half-edges. If $d_1 + \cdots + d_N$ is odd, we change $d_N$ into $d_N+1$ and do the same construction. Now, let $\mathbf{d}^{(N)}$ be a sequence of random variables defined on the same probability space $(\Omega,\mathcal{F},\mathbb{P})$ such that for every $N\geq 1$, $\mathbf{d}^ {(N)} = (d_1^{(N)}, \ldots, d_N^{(N)}) \in \mathbb{Z}_+^N$. Furthermore, suppose that there exists $\pi$ a probability measure on $\mathbb{Z}_+$ with finite first moment such that
\[  \forall k \geq 0, \quad \frac{1}{N} \sum\limits_{ j=1}^N \mathbbm{1}_{ d_j^{(N)} = k}  \underset{ N \rightarrow +\infty}{ \longrightarrow} \pi( \{k\}).   \]
The sequence of random configuration graphs associated to the $\mathbf{d}^{(N)}$ has asymptotically a positive probability to be simple. In addition, this sequence of random graphs, when uniformly rooted, converges locally in law to the UBGW random tree with offspring distribution $\pi$, see \cite{vanderHofstadBook2} for more details.
\end{itemize}

\subsection{Matchings, optimality and unimodularity}
We start with the definition of matchings on a graph.
\begin{definition}
 For any weighted graph $G=((V,E),w)$, a matching $M=(V,E')$ on $G$ is a subgraph of $G$ such that $E' \subset E$ and  every vertex of $V$ belongs to at most one edge of $E'$. A matched graph is a pair $(G,M)$, where $M$ is a matching on $G$.
\end{definition}
We can extend the notion of unimodularity for random graphs to random matched graphs by transforming a matching into a canonical decoration. If $(G,M)=((V,E),(V,E'))$  is a matched graph then we can define the associated decoration $\mathbbm{1}_{M}:E\to\{0,1\}$ with  $\mathbbm{1}_{M}(u,v)=1$ if and only if $(u,v) \in E'$.
\begin{definition}
Let $(\mathbb{G},\mathbb M,\mathbf{o})$ be a random matched rooted graph. 
    We say that $(\mathbb{G},\mathbb{M},\mathbf{o})$ is unimodular iff $(\mathbb{G},\mathbf{o},\mathbbm{1}_{\mathbb{M}})$ is unimodular.
\end{definition}
The correspondence between $(\mathbb{G},\mathbb{M},\mathbf{o})$ and $(\mathbb{G},\mathbf{o},\mathbbm{1}_{\mathbb{M}})$ gives a representation of the space of unimodular laws on decorated matched graphs as a subspace of unimodular laws on decorated graphs $\mathcal{L}_{U,\mathcal{M}}(\mathcal{G}^{\star}) \subseteq \mathcal{L}_{U}(\mathcal{G}^{\star})$.
\begin{definition}
The subspace $\mathcal{L}_{U,\mathcal{M}}(\mathcal{G}^{\star}) \subseteq \mathcal{L}_{U}(\mathcal{G}^{\star})$ is defined as the set of 
elements $\mu$ of $\mathcal{L}_{U}(\mathcal{G}^{\star})$ such that for $(G,o,w,(f_i)_{i \in \{1,..,I\}}) \sim \mu$, we have that $f_I(u,v)$ maps $E$ to $\{0,1\}$ and the subset $\{(u,v), f_I(u,v)=1 \}$ almost surely induces a matching on $G$ in the sense that
\begin{align*}
    \mathbb{P}\left( \exists u,u',v \in V, u \neq u' \text{ and } f(v,u)=f(v,u')=1\right)=0.
\end{align*}
Similarly we can define the edge rooted version $\mathcal{L}_{U,\mathcal{M}}(\hat{\mathcal{G}}^{\star})$.
\end{definition}

We recall that the central object of interest is optimal matching, which corresponds to the classical notion of maximal weight matching in finite graphs.
When $G$ is infinite, this optimality is ill-defined. However, in the case of a unimodular random weighted graph, since the root edge is informally a typical edge, we can define optimality via its expected weight when it belongs to the matching. This leads to the following definition:
\begin{definition}
    Let $(\mathbb{G}, \mathbf o, \mathbb M )$ be a unimodular random matched rooted graph, we define the performance of $(\mathbb{G},\mathbf o,\mathbb{M})$ as:
    \begin{align*}
     \perfE(\mathbb{G},\mathbf o, \mathbb{M}) &:= \bigg(\mathbb{P}((\mathbf{o}_-,\mathbf{o_+})\in \mathbb{M}), \mathbb{E}\left[ w((\mathbf{o}_-,\mathbf{o}_+))\mathbbm{1}_{(\mathbf{o}_-,\mathbf{o}_+) \in \mathbb{M}} \right]\bigg) \text{ in the edge-rooted setting,}  \\
    \perfV(\mathbb{G},\mathbf o,\mathbb{M}) &:=\bigg(\mathbb{P}( \mathbf{o} \text{ is matched by } \mathbb{M}) ,\mathbb{E}\left[ \sum_{v \sim o}w(\mathbf{o},v)\mathbbm{1}_{(\mathbf{o},v) \in \mathbb{M}} \right] \bigg)\text{in the vertex-rooted setting.} 
    \end{align*}
\end{definition}

When the context is clear, we will shorten the notation to $\perf (\mathbb{M})$. 
By extension, since those quantities only depend on the law of $(\mathbb{G},\mathbf o,\mathbb{M})$, we will freely use the same notation $\perfE(\mu_{e}):=\perfE(\mathbb{G},\mathbf{o},\mathbb{M})$ for $(\mathbb{G},\mathbf{o},\mathbbm{1}_\mathbb{M}) \sim \mu_{e}$ and $\perfV(\mu_{v}):=\perfV(\mathbb{G},\mathbf{o},\mathbb{M})$ for $(\mathbb{G},\mathbf{o},\mathbbm{1}_\mathbb{M}) \sim \mu_{v}$ where $\mu_{e}\in \mathcal{L}_{U,\mathcal{M}}(\hat{\mathcal{G}}^{\star})$ and $\mu_{v} \in \mathcal{L}_{U,\mathcal{M}}(\mathcal{G}^{\star})$.

\bigskip

Let $(\mathbb G, \mathbf o)$ be a (undecorated) unimodular vertex-rooted graph. We say that $(\mathbb{G}',\mathbf{o}',\mathbb{M})$ is optimal if it is a unimodular matched vertex-rooted graph such that $\perfV((\mathbb{G}',\mathbf{o}'),\mathbb{M})$ is lexicographically maximal among all unimodular rooted matched random graphs $(\mathbb{G}',\mathbf{o}',\mathbb{M}')$ such that $(\mathbb{G}',\mathbf{o}')$ has the same law as $(\mathbb{G},\mathbf{o})$.
In the edge-rooted setting, optimality is defined similarly.
The next proposition shows that the operator $R$, introduced in Definition~\ref{def:R}, preserves optimality.

\begin{prop}\label{prop:chgmtpdv}
    Let $\mu_{v} \in \mathcal{L}_{U,\mathcal{M}}(\mathcal{G}^{\star})$, assume $m = \int_{\mathcal G^{\star}} \deg (o) \, \mathrm{d} \mu_v(G,o) < \infty$, then:
    \[  \perfV(\mu_{v}) =m\perfE(R(\mu_{v})).  \]
\end{prop}
The proof of this proposition is nearly identical with the proof of \cite[Proposition~2.12]{enriquez2024optimalunimodularmatching} with bidimensional weights $(1,w)$ instead of unidimensional weights $w$ so we omit it here.
\begin{remark}
If the graph is finite and the root is chosen uniformly either among the vertices or among the directed edges, then $\perfV$ is simply the average contribution per vertex, and $\perfE$ is the average contribution per directed edge. 
It is then clear that the two quantities are proportional. The proposition shows that it generalises to unimodular matched graphs. 
\end{remark}
We state a useful property of unimodular graphs: events that have probability zero (resp. almost sure) at the root have probability zero (resp. almost sure) everywhere.
\begin{prop}[Lemma 2.3 in \cite{aldous2018processes}]
    Let $(\mathbf{G},\mathbf{o})$ be a unimodular graph, let $H>0$ and $f$ be a non-negative H-local function on $\mathcal{G}^{\star}$.
    Assume that 
    \[ \mathbb{E}\left[ f(\mathbf{G},\mathbf{o})\right]=0 .\]
    Then almost surely, for all $v \in \mathbf{V}$ 
    \[ f(\mathbf{G},v)=0. \]
\end{prop}

\section{Heuristics}\label{sec:heuristic}
In this section, we give more details about the heuristics behind our approach. We will look at a UBGW tree $(\mathbb{T},\overset{\rightarrow}{o})$ some reproduction law $\pi$ satisfying that
\begin{enumerate}
    \item The map $t \mapsto \hat{\phi}(1-\hat{\phi}(1-t))$ has finite amount of fixed points.
    \item The tree has leaves, which is equivalent to $\hat{\phi}(0)>0$.
\end{enumerate}
We will then try to find the "maximum size, maximum weight matching", in this order, on the tree.
The main idea is to see what happens when we consider the maximum weight matching, with no size constraints, when edge weights are of the form $w^\eps(u,v)=1+\eps w(u,v)$ for $(u,v) \in E$ for i.i.d.~continuous variables $w(u,v)$ and $\varepsilon \to 0$. The study of maximum weight matchings was the topic of our previous article \cite{enriquez2024optimalunimodularmatching} and we will first describe this setting in the next section as it is instrumental to our approach.

\subsection{Unimodular optimal weight matchings}
\label{subsec:premierpapier}
This section is dedicated to the intuition behind message-passing algorithms for maximum weight matchings with no size constraint.
We start by discussing the simpler setting of finite trees. We are looking for a dynamic program that builds the maximum matching.

\begin{figure}[t!]
\centering
\tikzset{every picture/.style={line width=0.75pt}} 

\begin{tikzpicture}[x=0.75pt,y=0.75pt,yscale=-1,xscale=1]

\draw    (170,161) -- (290,161) ;
\draw    (90,111) -- (170,161) ;
\draw    (90,161) -- (170,161) ;
\draw    (170,161) -- (90,211) ;
\draw    (290,161) -- (370,161) ;
\draw    (450,131) -- (370,161) ;
\draw    (450,191) -- (370,161) ;
\draw  [fill={rgb, 255:red, 0; green, 0; blue, 0 }  ,fill opacity=1 ] (165.65,161) .. controls (165.65,158.6) and (167.6,156.65) .. (170,156.65) .. controls (172.4,156.65) and (174.35,158.6) .. (174.35,161) .. controls (174.35,163.4) and (172.4,165.35) .. (170,165.35) .. controls (167.6,165.35) and (165.65,163.4) .. (165.65,161) -- cycle ;
\draw    (20,91) -- (90,111) ;
\draw    (20,131) -- (90,111) ;
\draw    (20,111) -- (90,111) ;
\draw    (20,151) -- (90,161) ;
\draw    (20,171) -- (90,161) ;
\draw  [fill={rgb, 255:red, 0; green, 0; blue, 0 }  ,fill opacity=1 ] (285.65,161) .. controls (285.65,158.6) and (287.6,156.65) .. (290,156.65) .. controls (292.4,156.65) and (294.35,158.6) .. (294.35,161) .. controls (294.35,163.4) and (292.4,165.35) .. (290,165.35) .. controls (287.6,165.35) and (285.65,163.4) .. (285.65,161) -- cycle ;
\draw  [fill={rgb, 255:red, 0; green, 0; blue, 0 }  ,fill opacity=1 ] (85.65,211) .. controls (85.65,208.6) and (87.6,206.65) .. (90,206.65) .. controls (92.4,206.65) and (94.35,208.6) .. (94.35,211) .. controls (94.35,213.4) and (92.4,215.35) .. (90,215.35) .. controls (87.6,215.35) and (85.65,213.4) .. (85.65,211) -- cycle ;
\draw  [fill={rgb, 255:red, 0; green, 0; blue, 0 }  ,fill opacity=1 ] (85.65,161) .. controls (85.65,158.6) and (87.6,156.65) .. (90,156.65) .. controls (92.4,156.65) and (94.35,158.6) .. (94.35,161) .. controls (94.35,163.4) and (92.4,165.35) .. (90,165.35) .. controls (87.6,165.35) and (85.65,163.4) .. (85.65,161) -- cycle ;
\draw  [fill={rgb, 255:red, 0; green, 0; blue, 0 }  ,fill opacity=1 ] (85.65,111) .. controls (85.65,108.6) and (87.6,106.65) .. (90,106.65) .. controls (92.4,106.65) and (94.35,108.6) .. (94.35,111) .. controls (94.35,113.4) and (92.4,115.35) .. (90,115.35) .. controls (87.6,115.35) and (85.65,113.4) .. (85.65,111) -- cycle ;
\draw  [fill={rgb, 255:red, 0; green, 0; blue, 0 }  ,fill opacity=1 ] (365.65,161) .. controls (365.65,158.6) and (367.6,156.65) .. (370,156.65) .. controls (372.4,156.65) and (374.35,158.6) .. (374.35,161) .. controls (374.35,163.4) and (372.4,165.35) .. (370,165.35) .. controls (367.6,165.35) and (365.65,163.4) .. (365.65,161) -- cycle ;
\draw  [draw opacity=0] (31.64,259.84) .. controls (117.13,256.93) and (185,213.34) .. (185,160) .. controls (185,106.32) and (116.26,62.51) .. (29.99,60.1) -- (22.5,160) -- cycle ; \draw  [color={rgb, 255:red, 208; green, 2; blue, 27 }  ,draw opacity=1 ] (31.64,259.84) .. controls (117.13,256.93) and (185,213.34) .. (185,160) .. controls (185,106.32) and (116.26,62.51) .. (29.99,60.1) ;  
\draw  [draw opacity=0] (30.04,248.19) .. controls (87.11,239.84) and (130.05,203.49) .. (130.03,159.93) .. controls (130.02,116.36) and (87.02,80.03) .. (29.92,71.74) -- (5.03,159.98) -- cycle ; \draw  [color={rgb, 255:red, 126; green, 211; blue, 33 }  ,draw opacity=1 ] (30.04,248.19) .. controls (87.11,239.84) and (130.05,203.49) .. (130.03,159.93) .. controls (130.02,116.36) and (87.02,80.03) .. (29.92,71.74) ;  
\draw  [draw opacity=0] (438.13,249.68) .. controls (433.8,249.89) and (429.42,250) .. (425,250) .. controls (339.4,250) and (270,209.71) .. (270,160) .. controls (270,110.29) and (339.4,70) .. (425,70) .. controls (427.64,70) and (430.26,70.04) .. (432.86,70.11) -- (425,160) -- cycle ; \draw  [color={rgb, 255:red, 144; green, 19; blue, 254 }  ,draw opacity=1 ] (438.13,249.68) .. controls (433.8,249.89) and (429.42,250) .. (425,250) .. controls (339.4,250) and (270,209.71) .. (270,160) .. controls (270,110.29) and (339.4,70) .. (425,70) .. controls (427.64,70) and (430.26,70.04) .. (432.86,70.11) ;  
\draw  [draw opacity=0] (434.75,229.82) .. controls (375.09,227.25) and (328,196.97) .. (328,160) .. controls (328,122.96) and (375.27,92.63) .. (435.1,90.16) -- (443,160) -- cycle ; \draw  [color={rgb, 255:red, 74; green, 144; blue, 226 }  ,draw opacity=1 ] (434.75,229.82) .. controls (375.09,227.25) and (328,196.97) .. (328,160) .. controls (328,122.96) and (375.27,92.63) .. (435.1,90.16) ;  

\draw (171,133.4) node [anchor=north west][inner sep=0.75pt]    {$u$};
\draw (291,133.4) node [anchor=north west][inner sep=0.75pt]    {$v$};
\draw (209,132.4) node [anchor=north west][inner sep=0.75pt]    {$w( u,v)$};
\draw (161,282.4) node [anchor=north west][inner sep=0.75pt]    {$ \begin{array}{l}
\mathbf{Z}( u,v) =
OPT(\textcolor[rgb]{0.56,0.07,1}{T_{(u,v)}}) -OPT(\textcolor[rgb]{0.29,0.56,0.89}{T_{(u,v)} \setminus \{v\}}) .
\end{array}$};
\draw (311,52.4) node [anchor=north west][inner sep=0.75pt]  [color={rgb, 255:red, 144; green, 19; blue, 254 }  ,opacity=1 ]  {$T_{(u,v)}$};
\draw (13,192.4) node [anchor=north west][inner sep=0.75pt]  [color={rgb, 255:red, 126; green, 211; blue, 33 }  ,opacity=1 ]  {$T_{(v,u)} \setminus \{u\}$};
\draw (129,52.4) node [anchor=north west][inner sep=0.75pt]  [color={rgb, 255:red, 208; green, 2; blue, 27 }  ,opacity=1 ]  {$T_{(v,u)}$};
\draw (371,192.4) node [anchor=north west][inner sep=0.75pt]  [color={rgb, 255:red, 74; green, 144; blue, 226 }  ,opacity=1 ]  {$\textcolor[rgb]{0.29,0.56,0.89}{T_{(u,v)} \setminus \{v\}}$};
\draw (161,332.4) node [anchor=north west][inner sep=0.75pt]    {$ \begin{array}{l}
\mathbf{Z}( v,u) = OPT(\textcolor[rgb]{0.82,0.01,0.11}{T_{(v,u)}}) -OPT(\textcolor[rgb]{0.49,0.83,0.13}{T_{(v,u)} \setminus \{u\}}) .
\end{array}$};
\draw (1,281.4) node [anchor=north west][inner sep=0.75pt]    {$ \begin{array}{l}
OPT( G) =\text{weight of} \\
\text{a maximum matching}\\
\text{on}\ G.
\end{array}$};

\end{tikzpicture}
\caption{Definitions of $\bf{Z}$.}\label{fig:Zdefinitions}
\end{figure}

Fix a finite weighted rooted deterministic tree $T$ with a unique optimal matching $\mathbb{M}_{\mathrm{opt}}$. Let $\{u,v\}$ be an edge of $T$, we denote by $T_{(v,u)}$ and $T_{(u,v)}$ the two connected components of $T \setminus \{u,v\}$ containing respectively $u$ and $v$. Let us start by simple but key observations illustrated in Figure~\ref{fig:Zdefinitions}:
\begin{itemize}
    \item The maximal weight of matchings of $T$ that exclude the edge $\{u,v \}$ is merely the sum of the maximal weights of matchings of $T_{(v,u)}$ and $T_{(u,v)}$, denoted $OPT(T_{(v,u)})$ and $OPT(T_{(u,v)})$.
    \item The maximal weight of matchings that include the edge $\{u,v\}$ is the sum of the weight of $\{u,v\}$ and of the maximal weights of matchings of $T_{(v,u)} \setminus \{ u \} $ and $T_{(u,v)} \setminus \{ v \}$. Note that both $T_{(v,u)} \setminus \{ u \} $ and $T_{(u,v)} \setminus \{ v \}$ consist of collections of disjoint subtrees of $T$ issued from the children of $u$ and $v$. We denote by $OPT(T_{(v,u)} \setminus \{ u\})$ and $OPT(T_{(u,v)} \setminus \{ v \})$ the relevant maximal weights.
\end{itemize}
From this discussion, one can see that the edge $\{ u,v \}$ is in the optimal matching of $T$ iff
\[
w(u,v) > OPT(T_{(v,u)}) + OPT(T_{(u,v)}) - \Big(OPT(T_{(v,u)} \setminus \{u\}) + OPT(T_{(u,v)} \setminus \{v\}) \Big).
\]
It will be instrumental to isolate quantities depending only on $T_{(v,u)}$ and $T_{(u,v)}$ in the previous display. This leads us to introduce the following quantities
\begin{align*}
\mathbf{Z}(u,v)&= OPT( T_{(u,v)})-OPT(T_{(u,v)} \setminus \{v\}), \\
\mathbf{Z}(v,u)&=OPT(T_{(v,u)} )- OPT( T_{(v,u)} \setminus \{u\}), 
\end{align*}
and the criterion of $(u,v) \in \mathbb{M}_{\mathrm{opt}}$ is simply 
\begin{equation}\label{eq:defmatching}
 w(u,v)> \mathbf{Z}(u,v)+\mathbf{Z}(v,u) .
\end{equation}
Note that the variable $\mathbf{Z}(u,v)$ has a neat interpretation  in terms of the matchings of $T_{(u,v)}$. Indeed, it is the marginal gain between allowing $v$ to be matched or not.

\bigskip

The variables $\mathbf{Z}$ have the nice property of satisfying a recursive equation. We describe this recursion for $\mathbf{Z}(u,v)$ and $T_{(u,v)}$. See Figure~\ref{fig:Zrecursion} for an illustration.

Listing $v_1,...v_k$ the children of $v$ in $T_{(u,v)}$, assume that the maximum matching of $T_{(u,v)}$ matches $v$ with $v_i$. 
In that situation, the maximum matching of $T_{(u,v)}$ and the maximum matching of $T_{(u,v)} \setminus \{v\}$ coincide on the subtrees $T_{(v,v_j)}$ for $j \neq i$.
On $T_{(v,v_i)}$, this maximum matching has matched $v_i$ with $v$, so $v_i$ is not matched to other vertices. Therefore, our maximum matching on $T_{(u,v)}$ restricted to $T_{(v,v_i)}$ is the union of $\{v,v_i\}$ with the maximum matching of
$T_{(v,v_i)} \setminus \{v_i\}$. In that case, the weight of the maximal matching is given by:
\[
w(v,v_i) + OPT(T_{(v,v_i)} \setminus \{v_i\}) + \sum_{j \neq i} OPT(T_{(v,v_j)}).
\]

On the other hand, if a vertex $v$ is not matched in the maximum matching of $T_{(u,v)}$, then the later is also the maximum matching in $T_{(u,v)} \setminus \{v\}$. Thus, inside each sub-tree $T_{(v,v_i)}$, it coincides with the maximal matching of $T_{(v,v_i)}$.
In that case, the weight of the maximal matching is given by:
\[ \sum_{j} OPT(T_{(v,v_j)}).  \]

Putting all the different cases together, we have the identity
\[
OPT(T_{(u,v)}) = \max \left\{ \sum_{j} OPT(T_{(v,v_j)}) \, , \max_{i\in \{ 1, \ldots , k \} } \left\{ w(v,v_i) + OPT(T_{(v,v_i)} \setminus \{v_i\}) + \sum_{j \neq i} OPT(T_{(v,v_j)}) \right\} \right\}.
\]
Recalling the definition of $Z(u,v)$, we get:
\begin{align*}
\mathbf{Z}(u,v)&= 
OPT(T_{(u,v)})
-\sum_{j} OPT(T_{(v,v_j)}) \\
&= \max\left\{0, \max_{i\in \{ 1, \ldots , k \} } \left\{w(v,v_i)-(OPT(T_{(v,v_i)})-OPT(T_{(v,v_i)} \setminus \{v_i\}) \right\} \right\} \\
&= \max\left\{0,\max_{i\in \{ 1, \ldots , k \} } w(v,v_i)-\mathbf{Z}(v,v_i)\right\}.
\end{align*}

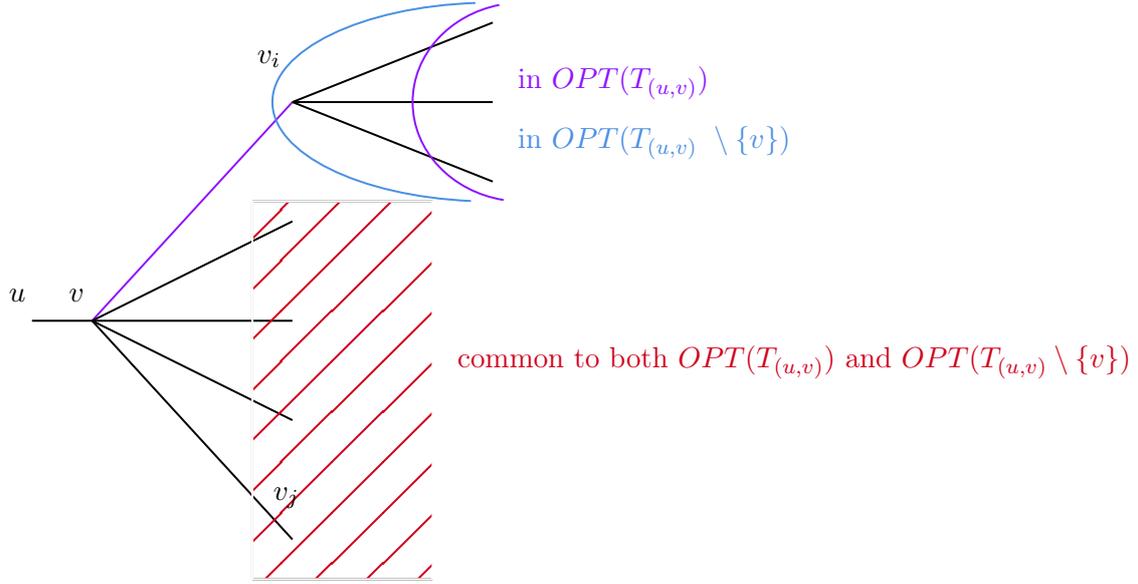
\begin{figure}[t!]
\tikzset{
pattern size/.store in=\mcSize, 
pattern size = 5pt,
pattern thickness/.store in=\mcThickness, 
pattern thickness = 0.3pt,
pattern radius/.store in=\mcRadius, 
pattern radius = 1pt}
\makeatletter
\pgfutil@ifundefined{pgf@pattern@name@_4nnhtjt6r}{
\pgfdeclarepatternformonly[\mcThickness,\mcSize]{_4nnhtjt6r}
{\pgfqpoint{0pt}{0pt}}
{\pgfpoint{\mcSize+\mcThickness}{\mcSize+\mcThickness}}
{\pgfpoint{\mcSize}{\mcSize}}
{
\pgfsetcolor{\tikz@pattern@color}
\pgfsetlinewidth{\mcThickness}
\pgfpathmoveto{\pgfqpoint{0pt}{0pt}}
\pgfpathlineto{\pgfpoint{\mcSize+\mcThickness}{\mcSize+\mcThickness}}
\pgfusepath{stroke}
}}
\makeatother

 
\tikzset{
pattern size/.store in=\mcSize, 
pattern size = 5pt,
pattern thickness/.store in=\mcThickness, 
pattern thickness = 0.3pt,
pattern radius/.store in=\mcRadius, 
pattern radius = 1pt}
\makeatletter
\pgfutil@ifundefined{pgf@pattern@name@_ik3oy0ks0}{
\pgfdeclarepatternformonly[\mcThickness,\mcSize]{_ik3oy0ks0}
{\pgfqpoint{0pt}{0pt}}
{\pgfpoint{\mcSize+\mcThickness}{\mcSize+\mcThickness}}
{\pgfpoint{\mcSize}{\mcSize}}
{
\pgfsetcolor{\tikz@pattern@color}
\pgfsetlinewidth{\mcThickness}
\pgfpathmoveto{\pgfqpoint{0pt}{0pt}}
\pgfpathlineto{\pgfpoint{\mcSize+\mcThickness}{\mcSize+\mcThickness}}
\pgfusepath{stroke}
}}
\makeatother
\tikzset{every picture/.style={line width=0.75pt}} 

\centering
\begin{tikzpicture}[x=0.75pt,y=0.75pt,yscale=-1,xscale=1]

\draw [color={rgb, 255:red, 144; green, 19; blue, 254 }  ,draw opacity=1 ]   (30,160) -- (130,50) ;

\draw (0,160) -- (30,160);
\draw    (30,160) -- (130,110) ;
\draw    (30,160) -- (130,160) ;
\draw    (30,160) -- (130,210) ;
\draw    (30,160) -- (130,270) ;
\draw    (130,50) -- (230,10) ;
\draw    (130,50) -- (230,50) ;
\draw    (130,50) -- (230,90) ;
\draw [pattern=_4nnhtjt6r,pattern size=6pt,pattern thickness=0.75pt,pattern radius=0pt, pattern color={rgb, 255:red, 0; green, 0; blue, 0}]   (110,100) -- (110,290) ;
\draw    (110,100) -- (200,100) ;
\draw    (110,290) -- (200,290) ;
\draw  [color={rgb, 255:red, 255; green, 255; blue, 255 }  ,draw opacity=1 ][pattern=_ik3oy0ks0,pattern size=19.049999999999997pt,pattern thickness=0.75pt,pattern radius=0pt, pattern color={rgb, 255:red, 208; green, 2; blue, 27}] (110,100) -- (200,100) -- (200,290) -- (110,290) -- cycle ;
\draw  [draw opacity=0] (235.37,99.24) .. controls (209.58,95.1) and (190,74.63) .. (190,50) .. controls (190,26.04) and (208.54,6.02) .. (233.29,1.14) -- (245,50) -- cycle ; \draw  [color={rgb, 255:red, 144; green, 19; blue, 254 }  ,draw opacity=1 ] (235.37,99.24) .. controls (209.58,95.1) and (190,74.63) .. (190,50) .. controls (190,26.04) and (208.54,6.02) .. (233.29,1.14) ;  
\draw  [draw opacity=0] (219.11,99.76) .. controls (163.47,97.27) and (120,75.94) .. (120,50) .. controls (120,23.71) and (164.63,2.16) .. (221.34,0.15) -- (230,50) -- cycle ; \draw  [color={rgb, 255:red, 74; green, 144; blue, 226 }  ,draw opacity=1 ] (219.11,99.76) .. controls (163.47,97.27) and (120,75.94) .. (120,50) .. controls (120,23.71) and (164.63,2.16) .. (221.34,0.15) ;  

\draw (111,22.4) node [anchor=north west][inner sep=0.75pt]    {$v_{i}$};
\draw (17,142.4) node [anchor=north west][inner sep=0.75pt]    {$v$};
\draw (-13,142.4) node [anchor=north west][inner sep=0.75pt]    {$u$};
\draw (119,242.4) node [anchor=north west][inner sep=0.75pt]    {$v_{j}$};
\draw (211,171) node [anchor=north west][inner sep=0.75pt]   [align=left] {$\textcolor[rgb]{0.82,0.01,0.11}{{\text{common to both } OPT(T_{(u,v)}) \text{ and }  OPT(T_{(u,v)}\setminus \{ v \})}}$};
\draw (241,30.4) node [anchor=north west][inner sep=0.75pt]    {$\textcolor[rgb]{0.56,0.07,1}{\text{in } OPT( T_{(u,v)})}$};
\draw (241,60.4) node [anchor=north west][inner sep=0.75pt]    {$\textcolor[rgb]{0.29,0.56,0.89}{\text{in} \ OPT( T_{(u,v)} \ \setminus \{v\})}$};

\end{tikzpicture}
\caption{Illustration of the deduction of the recursive equation assuming $v_i$ is matched to $v$.}\label{fig:Zrecursion}
\end{figure}

Note that, since $T$ is a finite tree, it is possible to calculate $\mathbf{Z}(u,v)$ for all vertices $u \sim v \in T$ by starting when $v$ is a leaf, in which case $\mathbf{Z}(u,v) = 0$. By construction, our decision rule \eqref{eq:decisionrule} constructs the optimal matching on $T$ from the values of $\mathbf{Z}$.

In the context of an infinite tree, it may not be possible to build the messages measurably, so we look for a joint distribution $(\mathbb{T},\overset{\rightarrow}{o},\mathbf{Z})$ instead. This is done in our previous work \cite{enriquez2024optimalunimodularmatching} and the first step is to look for a stationary distribution for the law of $\mathbf{Z}$:

We look for a distribution such that for all sequence $(\mathbf{Z}_k)_{k \geq 0}
    $ i.i.d of this unknown distribution and independent of a variable $N \sim \hat{\pi}$ and of the sequence $(w_k)_{k \in \mathbb{N}}$ i.i.d of law $\omega$, the following equality in law holds:
    
    \begin{equation}\label{eq:Zlawfirstpaper}
    \mathbf{Z}_0\overset{\mathrm{(law)}}{=}\max(0,\max_{1 \leq i \leq N} (w_i-\mathbf{Z}_i)).
    \end{equation}
Theorem~3 of \cite{enriquez2024optimalunimodularmatching} states that the distributional equation~\ref{eq:Zlawfirstpaper} has a unique solution $\mu$. Given this distribution $\mu$, Proposition~3.2 of \cite{enriquez2024optimalunimodularmatching} builds a unimodular joint distribution $\left(\mathbb{T},\overset{\rightarrow}{o},w,( \mathbf{Z}(u,v))_{(u,v) \in \overset{\rightarrow}{E}} \right)$ such that :
\begin{itemize}
\item $\mathbf{Z}(\overset{\rightarrow}{o})$ has distribution $\mu$,

\item Almost surely, for any $(u,v) \in \mathbb T$ one has
\begin{equation} \label{eq:Zweightrec}
    \mathbf{Z}(u,v) = \max \left\{ 0,\max_{v' \sim v}  w(v,v')-\mathbf{Z}(v,v')\right\}.
\end{equation}
\item The rule
\begin{equation} \label{eq:ruleM}
(u,v) \in \mathbb{M}_{\mu}\Leftrightarrow w(u,v)> \mathbf{Z}(u,v)+\mathbf{Z}(v,u).
\end{equation}
defines an unimodular matching on $\mathbb T$.
\end{itemize}
The main result from \cite[Theorem~1 and Theorem~3]{enriquez2024optimalunimodularmatching} is that $(\mathbb{T},\overset{\rightarrow}{o},w,\mathbb{M}_{\mu})$ is the unique (in distribution) optimal matching on $\mathbb{T}$:
\begin{theorem} \label{th:firstpaper}
Let $(\mathbb{T},\overset{\rightarrow}{o},w)$ be a unimodular Bienaymé-Galton-Watson Tree with reproduction law $\pi$ and i.i.d weights of law $\omega$.  
Assume that $\omega$ is integrable, non atomic and that $\pi$ is integrable.
Then the law of $(\mathbb{T},\overset{\rightarrow}{o},w,\mathbb{M}_{\mu})$ is the unique unimodular distribution on matched graphs maximising $\mathbb{E}[w(o{})\mathbbm{1}_{o \in \mathbb{M}}]$.
\end{theorem}

\subsection{Renormalising the message variables}

We return to maximum size, maximum weight matchings. We assign weights $w^\varepsilon := 1 + \varepsilon \, w$ to edges, where the variables $w$ are i.i.d.~and have a continuous distribution to use the approach described in the previous section. Indeed, Theorem~\ref{th:firstpaper} states that the unique (in law) unimodular matching $\mathbb{M}_\eps$ on $(\mathbb{T},\overset{\rightarrow}{o},w_\eps)$ maximising $\mathbb{E}[w^\varepsilon(o)\mathbbm{1}_{o \in \mathbb{M}}]$ can be generated by variables $\mathbf{Z}^{(\eps)}$ verifying
\begin{equation}\label{eq:Zepsilonrecursion}
    \mathbf{Z}^{(\eps)}(u,v)= \max\left(0,\max_{\substack{ u' \sim v \\ u' \neq u}}\left( w^\eps(v,u')-\mathbf{Z}^{(\eps)}(v,u') \right) \right).
\end{equation}
The cumulative distribution function $h^{\eps}$ of the common distribution of the variables $\mathbf Z^{(\varepsilon)}$ satisfies
\begin{equation}\label{eq:hequation}
    \forall t \in \mathbb{R}, h^{\eps}(t)=\mathbbm{1}_{t \geq 0} \hat{\phi}(1-\mathbb{E}[h^{\eps} (W^\eps-t)]),
\end{equation}
where $W^\eps$ has the distribution of $1+\eps W$ with $W \sim \omega$ and $\hat \phi$ denotes the probability generating function of the size-biased reproduction law of the tree.
Finally, the inclusion rule would be:
\begin{equation}\label{eq:Zepsiloninclusionrule}
    (u,v) \in \mathbb{M}^\eps \Leftrightarrow w^\eps (u,v) > \mathbf{Z}^{(\eps)}(u,v)+\mathbf{Z}^{(\eps)}(v,u).
\end{equation}
If one naively took $\eps \rightarrow 0$ in the above equation, assuming $h^{\eps}(t) \rightarrow h^0(t)$, we would get:
\begin{equation}\label{eq:equationnaive}
    \forall t \in \mathbb{R} , h^0(t) = \mathbbm{1}_{t \geq 0 }\hat{\phi}(1-h^0(1-t)).
\end{equation}
which, when iterated inside $]0,1[$ to avoid the boundary conditions, gives that $h^0(t)$ has to solve:
\begin{equation}\label{eq:fixedpointpoisson}
    \forall t \in ]0,1[ , h^0(t)=\hat{\phi}(1-\hat{\phi}(1-h^0(t))).
\end{equation}
This implies that the values of $h^0$ at every point of continuity has to be a fixed point of $t \mapsto \hat{\phi}(1-\hat{\phi}(1-t))$. So $\mathbf{Z}^{(0)}$ must be a discrete variable in $[0,1]$.
In general if $t \mapsto \hat{\phi}(1-\hat{\phi}(1-t))$ has $p+1$ fixed points $l_0,...l_p$, we must have that $l_k=\hat{\phi}(1-l_{p-k})$ so we will set $l'_k=l_{p-k}$. We can then enumerate all the possible solutions for $h^0$ with the conditions we have found so far. Let $p'\leq \tfrac{p}{2}$ be a integer, $l_{k_0}<...<l_{k_{p'}}<l_{\lfloor \frac{p}{2} \rfloor}$ be fixed points. Consider $x_0=0<x_1<\cdots<x_{p'}\leq \frac{1}{2}$ be real numbers with $x_{p'}=\frac{1}{2}$ only if ${k_{p'}}=\frac{p}{2}$. The solutions are parametrised with $(l_{k_0},\cdots,l_{k_{p'}},x_0,...,x_{p'})$ as follows:
\begin{itemize}
\item If $x_p'\neq \frac{1}{2}$, setting $l_{k_{-1}}=0$ and $l'_{k_{-1}}=1$,
\[h^0(t)= \sum_{n=0}^{p'} \left( \mathbbm{1}_{t \geq x_n}\left(l_{k_n}-l_{k_{(n-1)}}\right) +\mathbbm{1}_{t \geq 1-x_n}\left(l'_{k_{n-1}}-l'_{k_n}\right) \right) . \] 
\item Otherwise, if $x_p' = \frac{1}{2}$,
\[h^0(t)= \left(\sum_{n=0}^{p'-1} \mathbbm{1}_{t \geq x_n}\left(l_{k_n}-l_{k_{(n-1)}}\right) +\mathbbm{1}_{t \geq 1-x_n}\left(l'_{k_{n-1}}-l'_{k_n}\right) \right) + \left(l_{k_{p'}}-l_{k_{p'-1}}\right)\mathbbm{1}_{t \geq \frac{1}{2}} . \] 
\end{itemize}
In essence $\mathbf{Z}^{(0)}$ is concentrated at $x_0,...x_{p'},1-x_{p'},\dots,1-x_0$ with probabilities given by the consecutive differences between the $l_{k_n}$.
 
 The decision rule \[ 1+\eps w(u,v) > \mathbf{Z}^{(\eps)}(u,v)+\mathbf{Z}^{(\eps)}(v,u) \]
 then becomes as $\eps \rightarrow 0$, \[1 \geq \mathbf{Z}^{(0)}(u,v)+\mathbf{Z^{(0)}}(v,u).\] The issue is that the event \[1 = \mathbf{Z}^{(0)}(u,v)+\mathbf{Z}^{(0)}(v,u) \] now happens with positive probability, and when this event happens, there will be multiple edges $v'$ around $u$ such that $\mathbf{Z}^{(0)}(u,v')+\mathbf{Z}^{(0)}(v',u)=1$. In this case, the rule does not define a matching anymore since multiple neighbours of $u$ can be tied at $1$ in the decision rule. 

To break those ties, we will renormalise  $h^{\eps}$ around the limiting atoms $1,x_1,\dots, x_{p'},1-x_{p'},1-x_1,\dots,1$. We will know \textit{a posteriori} that for well behaved degree distributions, the only atoms will appear at $0,\frac{1}{2},1$. In this heuristics section, for simplicity's sake, we assume as such. We refer to Figure~\ref{fig:zoom} for an illustration of this discussion. 

\begin{figure}[h]
    \centering
    \includegraphics[width=0.5\linewidth]{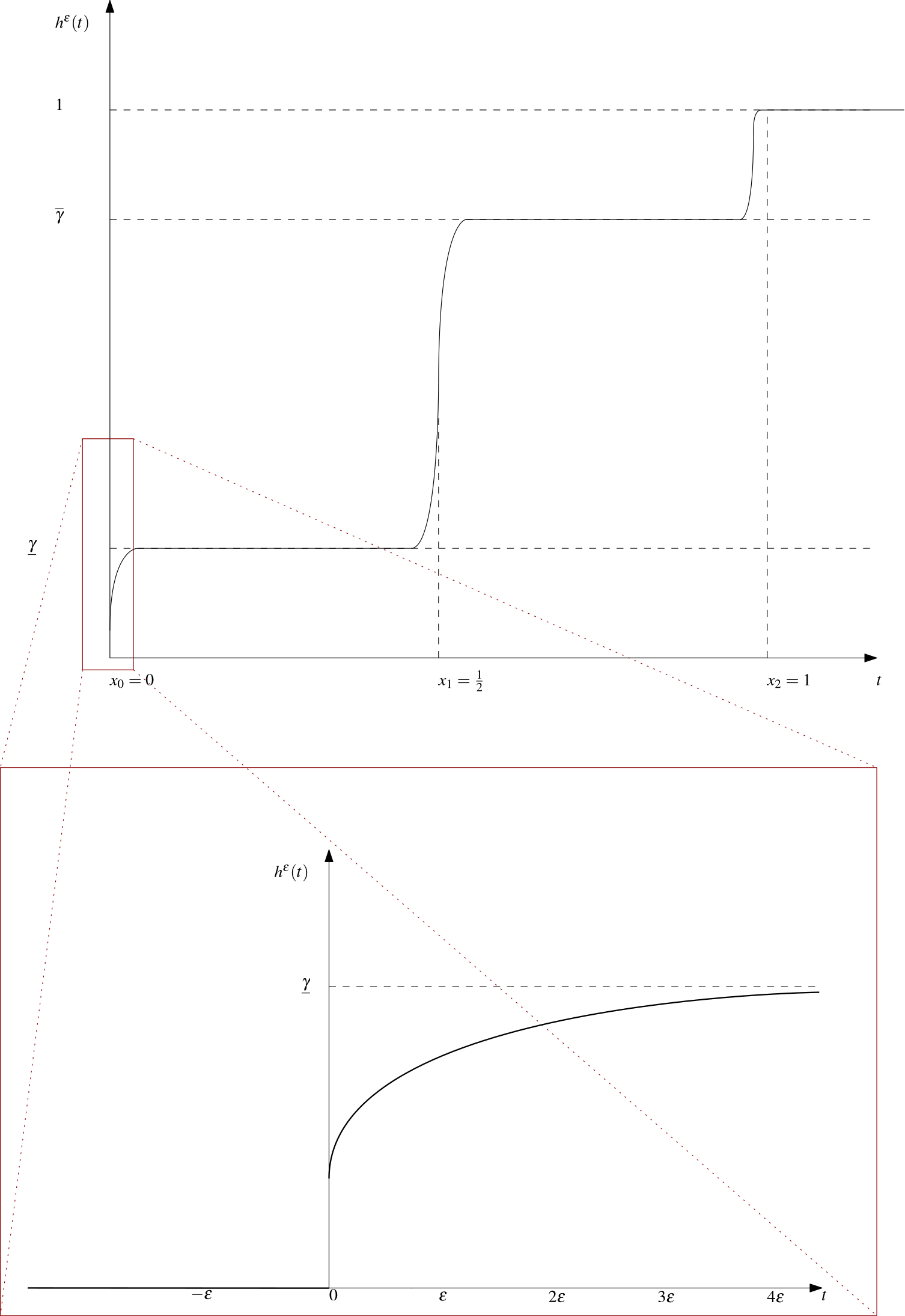}
    \caption{Illustration of the renormalisation procedure, the graph of $h^{0,\eps}$ is shown on the bottom when we divide the scale by $\eps$.}
    \label{fig:zoom}
\end{figure}

Set $h^{0,\eps}:t\mapsto h^{\eps}(\eps t)$,$h^{\frac{1}{2},\eps}: t \mapsto h^{\eps}(\frac{1}{2}+\eps t)$ and $h^{1,\eps}:t \mapsto h^{\eps}(1+\eps t)$. The system of equations on $(h^{0,\eps}$ ,$h^{\frac{1}{2},\eps}$, $h^{1, \eps})$ becomes:
\begin{equation}
\begin{aligned}\label{eq:systemezoomavantlimiteER}
    h^{0,\eps}(t)&= \mathbbm{1}_{t \geq 0} \hat{\phi}(1{-\mathbb{E}[h^{1,\eps}(W-t)]}),\\
    h^{\frac{1}{2},\eps}(t) &= \mathbbm{1}_{t \geq \frac{-1}{2\eps}}\hat{\phi}(1-\mathbb{E}[h^{\frac{1}{2},\eps}(W-t)]), \\
    h^{1,\eps}(t)&= \mathbbm{1}_{t \geq \frac{-1}{\eps}}\hat{\phi}(1-\mathbb{E}[h^{0,\eps}(W-t)]).
\end{aligned}
\end{equation}

Taking $\eps \rightarrow 0$ and assuming simple convergence of $h^{i,\eps}$ to $h_i$ for $i \in \{0, \frac{1}{2},1\}$ one obtains the limiting system on $(h^0, h^{\frac{1}{2}}, h^1)$:
\begin{equation}
\begin{aligned}\label{eq:systemzoomER}
    h^{0}(t)&= \mathbbm{1}_{t \geq 0} \hat{\phi}(1-\mathbb{E}[h^{1}(W-t)]), \\
    h^{\frac{1}{2}}(t) &= \hat{\phi}(1-\mathbb{E}[h^{\frac{1}{2}}(W-t)]) ,\\
    h^{1}(t)&= \hat{\phi}(1-\mathbb{E}[h^{0}(W-t)]).
\end{aligned}
\end{equation}
By taking $t\rightarrow \pm \infty $, we see that $(h^{0}(+\infty),h^1(-\infty))$ and $(h^{\frac{1}{2}}(-\infty), h^{\frac{1}{2}}(+\infty))$ are solutions to the system:
\begin{align*}
\begin{cases}
    x&=\hat{\phi}(1-y), \\
    y&=\hat{\phi}(1-x),
\end{cases}
\end{align*}
which is consistent with the atoms we found for $h^0$ above.
\begin{remark}
Had we not assumed concentration around $0$, $1$ and $\frac{1}{2}$, the system would write:
\begin{equation}\label{eq:systemzoomh0singular}
\left\{
\begin{aligned}
    \forall 0 < j \leq k, h_j(t)&= \hat{\phi}(1-\mathbb{E}[h_{{k-j}}(W-t)]), \\
    h_0(t)&= \mathbbm{1}_{t \geq 0} \hat{\phi}(1-\mathbb{E}[h_{k}(W-t)]),\\
    \forall 0 \leq j \leq k, \lim_{t \rightarrow -\infty} h_j(t) &< \lim_{t \rightarrow + \infty}h_j(t),  \\
    \forall 0 \leq j < k, \lim_{t \rightarrow + \infty}h_j(t) &= \lim_{t \rightarrow -\infty}h_{j+1}(t), 
\end{aligned}
\right.
\end{equation}
where the functions $h_j$ correspond to the renormalisation around $x_j$ and we added further conditions to ensure that no function $h_j$ is redundant and that all the mass is captured by the system.
    In the case where the tree has no leaves, it is possible that $x_0=0$ does not appear in the list of atoms, so there is an alternative system 
\begin{equation}\label{eq:systemzoomh1/2singular}
\left\{
\begin{aligned}
    \forall 0 \leq j \leq k, h_j(t)&= \hat{\phi}(1-\mathbb{E}[h_{{k-j}}(W-t)]), \\
    \forall 0 \leq j \leq k, \lim_{t \rightarrow -\infty} h_j(t) &< \lim_{t \rightarrow + \infty}h_j(t),  \\
    \forall 0 \leq j < k, \lim_{t \rightarrow + \infty}h_j(t) &= \lim_{t \rightarrow -\infty}h_{j+1}(t), \\
    \lim_{t \rightarrow -\infty} h_0(t)&= 0 \\
    \lim_{t \rightarrow +\infty} h_k(t)&=1.
\end{aligned}
\right.
\end{equation}\end{remark}

Now, none of the $h^i$ solution to \eqref{eq:systemzoomER} are cumulative distribution functions of real variables because they assign atoms at infinity, however the vector $(h^0,h^{\frac{1}{2}},h^1)$ can be interpreted as a cumulative distribution function of a multidimensional random variable $\mathbf{Z}$.
Formally, we can represent $\mathbf{Z}=(x,Z)$ as a variable in $\left\{ 0, \frac{1}{2}, 1 \right\} \times \mathbb{R}$ such that for all $a<b, i \in \left\{0, \frac{1}{2},1\right\}$:
\begin{align*}
    \mathbb{P}( x= i , Z \in [a,b])=h^i(b)-h^i(a),
\end{align*}
or as the equivalent $\mathbf{Z}=(Z^0,Z^{\frac{1}{2}},Z^1)$ with $Z^i$ degenerate variables in $\overline{\mathbb{R}}$ and for all $a<b$:
\begin{align*}
    \mathbb{P}( \mathbf{Z} \in ]a,b] \times \{ -\infty \} \times \{ -\infty \} )&= h^0(b)-h^0(a), \\
    \mathbb{P}( \mathbf{Z} \in \{ + \infty \} \times ]a,b] \times \{ -\infty \} )&= h^{\frac{1}{2}}(b)-h^{\frac{1}{2}}(a), \\
    \mathbb{P}( \mathbf{Z} \in \{ +\infty \} \times \{ +\infty \} \times ]a,b] )&= h^1(b)-h^1(a).
\end{align*}
The advantage of this second representation is that the marginals $Z_i$ preserve the meaning of each $h^i$, so they will be adapted when reasoning on their distribution. However, the first representation will usually be more convenient when reasoning on the relationships between variables that verify the recursion on a given tree.
Now, let us see how the recursion
\begin{equation}
    Z^{(\eps)}(u,v)= \max(0,\max_{\substack{ u' \sim v \\ u' \neq u}}\left( w^\eps(v,u')-Z^{(\eps)}(v,u') \right) )
\end{equation}
transfers to $\mathbf{Z}$. Since $w_\eps=1+\eps w$, and $0,\frac{1}{2},1$ indicates the "macroscopic" behavior of $\mathbf{Z}^\varepsilon$, we see that the corresponding system of recursions on $(\mathbb{T},o, \mathbf{Z})$ would be rewritten , using the first representation as:
\begin{equation}\label{eq:recursion(x,Z)}
    \begin{aligned}
        (x,Z)(u,v)=\maxlex\left( (0,0) , \maxlex_{\substack{ u' \sim v \\ u' \neq u}}\left( (1,w(v,u'))-(x,Z)(v,u') \right) \right)  ,
    \end{aligned}
\end{equation}
where the maximum on elements of $\mathbb{R}^2$ is taken in the lexicographic sense and the maximum of an empty list is set to be $(-\infty,-\infty)$. Using the $\mathbf{Z}=(Z^0,Z^{\frac{1}{2}},Z^1)$ representation, this tanslates into:
\begin{equation}
\begin{aligned}\label{eq:systemrecursionZoomER}
    Z^0(u,v) &=  \max\left(0,\max_{\substack{ u' \sim v \\ u' \neq u}}\left( w(v,u')-Z^1(v,u') \right)\right) ,\\
    Z^{\frac{1}{2}}(u,v) &= \max_{\substack{ u' \sim v \\ u' \neq u}}\left( w(v,u')-Z^\frac{1}{2}(v,u') \right) ,\\
    Z^1(u,v) &= \max_{\substack{ u' \sim v \\ u' \neq u}}\left( w(v,u')-Z^0(v,u') \right), 
\end{aligned}
\end{equation}
where the maximum of an empty list is $-\infty$.

Let us now turn to the inclusion rule
\begin{equation*}
    (u,v) \in \mathbb{M}^\eps \Leftrightarrow w^\eps (u,v) > Z^{(\eps)}(u,v)+Z^{(\eps)}(v,u).
\end{equation*}
Again, writing $w^\eps(u,v)=1+\varepsilon w(u,v)$, we see that it would be equivalent for the first representation:
\begin{align*}
    (u,v) \in \mathbb{M} \Leftrightarrow (x,Z)(u,v)+(x,Z)(v,u) \overset{\mathrm{\lex}}{<} (1,w(u,v))
\end{align*}
or, using the second representation:
\begin{align}\label{eq:systeminclusionZoomER}
    (u,v) \in \mathbb{M} \Leftrightarrow \exists x \in \left\{0,\frac{1}{2},1\right\} , Z^{i}(u,v)+Z^{1-x}(u,v) < w(u,v).
\end{align}

In the remainder of the paper, we see that the values of $x$, namely $0$, $\frac{1}{2}$, and $1$ do not matter. Indeed, only their ranking does. So we will instead replace $(Z^0,Z^\frac{1}{2},Z^1)$ by $(Z_0,Z_1,Z_2)$ and $(x,Z)$ with $(i,Z)$ where $i=2x$. In this case, the recursions become
\begin{equation}
    \begin{aligned}
        (i,Z)(u,v)=\max\left( (0,0) , \max_{\substack{ u' \sim v \\ u' \neq u}}\left( (2,w(v,u'))-(i,Z)(v,u') \right) \right),
    \end{aligned}
\end{equation}
and
\begin{equation}\label{eq:recursionZlist2}
\begin{aligned}
    Z_0(u,v) &=  \max\left(0,\max_{\substack{ u' \sim v \\ u' \neq u}}\left( w(v,u')-Z_2(v,u') \right)\right) \\
    Z_{1}(u,v) &= \max_{\substack{ u' \sim v \\ u' \neq u}}\left( w(v,u')-Z_1(v,u') \right) \\
    Z_2(u,v) &= \max_{\substack{ u' \sim v \\ u' \neq u}}\left( w(v,u')-Z_0(v,u') \right).
\end{aligned}
\end{equation}
The inclusion rules become
\begin{align}
    &(u,v) \in \mathbb{M} \Leftrightarrow \exists i \in \left\{0,1,2\right\} , Z_{i}(u,v)+Z_{2-i}(u,v) < w(u,v),
\end{align}
and
\begin{align}
    (u,v) \in \mathbb{M} \Leftrightarrow (i,Z)(u,v)+(i,Z)(v,u) \overset{\mathrm{\lex}}{<} (2,w(u,v)).
\end{align}
We also denote by $(h_0,h_1,h_2)$ the renamed vector $(h^0,h^{\frac{1}{2}},h^1)$. This list notation where we discard the location of the atoms is convenient in the rest of the paper as we will have to deal with an unknown number of jumps in the general case. This avoids having to remember the location of jumps which would complexify notation.

\subsection{Recovering Karp-Sipser formula}

In this subsection, we will see how to compute the known Karp-Sipser formula for maximum matchings from our renormalised message variables. We will thus assume that $\pi$ is a Poisson distribution of parameter $c>0$. We focus on the harder case $c>e$, where we will assume that $\mathbf{Z}^{0}$ concentrates around $0,\frac{1}{2}$ and $1$ with the relevant fixed points denoted by $\underline{\gamma}$ for the smallest and $\overline{\gamma}$ for the largest. The probability that a uniform edge is matched is then:
\begin{equation}
    \mathbb{P}((i_{a},Z_{a})+(i_{b},Z_{b}) \overset{\lex}{<} (1,w))
\end{equation}
with $(i_{a},Z_{a}),(i_{b},Z_{b})$ and $w$ independent.
As we said before, the representation $(Z_0,Z_1,Z_2)$ is more convenient for computing probabilities, so we can transform $(i_a,Z_a)$ into $(Z_{a,0},Z_{a,1},Z_{a,2})$ and $(i_b,Z_b)$ into $(Z_{b,0},Z_{b,1},Z_{b,2})$. In this representation, this probability becomes:
\begin{equation}
    \mathbb{P} (  Z_{a,2}+Z_{b,0} < w \text{ or }Z_{a,1}+Z_{b,1}<w \text{ or } Z_{a,0}+Z_{b,2}<w).
\end{equation}
We can then express this by integrating with respect to the value of $(Z_0^0,Z_0^1,Z_0^2)$. With a slight abuse of notation, we will write $\mathrm{d}h_i$ for the \textbf{measure} corresponding to $Z_i$ on $\mathbb{R}$ without the atoms at $\pm \infty$. This gives:
\begin{equation}\label{eq:ERtaillecalcul}
\begin{aligned}
    &\int_{-\infty}^{+\infty}   \mathbb{E}\left[h_0(W-t)\right] \mathrm{d}h_2(t) + \int_{-\infty}^{+\infty}\mathbb{E}\left[h_{1}(W-t)\right]\mathrm{d}h_{1}(t) + \int_{0}^{+\infty}\mathbb{E}\left[h_{2}(W-t)\right]\mathrm{d}h_0(t).
\end{aligned}
\end{equation}
We can rewrite Equations~\eqref{eq:systemzoomER} as:
\begin{equation*}
\begin{aligned}
    \mathbb{E}[h_{2}(W-t)]&=\mathbbm{1}_{t \geq 0} \frac{-\ln(h_{0}(t))}{c},  \\
    \mathbb{E}[h_{1}(W-t)]&=\frac{-\ln(h_{1}(t))}{c}, \\
    \mathbb{E}[h_{0}(W-t)]&=\frac{-\ln(h_{2}(t))}{c} .
\end{aligned}
\end{equation*}
Injecting these expressions into formula~\eqref{eq:ERtaillecalcul} gives:
\begin{equation}
    \int_{-\infty}^{+\infty}   \frac{-\ln(h_{2}(t))}{c} \mathrm{d}h_2(t) + \int_{-\infty}^{+\infty}\frac{-\ln(h_{1}(t))}{c}\mathrm{d}h_{1}(t) + \int_{0}^{+\infty} \frac{-\ln(h_{0}(t))}{c}\mathrm{d}h_0(t).
\end{equation}
We can now substitute $u=h_i(t)$ in each integral. We must be careful that $\mathrm{d}h_0$ has an atom at $0$ with weight $\beta:=h_0(0)$. We thus get:
\begin{align}
    &\int_{\overline{\gamma}}^{1} \frac{-\ln(u)}{c}\mathrm{d}u + \int_{\underline{\gamma}}^{\overline{\gamma}} \frac{-\ln(u)}{c}\mathrm{d}u + \int_{\beta}^{\underline{\gamma}} \frac{-\ln(u)}{c} \mathrm{d}u  + \frac{-\beta\ln(\beta)}{c} \nonumber\\
    &=\int_{\beta}^{1}\frac{-\ln(u)}{c}\mathrm{d}u-\frac{\beta\ln(\beta)}{c} \nonumber \\
    &= \frac{1-\beta}{c}. \label{eq:ERtailleatomzero}
\end{align}
So the asymptotic matching size only depends on the atom $\beta$ of $h_0$ at $0$.

It remains to express $\beta$ as a function of $\overline{\gamma}$ and $\underline{\gamma}$. To this end, we will find an equation linking these three values by considering the following probability:
\begin{equation*}
    \mathbb{P}(Z_{a,0}+Z_{b,2} < W).
\end{equation*}
We will compute this probability in two different ways to obtain the desired equation. By conditioning with respect to $Z_{b,2}$ or with respect to $Z_{a,0}$, we get the identity:
\begin{equation*}
    \int_{0}^{\infty} \mathbb{E}\left[ h_2(W-t) \right]\mathrm{d}h_0(t)=\mathbb{P}(Z_{b,2}=-\infty, Z_{a,0}<\infty)+\int_{-\infty}^{0} \mathbb{E} \left[ h_0(W-t) \right]\mathrm{d}h_2(t) .
\end{equation*}
Using the same trick as above gives the following:
\begin{align*}
    \frac{\beta\ln(\beta)}{c}-\int_{\beta}^{\underline{\gamma}}\frac{\ln(u)}{c}  = \overline{\gamma}\underline{\gamma}-  \int_{\overline{\gamma}}^{1} \frac{\ln(u)}{c} \mathrm{d}u .
\end{align*}
Computing the integrals gives
\begin{align*}
\underline{\gamma}\ln(\underline{\gamma})-\underline{\gamma} - \beta\ln(\beta)+\beta + \beta\ln(\beta)= \overline{\gamma}-1-\overline{\gamma}\ln(\overline{\gamma})- c\overline{\gamma}\underline{\gamma}.
\end{align*}
Isolating $\beta$, we end up with:
\begin{equation*}
\beta=-c\overline{\gamma}\underline{\gamma}+\overline{\gamma}+\underline{\gamma}-1-\overline{\gamma}\ln(\overline{\gamma})-\underline{\gamma}\ln(\underline{\gamma}).
\end{equation*}
Recalling that $\overline{\gamma}=e^{-c\underline{\gamma}}$ and $\underline{\gamma}=e^{-c\overline{\gamma}}$, the last two terms are equal to $c\overline{\gamma}\underline{\gamma}$ hence
\begin{equation}  \beta=c\underline{\gamma}\overline{\gamma}+\underline{\gamma}+\overline{\gamma}-1.
\end{equation}
Injecting into expression~\eqref{eq:ERtailleatomzero}, we obtain that the asymptotic probability that a uniform edge is in a maximum matching of $G_n$ is:
\begin{equation}
    \frac{1}{c}\left[ 2-\overline{\gamma}-\underline{\gamma} -c\underline{\gamma}\overline{\gamma}\right].
\end{equation}
By virtue of Proposition~\ref{prop:chgmtpdv}, we obtain the formula for the vertex size of maximum matchings after multiplying by $c$. We precisely recover Karp-Sipser's formula (see \cite{Karp}) for the vertex size of the maximum matching.

\section{Maximum size maximum weight matching on UBGW trees}\label{sec:infinitematching}
In this section, we rigorously build the maximum size maximum weight matching on unimodular BGW trees by directly constructing the variables $\mathbf{Z}$ from the previous section and proving they indeed define the unique (in law) unimodular matching that maximises the weight while maximizing the size.
We will treat the case where the BGW tree has leaves (i.e. $\hat{\phi}(0)>0$). This guarantees that we are in the case of System~\eqref{eq:systemzoomh0singular} and not of System~\eqref{eq:systemzoomh1/2singular}.
This alternative does not change much of the proofs, because we will mostly rely on the recursive equation which writes the same regardless of whether there are leaves or not.  We will thus concentrate on the $\hat{\phi}(0)>0$ case to avoid having to constantly treat two separate cases. 

\subsection{Constructing the maximum size maximum weight matching}
The goal of this section is to construct the joint law $(\mathbb{T}, (i,Z)(u,v)_{(u,v)\in\overset{\rightarrow}{E}})$ from the heuristics. For a given solution $\mathcal{H}=(k,(h_0,...,h_k))$ to System~\eqref{eq:systemzoomh0singular}, we write $\zeta_\mathcal{H}$ the distribution of a variable $\mathbf{Z}_{\mathcal{H}}=(Z_{\mathcal{H},0},....,Z_{\mathcal{H},k})$ on $\overline{\mathbb{R}}^k$ given by:
\begin{equation}\label{eq:zetatranslation}
\begin{aligned}
    \mathbb{P}( \mathbf{Z}_\mathcal{H} \in ]a,b] \times \{ -\infty \} \times \cdots \times \{ -\infty \} )&= h^0(b)-h^0(a), \\
    \mathbb{P}( \mathbf{Z}_\mathcal{H} \in \{ + \infty \} \times ]a,b] \times \cdots \times \{ -\infty \} )&= h^{1}(b)-h^{1}(a), \\
    \cdots \\
    \mathbb{P}( \mathbf{Z}_\mathcal{H} \in \{ +\infty \}  \times \cdots \times \{ +\infty \} \times ]a,b] )&= h^k(b)-h^k(a).
\end{aligned}
\end{equation}
Note that the random variable $
\mathbf{Z}_{\mathcal{H}}$ has only one non degenerate coordinate. We will denote by $i_{\mathcal{H}}$ the corresponding coordinate and by $Z_\mathcal{H}$ the associated value. Formally, this means that
\[
i_{\mathcal{H}} := \min \{ j :\, Z_{\mathcal{H},j} > - \infty \}
\text{ and }
Z_{\mathcal{H}} := Z_{\mathcal{H},i_\mathcal{H}}.
\]
This defines a bijective operator that transforms $\mathbf{Z}_\mathcal{H}$ into $(i_\mathcal{H},Z_\mathcal{H})$ whose distribution is supported by $\{0,\ldots ,k\} \times \mathbb{R}$.
We denote the distribution of this random variable by $\zeta_{\mathcal{H}}'$. It verifies
\begin{align}\label{eq:zeta'translation}
    \forall 0 \leq j \leq k, t \in \mathbb{R}, \, \mathbb{P}_{\zeta_\mathcal{H}'} \left( 
i_\mathcal{H}=j, Z_\mathcal{H} \leq t\right)=h_j(t)-\lim_{t' \rightarrow -\infty}h_j(t').
\end{align}

\bigskip

We want to prove the following proposition:
\begin{prop}\label{prop:Zconstruction}
    Let $\mathbb{T}$ be a unimodular decorated BGW tree with reproduction law $\pi$ and edge weight law $\omega$. Let $\mathcal{H}$ be a solution to the system given by System~\eqref{eq:systemzoomh0singular} with input $\pi$ and $\omega$.
    \begin{enumerate}[label=\roman*)]
        \item There exists a random decorated tree $(\mathbb{T}',((i_{\mathcal{H}},Z_{\mathcal{H}})(u,v))_{(u,v) \in \overset{\rightarrow}{E}})$ such that the marginal law of each $(i_{\mathcal{H}},Z_{\mathcal{H}})(u,v)$ is $\zeta'_{\mathcal{H}}$, the marginal law of $\mathbb{T}'$ is the law of $\mathbb{T}$ and such that for any $(u,v) \in \overset{\rightarrow}{E}$:
        \begin{equation}\label{eq:systemrecursionzoombis}
    \begin{aligned}
        (i_{\mathcal{H}},Z_{\mathcal{H}})(u,v)=\maxlex\left( (0,0) , \maxlex_{\substack{ u' \sim v \\ u' \neq u}}\left( (k,w(v,u'))-(i_{\mathcal{H}},Z_{\mathcal{H}})(v,u') \right) \right).
    \end{aligned}
\end{equation}
        \item The rule
        \begin{align}\label{eq:decisionrule}
            (u,v) \in \mathbb{M}_{\mathcal{H}} \Leftrightarrow (i_{\mathcal{H}},Z_{\mathcal{H}})(u,v)+(i_{\mathcal{H}},Z_{\mathcal{H}})(v,u)\overset{\lex}{<} (k,w(u,v)) 
        \end{align}
        defines a matching on $\mathbb{T}$. This rule is equivalent to the following vertex rule:
        $u$ is matched if and only if there exists $v\sim u$ such that
\begin{equation}\label{eq:vertexrulebis}
        \vphantom{\underset{v' \sim u}{\overset{\lex}{\argmax}}}
        (k,w(u,v))-(i_{\mathcal{H}},Z_{\mathcal{H}})(u,v) \overset{\lex}{>} (0,0),
\end{equation}
in which case $u$ is matched to the vertex $v$ defined by
\begin{equation}
        v=\underset{v' \sim u}{\overset{\lex}{\argmax}} \left( (k,w(u,v'))-(i_{\mathcal{H}},Z_{\mathcal{H}})(u,v') \right).
\end{equation}

        \item The random decorated tree $(\mathbb{T}', (i_{\mathcal{H}},Z_{\mathcal{H}})_{(u,v) \in \overset{\rightarrow}{E}})$ viewed as $\mathbb{T}'$ with an additional decoration is unimodular.
        \end{enumerate}
\end{prop}

We will first show that distributions $\zeta'_{\mathcal H}$ satisfying~\eqref{eq:zeta'translation} exist. Let $W$ be a variable of continuous law $\omega$, we seek to find  a positive integer $k$ such that a solution to the following system of increasing c\`adl\`ag functions exists:
\begin{subequations} \label{eq:system}
\begin{align}
    \forall 0 < j \leq k, h_j(t)&= \hat{\phi}(1-\mathbb{E}[h_{{k-j}}(W-t)])\label{eq:systemzoomh} \\
    h_0(t)&= \mathbbm{1}_{t \geq 0} \hat{\phi}(1-\mathbb{E}[h_{k}(W-t)])\label{eq:systemzoomhzero}\\
    \forall 0 \leq j \leq k, \lim_{t \rightarrow -\infty} h_j(t) &< \lim_{t \rightarrow + \infty}h_j(t) \label{eq:systemzoomnondegenerate} \\
    \forall 0 \leq j < k, \lim_{t \rightarrow + \infty}h_j(t) &= \lim_{t \rightarrow -\infty}h_{j+1}(t). \label{eq:systemzoomhnoatoms}
\end{align}
\end{subequations}
Conditions \eqref{eq:systemzoomh} and \eqref{eq:systemzoomhzero} translate the recursion that the renormalised variables must satisfy. Condition \eqref{eq:systemzoomnondegenerate} guarantees that $h_j$ is non-trivial and finally condition \eqref{eq:systemzoomhnoatoms} guarantees that the renormalisation procedure captures the entire mass of $Z^\eps$ as $\eps \rightarrow 0$.
The goal of the next lemma is thus to show that such a solution exists:
\begin{lemma}\label{prop:distributionzoom}
    Assume $t \mapsto \hat{\phi}(1-\hat{\phi}(1-t))$ has a finite amount $p \in \mathbb{N}$ of fixed points in $[0,1]$, then there exists $k \leq p $ such that a solution to the system \eqref{eq:system} exists.
\end{lemma}

\begin{proof}
    We will prove the proposition when the weights have compact support, it is then extended to $L^1$ weights by approximation.    
    By Theorem~3 of \cite{enriquez2024optimalunimodularmatching}, we know that for any $\eps >0$ there exists a unique solution $h^{\eps}$ to Equation~\eqref{eq:hequation}:
    \begin{equation*}
    \forall t \in \mathbb{R}, h^{\eps}(t)=\mathbbm{1}_{t \geq 0} \hat{\phi}(1-\mathbb{E}[h^{\eps}(1+\eps W-t)]).
\end{equation*}
    We will construct $h_j(t)$ by considering convergent subsequences of $h^{\eps}(i_\eps+\eps t)$ for appropriately chosen sequences $i_\eps$.
    Consider $x_1<....<x_p$ the ordered list of fixed points of $t \mapsto \hat{\phi}(1-\hat{\phi}(1-t))$.
    We will recursively build $h_j$ for increasing $j$. Refer to Figure~\ref{fig:zoombis} for an illustration of this proof.

    \begin{figure}[h]
        \centering
        \includegraphics[width=0.5\linewidth]{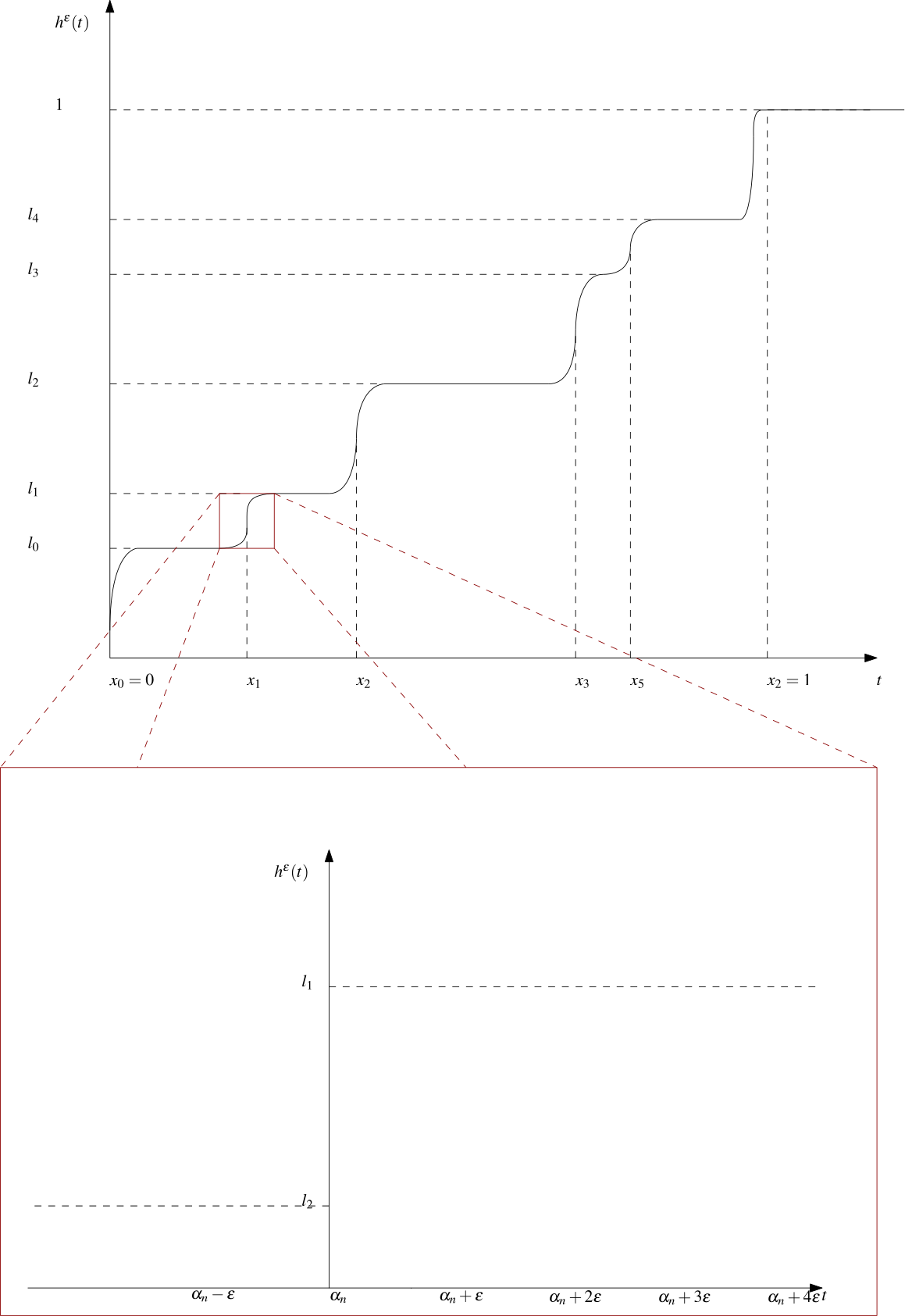}
        \caption{Illustration of the proof, we construct $h_1$ by renormalising as shown.}
        \label{fig:zoombis}
    \end{figure}
    
    First, let us consider a subsequence $\eps_n$ such that $h^{\eps_n}(0)$ converges to some $\beta$, since $\beta\geq \hat{\phi}(0)>0$, we must have that $\beta >0$.
    By dominated convergence theorem, $h^{\eps}$ is continuous on $\mathbb{R}_{+}^{*}$ for all $\eps>0$. 
    Let us consider the sequence of functions $\underline{g}^{(n)}: t\mapsto  h^{\eps_n}(\eps_n t)$ and $\overline{g}^{(n)}: t \mapsto h^{\eps_n}(1+\eps_n t)$.
    By Helly's selection Theorem there exists a subsequence of $\underline{g}^{(n)}$ that converges simply to some function that we will name $h_{0}$.
    Since $\underline{g}^{(n)}$ and $\overline{g}^{(n)}$ verify
    \begin{align*}
        \forall t \in \mathbb{R},  \overline{g}^{(n)}(t) &= \mathbbm{1}_{(1+\eps_n t) \geq 0} \hat{\phi} \left( 1- \mathbb{E}[h^{\eps_n}(1+\eps_nW-(1+\eps_n t)]\right). \\
        &= \mathbbm{1}_{(1+\eps_n t) \geq 0} \hat{\phi}\left( 1- \mathbb{E} [\underline{g}^{(n)}(W-t)] \right),
    \end{align*}
    we will set $h_{k}$ to be the corresponding limit to $\overline{g}^{(n)}$.
    Taking $n \rightarrow \infty$ in the equation above shows that $(h_0,h_k)$ verifies Equation~\eqref{eq:systemzoomhzero}.
    By construction, $h_{0}(0) \rightarrow \beta$ , taking $t \rightarrow \infty$ in Equation~\eqref{eq:systemzoomhzero}, we get that $h_{0}(+ \infty)=\hat{\phi}(1-h_{k}(-\infty))= \hat{\phi}(1-\hat{\phi}(1-h_{0}(+\infty))$.
    So $h_{0}(+\infty)$ is among the $x_l$. Since $\hat{\phi}(0)>0$, $0$ cannot be among the $x_l$ hence $\underset{t \rightarrow -\infty}{\lim}h_0(t)=0 < \underset{t \rightarrow \infty}{\lim}h_0(t)$ so Condition~\eqref{eq:systemzoomnondegenerate} is satisfied.

    Now let us build $h_{1}$ and $h_{{k-1}}$.
    Let $l_0$ such that $x_{l_0}= \underset{t \rightarrow \infty}{\lim}h_0(t)$, We then pick $\alpha_n=(h^{\eps_n})^{-1}(\frac{x_{l_0}+x_{l_0+1}}{2})$ which is well defined for $n$ large enough since $h^{\eps_n}(0) \leq  x_{l_0} < \frac{x_{l_0}+x_{l_0+1}}{2} < 1$.
    Let us consider the sequence of functions $\underline{g}^{(n)}: t\mapsto  h^{\eps_n}(\alpha_n + \eps_n t)$ and $\overline{g}^{(n)}: t \mapsto h^{\eps_n}(1-\alpha_n+\eps_n t)$.
    Once again, by Helly's Selection theorem,  we can find a subsequence of $\overline{g}^{(n)}$ that converges simply to some $h_{1}$. Since $(\underline{g}^{(n)},\overline{g}^{(n)})$ verifies:
    \begin{align*}
        \forall t \in \mathbb{R},  \underline{g}^{(n)}(t) &= \mathbbm{1}_{\alpha_n+\eps_n t \geq 0} \hat{\phi} \left( 1-\alpha_n- \mathbb{E}[h^{\eps_n}(1-\eps_n(t+W)]\right). \\
        &= \mathbbm{1}_{\alpha_n+\eps_n t \geq 0} \hat{\phi}\left( 1- \mathbb{E} [\overline{g}^{(n)}(-W-t)] \right),
    \end{align*}
    taking $n \rightarrow \infty$ shows that $\underline{g}^{(n)}$ also converges simply to some $h_{k-1}$ and the couple $(h_{1},h_{k-1})$ satisfies Equation~\eqref{eq:systemzoomh}.
    Now taking $t \rightarrow \pm \infty$ in Equation~\eqref{eq:systemzoomh} shows that $\underset{t \rightarrow \infty}{\lim}h_{1}(t)$ is once again among the $x_l$. Since this value is bigger than $h_{1}(0)=\frac{x_{l_0}+x_{l_0+1}}{2}$ by construction, it is thus strictly bigger than $g(0)$ and $\underset{t \rightarrow -\infty}{\lim}h_{1}(t)$. Condition~\eqref{eq:systemzoomnondegenerate} is satisfied.
    Finally, since we extract around different reals and that the original function is increasing, it is obvious that we have $h_1(-\infty) \geq h_{0}(+\infty)$. But since we extracted just above $x_{l_0}$ and that $h_1(-\infty)$ is among the fixed points, we must have that $h_1(-\infty)=x_{l_0}$ so condition~\ref{eq:systemzoomhnoatoms} is satisfied.
    
    Now for increasing $j$, as long as $x_{l_j}:=h_j(+\infty) < h_{k-j}(-\infty)$, set $\alpha_n=(h^{\eps_n})^{-1}\left(\frac{x_{l_j}+x_{l_{j+1}}}{2}\right)$ and repeat the first step to define $h_{j+1}$ and $h_{k-j+1}$.
    At the end, there are two cases:
    \begin{enumerate}
    \item If $h_j(+\infty) = h_{k-{j}}(-\infty)$, then it must be that $\alpha_n \rightarrow \frac{1}{2}$. Furthermore there must exist a constant $b$ such that $\alpha_n=\frac{1}{2}+c \eps + o(\eps)$ or $h_j$ and $h_{k-j}$ would be too far part for equality to hold. In this case, we can simply consider $\alpha_n=\frac{1}{2}$ instead. With this choice, we have $h_j=h_{k-j}$ by definition. We then decrease $k$ so that $k=2j$, and we are done.
    \item If $h_j(+\infty) > h_{k-j}(-\infty)$, then we can set $k=p-j$, replace $h_j$ and $h_{k-j}$ defined by their common value $h_{j}$.
    \end{enumerate}
\end{proof}

A direct consequence of Lemma~\ref{prop:distributionzoom} is the following:
\begin{coro}\label{coro:Zlawexistencebis}
    There is a distribution vector $\mathcal{H}$ such that the law $\zeta'_{\mathcal{H}}$ constructed by Equation~\eqref{eq:zeta'translation} on $\mathbb{R}^2$ satisfies the following recursive distributional equation:

    Let $(i_{\mathcal{H}}^{(m)},Z_{\mathcal{H}}^{(m)})_{m \geq 0}$ be i.i.d of law $\zeta'_{\mathcal{H}}$, $w_m$ an i.i.d sequence of law $\omega$ and $N$ of law $\hat{\pi}$ all mutually independent, Then the following equality in law holds:
    \begin{equation}\label{eq:Zlawbis}
    (i_{\mathcal{H}}^{(0)},Z_{\mathcal{H}}^{(0)}) \overset{\mathcal{L}}{=} \maxlex \left( (0,0), \maxlex_{1 \leq m \leq N}  \left( (k,w_m)-(i_{\mathcal{H}}^{(m)},Z_{\mathcal{H}} ^{(m)}) \right)  \right). 
    \end{equation}
\end{coro}

We are now ready to prove Proposition~\ref{prop:Zconstruction}:

\begin{proof}[Proof of Proposition~\ref{prop:Zconstruction}]

The proof is based on Aldous' objective method. It is the renormalised analogue of the proof of Proposition~3.2 of \cite{enriquez2024optimalunimodularmatching}.

\begin{figure}
    \centering
    \includegraphics[width=0.8\linewidth]{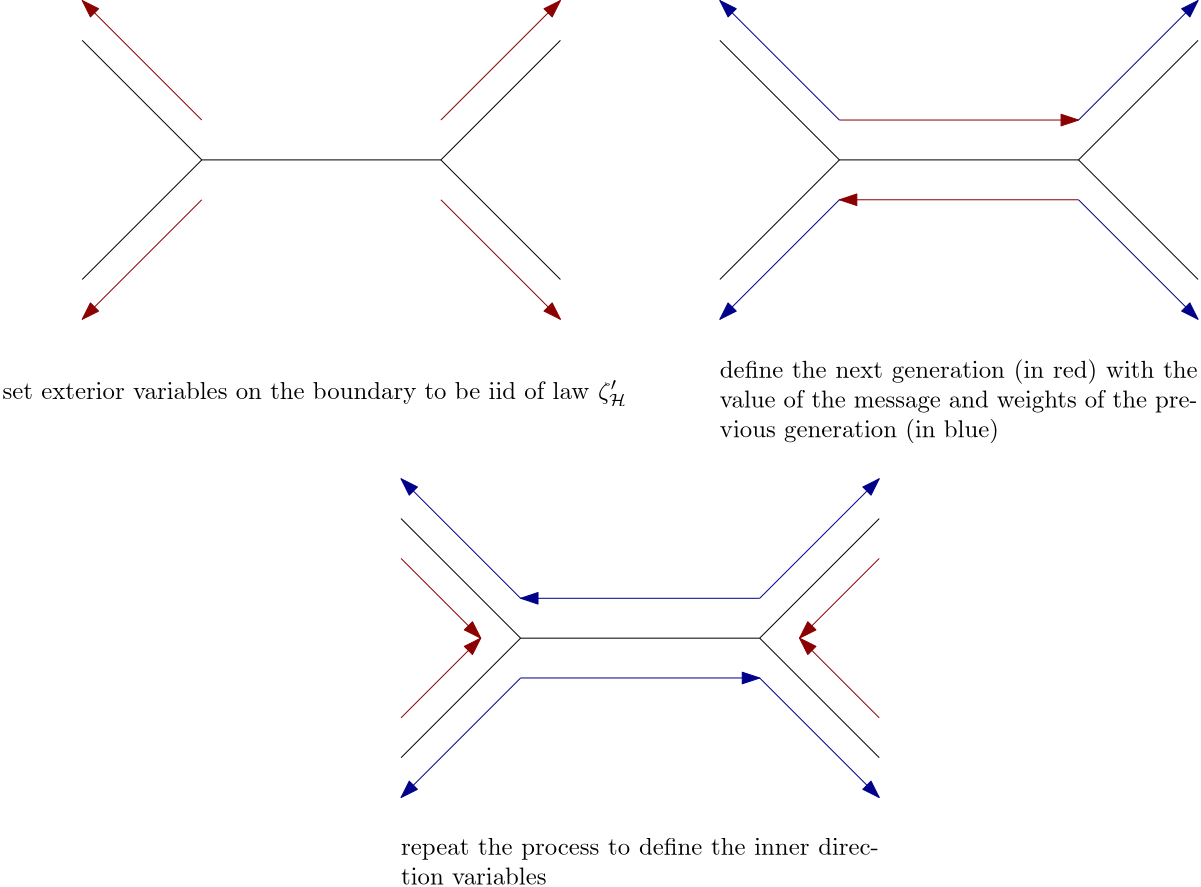}
    \caption{Illustration of the definition of the tree and messages.}
    \label{fig:kolmogorov}
\end{figure}
i)  For an illustration of this proof, we refer to Figure~\ref{fig:kolmogorov}. Let $H \in \mathbb{N}^{*}$, recall $N_H(\mathbb{T})$ the $H-$neighbourhood of the root edge of $\mathbb{T}$. 
We call the depth of a non-oriented edge $\{u,v\}$ its distance to the root edge. For directed edges, we will identify its depth with the depth of its non oriented version.
Let $E_p$ be the set of directed edges  $(u,v) \in \overset{\rightarrow}{E}$ of depth $p$ pointing away from the root. We set $(i_{\mathcal{H}},Z_{\mathcal{H}}(u,v))_{(u,v) \in E_H}$ to be independent variables with law $\zeta_{\mathcal{H}}'$ from the previous lemma. 
\bigskip

We can then use Recursion~\eqref{eq:systemrecursionzoombis}: \begin{equation*}
        (i_{\mathcal{H}},Z_{\mathcal{H}})(u,v)=\maxlex\left( (0,0),\maxlex_{\substack{u' \sim v \\ u' \neq u}} \left( (k,w(v,u'))-(i_{\mathcal{H}},Z_{\mathcal{H}})(v,u'))  \right) \right)
    \end{equation*}
    to define $((i_\mathcal{H},Z_{\mathcal{H}})(u,v))_{(u,v) \in E_{H-1}}$. 
By induction, we define every $(i_{\mathcal{H}},Z_{\mathcal{H}})(u,v)$ for $(u,v)$ pointing away from the root and then on the edge-root and its symmetric.
We can then define $(i_{\mathcal{H}},Z_{\mathcal{H}}(u,v))$ on $E_{-1}$ the set of directed edges of depth $1$ pointing towards the root, then by the same induction, we define them on all $E_{-p}$, the set of directed edges pointing towards the root of depth $p$, for $p$ running from $1$ to $H$. 

In this way, we have defined $(i_{\mathcal{H}},Z_{\mathcal{H}})(u,v)$ on $N_H(\mathbb{T})$, and we can see that because $\zeta_{\mathcal{H}}'$ is an invariant law for the RDE, the restriction of $(i_{\mathcal{H}} ,Z_{\mathcal{H}})(u,v)$ to $N_{H-1}(\mathbb{T})$ has the same law as if we defined it directly on $N_{H-1}(\mathbb{T})$. 
By Kolmogorov's extension theorem, we deduce that there exists a process $(\mathbb{T}', (i_{\mathcal{H}},Z_{\mathcal{H}}))$ such that the first marginal's law is $\mathbb{T}$ and $(i_{\mathcal{H}},Z_{\mathcal{H}})$ satisfies Recursion~\eqref{eq:systemrecursionzoombis} on $\mathbb{T}'$ with the prescribed law.

ii) For an illustration of this proof, refer to Figure~\ref{fig:vertexrule}.

\begin{figure}
    \centering
    \includegraphics[width=0.6\linewidth]{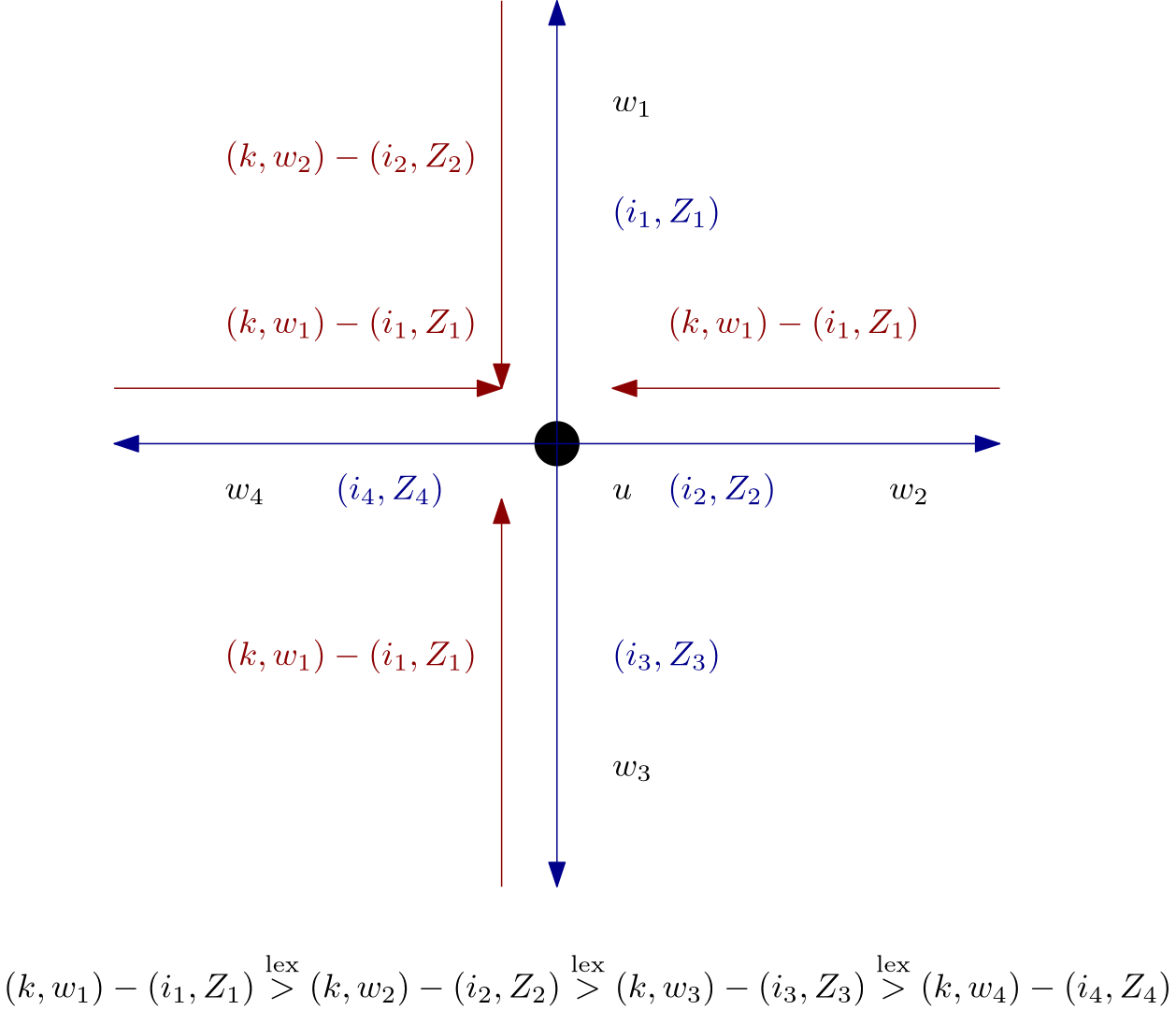}
    \caption{The ordering given at the bottom determines the values of the messages pointing towards $u$ as a function of the messages pointing away and the weights.}
    \label{fig:vertexrule}
\end{figure}

First let us prove the equivalent rule:
Take any other $v' \sim u,$ we have that
\begin{align*}
 (i_{\mathcal{H}},Z_{\mathcal{H}})(v',u)&= \maxlex \left( (0,0),\maxlex_{v'' \sim u, v'' \neq v'} (k,w(u,v''))-(i_{\mathcal{H}},Z_{\mathcal{H}})(u,v'') \right) \\
 &\overset{\lex}{\geq} (k,w(u,v'))-(i_{\mathcal{H}},Z_{\mathcal{H}})(u,v')
\end{align*} 
so no other $v' \neq u$ verifies the condition. By symmetry arguments by switching $u$ and $v$, the rule does define a matching. Furthermore, to prove the vertex rule, we have that
\begin{align*}
(i_{\mathcal{H}},Z_{\mathcal{H}})(v,u) \overset{\lex}{<} (k,w(u,v))-(i_{\mathcal{H}},Z_{\mathcal{H}})(u,v)
\end{align*}
and by applying the recursion on the left side:
\begin{align*}
    (i_{\mathcal{H}},Z_{\mathcal{H}})(v,u)&= \maxlex \left( (0,0),\maxlex_{v'' \sim u, v'' \neq v} (k,w(u,v''))-(i_{\mathcal{H}},Z_{\mathcal{H}})(u,v'') \right).
\end{align*}
Putting both together gives:
\begin{align*}
     (i_{\mathcal{H}},Z_{\mathcal{H}})(v,u)&= \maxlex \left( (0,0),\maxlex_{v'' \sim u, v'' \neq v} (k,w(u,v''))-(i_{\mathcal{H}},Z_{\mathcal{H}})(u,v'') \right)\\
     &\overset{\lex}{<}(k,w(u,v))-(i_{\mathcal{H}},Z_{\mathcal{H}})(u,v).
\end{align*}
which means exactly that \[ v=\argmaxlex_{v''\sim u} (k,w(u,v''))-(i_\mathcal{H},Z_{\mathcal{H}})(u,v'') \] and that
\[  \maxlex_{v''\sim u} (k,w(u,v''))-(i_\mathcal{H},Z_{\mathcal{H}})(u,v'') > (0,0). \]

Now if we assume that 
\[ v=\argmaxlex_{v''\sim u} (k,w(u,v''))-(i_\mathcal{H},Z_{\mathcal{H}})(u,v'') \] 
and
\[  \maxlex_{v''\sim u} (k,w(u,v''))-(i_\mathcal{H},Z_{\mathcal{H}})(u,v'') > (0,0), \]
 then we have that 
\[ (i_{\mathcal{H}},Z_{\mathcal{H}})(v,u) = \maxlex\left( (0,0),\maxlex_{v'' \sim u, v'' \neq v} (k,w(u,v''))-(i_{\mathcal{H}},Z_{\mathcal{H}})(u,v'') \right).\]
By definition, we excluded the maximum in this list so it is evaluated as the second largest among the
\[ (k,w(u,v''))-(i_{\mathcal{H}},Z_{\mathcal{H}})(u,v'')\]
over $v'' \sim u$ and $(0,0)$. Since the $w$ are non atomic, with probability one, there is no equality between the largest and second largest, hence:
\[ (i_{\mathcal{H}},Z_{\mathcal{H}})(v,u) < (k,w(u,v))-(i_{\mathcal{H}},Z_{\mathcal{H}})(u,v) \]
which is the required inequality.

iii) The goal is to show that the law of $(i_{\mathcal{H}},Z_{\mathcal{H}})$ restricted to a $H-$neighbourhood of $(o_+,v)$ with $v$ chosen uniformly among the children of $o_+$ is the same as the law of $(i_{\mathcal{H}},{Z}_{\mathcal{H}})$ restricted to the $H-$neighbourhood of $(o_-,o_+)$.
Since Recursion~\eqref{eq:systemrecursionzoombis} is preserved, one only needs to show that the exterior variables of  $(i_{\mathcal{H}},Z_{\mathcal{H}})$ on the $H-$boundary of $(o_+,v)$ are i.i.d variables of law $\zeta_{\mathcal{H}}'$.
The $H-$neighbourhood of $(o_+,v)$ is included in the $H+1-$neighbourhood of $(o_-,o_+)$ and can be computed from it. One can check by applying Recursion~\eqref{eq:systemrecursionzoombis} up to twice from the $H+1$-boundary of $(o_-,o_+)$, that we effectively recover variables $(i_{\mathcal{H}},Z_{\mathcal{H}})$ on the $H-$boundary of $(o_+,v)$ that are i.i.d of law $\zeta_{\mathcal{H}}'$.
$(\mathbb{T},(i_{\mathcal{H}},Z_{\mathcal{H}}))$ is thus stationary. 

Clearly, the law of the $(i_{\mathcal{H}},Z_{\mathcal{H}})$ are also symmetric so changing $(o_-,o_+)$ into $(o_+,o_-)$ does not change their laws, $(\mathbb{T},(i_{\mathcal{H}},Z_{\mathcal{H}}))$ is thus reversible. 
\end{proof}

\subsection{Adding Self-Loops}
In the remainder of this section, we will adopt the vertex-rooted point of view. Dealing with partial matchings rather than perfect matchings will be very cumbersome in future proofs. For example, when comparing two matchings, we would need to discuss several cases depending on whether a vertex is matched or not for the two matchings. A simple solution to deal with this is to add self-loops, in which case the graphs always have perfect matchings, and any partial matching on the graph without self-loops can be augmented into a perfect matching on the graph with self-loops. Reciprocally, a perfect matching on a graph with self-loops can be restricted to a partial matching on the corresponding graph without self-loops. 

We need to choose adequate weights for the self-loops so that a partial matching on the graph without self-loops is optimal among all partial matchings of the original graph if and only if the corresponding perfect matching on the augmented graph with self-loops is optimal among all perfect matchings of the augmented graph.

Starting with a system of variables $(i(u,v),Z(u,v))$ for $(u,v) \in \overset{\rightarrow}{E}$ that satisfies recursion \eqref{eq:systemrecursionzoombis}, we want to extend it to variables $(i(v,v),Z(v,v))$ for each self-loop so that we get a similar recursion as \eqref{eq:systemrecursionzoombis} and the decision rule \eqref{eq:decisionrule} constructs the perfect matching of the augmented graph associated to the matching defined by the original variables.

In Recursion~\eqref{eq:systemrecursionzoombis}, the variable $(k,w(u,v))$ informally represents $(1+\eps w(u,v))$ so that the first coordinate $k$ is equal to $k$ times the macroscopic component of the weight of $(u,v)$ whereas the second coordinate $w(u,v)$ is equal to the microscopic component of the weight of $(u,v)$. The definition of the weight of the self-loops will assign a non-constant macroscopic weight to self-loops, this motivates the introduction of a variable $j$ on directed edges that quantifies the macroscopic weight on edges.

Formally, we define the extension
\begin{equation} \label{eq:Tsh}
(\mathbb{T}^{s}_\mathcal{H},o,(j_{\mathcal{H}}^s,w_{\mathcal{H}}^s,(i_\mathcal{H}^s,Z_\mathcal{H}^{s}),(f_i^s)_{i \in \{0,...,I\}})=((V^{s},E^{s}),o,(j_{\mathcal{H}}^s,w^{s}_\mathcal{H}),(i_\mathcal{H}^s,Z_\mathcal{H}^{s}),(f_i^s)_{i \in \{0,....,I\}})
\end{equation}
deterministically on every outcome of $(\mathbb{T},(i_\mathcal{H},Z_\mathcal{H}))=((V,E),o,w,(i_\mathcal{H},Z_\mathcal{H}), (f_i)_{i \in \{0,...,I\}})$. 
We keep the vertex set $V^s=V$, we add self-loops to construct $E^s=E \cup \{ \{u,u\} , u \in V\}$. For $(u,v)$ such that $u \neq v$, simply keep the values:
\begin{equation}
    \begin{aligned}
        (i_\mathcal{H}^s,Z_\mathcal{H}^s)(u,v)&=(i_\mathcal{H},Z_\mathcal{H})(u,v) ,\\
        (j_{\mathcal{H}}^s(u,v),w_{\mathcal{H}}^s(u,v))&=(k,w(u,v)).
    \end{aligned}
\end{equation}

\begin{equation}\label{def:selfloop}
\begin{aligned}
 (i_{\mathcal{H}}^s(u,u),Z_{\mathcal{H}}^s(u,u)) &:= \maxlex_{\substack{ u' \sim u}}\left( (k,w(u,u'))-(i_{\mathcal{H}},Z_{\mathcal{H}})(u,u') \right) \\
 (j_{\mathcal{H}}^s(u,u),w_{\mathcal{H}}^s(u,u))&:= (i_{\mathcal{H}}^s(u,u),Z^s_{\mathcal{H}}(u,u)).
\end{aligned}
\end{equation}

The variables $(i_{\mathcal{H}}^s,Z_{\mathcal{H}}^s)$ verify Recursion~\ref{eq:systemrecursionzoombis} without the $(0,0)$, regardless of whether $u=v$ or $u \neq v$:

\begin{equation}\label{eq:recursionzoomselfloop}
    \begin{aligned}
        (i_{\mathcal{H}}^s,Z_{\mathcal{H}}^s)(u,v)= \maxlex_{\substack{ u' \sim v \\ u' \neq u}}\left( (j_{\mathcal{H}}^s,w_{\mathcal{H}}^s)(v,u')-(i_{\mathcal{H}}^s,Z_{\mathcal{H}}^s)(v,u') \right) 
    \end{aligned}
\end{equation}
Furthermore, the vertex-rule for the matching now becomes:
\begin{align}\label{eq:selflooprule}
    (u,u) \in \mathbb{M}_{\mathcal{H}}^s \Leftrightarrow  (j_{\mathcal{H}}^s,w_{\mathcal{H}}^s)(u,u) \overset{\lex}{<} (0,0).
\end{align}
\begin{align}\label{eq:vertexrule}
    (u,v) \in \mathbb{M}_{\mathcal{H}}^s \Leftrightarrow v=\underset{u' \sim v}{\argmaxlex}\left( (j_{\mathcal{H}}^s,w_{\mathcal{H}}^s)(u,u'))-(i_{\mathcal{H}}^s,Z_{\mathcal{H}}^s)(u,u') \right).
\end{align}

The same trick is used in our first paper \cite{enriquez2024optimalunimodularmatching} on maximum weight matchings. Since then, another paper treating maximum matchings on $\mathbb{Z}^d$ \cite{selfloopZd} by Kesav, Krishnan and Ray introduced the same type of quantity  as $(i_{\mathcal{H}}^s,Z_{\mathcal{H}}^s)(u,u)$, which they call \emph{flexibility}. This variable represents the gap in performance between the ability to use the vertex $u$ in the matching and forcing its absence in the matching.

\subsection{Optimality}
\begin{prop}\label{prop:optimality}
    The matching  $\mathbb{M}_{\mathcal{H}}$ defined in Proposition~\ref{prop:Zconstruction} is optimal on $\mathbb{T}$. Furthermore, for any coupling $(\mathbb{T},\mathbb{M}_{\mathcal{H}},\mathbb{M})$ with $\mathbb{M}$ an optimal matching on $\mathbb{T}$, $\mathbb{M}$ and $\mathbb{M}_{\mathcal{H}}$ share the same set of unmatched vertices.
\end{prop}
Fix $\mathbb{M}^{s}$ any unimodular perfect matching on $\mathbb{T}^s$. Sample $(\mathbb{T}^s,\mathbb{M}^s,\mathbb{M}_{\mathcal{H}}^s)$ as in Proposition \ref{prop:Zconstruction} iii).
 Define the neighbour functions $n_{\mathcal{H}}$ such that for any $v \in V$, $(v,n_{\mathcal{H}}(v)) \in \mathbb{M}^s_{\mathcal{H}}$ and similarly $n$ such that $(v,n(v)) \in \mathbb{M}^s$, we will prove the following lemma:
\begin{restatable}{lemma}{lemperfS}\label{lem:perfS}
There exists random variables $X',Y$ and $Y'$ in $\mathbb{R}^2$ satisfying:
     \begin{align*}
          \Big( \left( \mathbbm{1}_{o \neq n_{\mathcal{H}}(o)} - \mathbbm{1}_{o \neq n(o)}  \right) \mathbbm{1}_{n(o) \neq n_{\mathcal{H}}(o)}  &,\ \left( w(o,n_{\mathcal{H}}(o))\mathbbm{1}_{o \neq n_{\mathcal{H}}(o)} - w(o,n(o))\mathbbm{1}_{o \neq n(o)} \right)\mathbbm{1}_{n(o) \neq n_{\mathcal{H}}(o)}   \Big) + X' \\
          &\overset{\lex}{\geq} Y-Y'.
     \end{align*}
     such that:
     \begin{itemize}
         \item Almost surely, $X' \overset{\lex}{\leq }(0,0)$ and $X'=(0,0)$ if and only if $n_{\mathcal{H}}(o)=n(o)=o$. 
         \item The variables $Y$ and $Y'$ have the same law.
     \end{itemize}
\end{restatable}
\begin{remark}
    In the previous work \cite{enriquez2024optimalunimodularmatching}, this step was done through the expectation of the quantities above. Here there is no guarantee that $Z_{\mathcal{H}}$ is integrable so expectations of $X',Y$ and $Y'$ may not be defined.
\end{remark}

\begin{proof}[Proof of Lemma~\ref{lem:perfS}]
Let us work on the event $n_{\mathcal{H}}(o)\neq n(o)$ from now on.

Set $v_0=o$, $v_{1}=n(o)$, $v_{-1}=n_{\mathcal{H}}(o)$, $v_{-2}=n(v_{-1})=n(n_{\mathcal{H}}(o))$, see Figure~\ref{fig:optimalitydef} for an illustration. 
\begin{figure}
    \centering
    \includegraphics[width=0.5\linewidth]{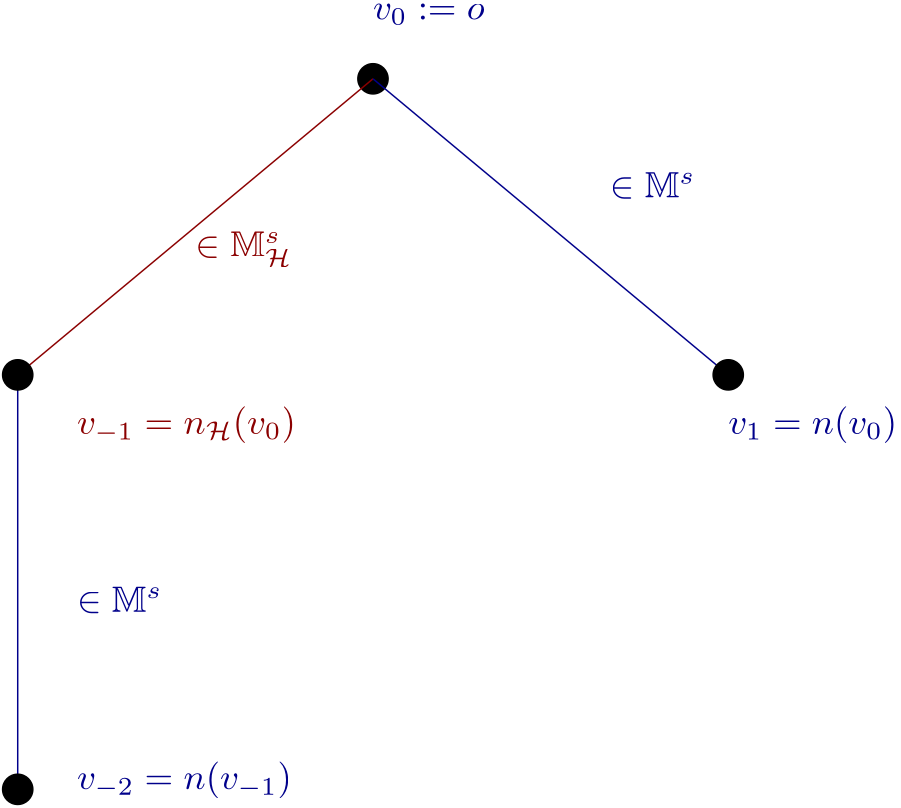}
    \caption{Definition of the vertices $v_0,v_1,v_{-1},v_{-2}$.}
    \label{fig:optimalitydef}
\end{figure}

By definition of $(i^s_{\mathcal{H}},Z^s_{\mathcal{H}})(v_{-2},v_{-1})$:
\begin{align*}
    (i^s_{\mathcal{H}},Z^s_{\mathcal{H}})(v_{-2},v_{-1})=\underset{{\substack{y \sim v_{-1}\\ y \neq v_{-2}}}}{\maxlex}\left((j_{\mathcal{H}}^s,w_{\mathcal{H}}^s)(v_{-1},y)-(i^s_{\mathcal{H}},Z^s_{\mathcal{H}})(v_{-1},y)\right).
\end{align*}
Now since $(v_{-1},v_0)\in \mathbb{M}^s_{\mathcal{H}}$, we know from Proposition~\ref{prop:Zconstruction} ii) that:
\begin{align*}
v_0=\overset{\lex}{\underset{y \sim v_{-1}}{\argmax}}\left((j_{\mathcal{H}}^s,w_{\mathcal{H}}^s)(v_{-1},y)-(i^s_{\mathcal{H}},Z^s_{\mathcal{H}})(v_{-1},y)\right) 
\end{align*}

By hypothesis, $v_{-1} \neq v_{1}$ hence by applying $n$, $n(v_{-1})=v_{-2} \neq v_0=n(v_1)$, we deduce that 
\begin{align*}
    (i^s_{\mathcal{H}},Z^s_{\mathcal{H}})(v_{-2},v_{-1})=(j_{\mathcal{H}}^s,w_{\mathcal{H}}^s)(v_{-1},v_{0})-(i^s_{\mathcal{H}},Z^s_{\mathcal{H}})(v_{-1},v_0).
\end{align*}
Now let us look at $(i^s_{\mathcal{H}},Z^s_{\mathcal{H}})(v_{-1},v_0)$:
\begin{equation}
\begin{aligned}\label{eq:aldous2}
    (i^s_{\mathcal{H}},Z^s_{\mathcal{H}})(v_{-1},v_0) &= \underset{{ y \sim v_0}}{\maxlex}\left( (j_{\mathcal{H}}^s,w_{\mathcal{H}}^s)(v_0,y)-(i^s_{\mathcal{H}},Z^s_{\mathcal{H}})(v_0,y) \right). \\
    &\geq (j_{\mathcal{H}}^s,w_{\mathcal{H}}^s)(v_0,v_1)-(i^s_{\mathcal{H}},Z^s_{\mathcal{H}})(v_0,v_1).
\end{aligned}
\end{equation}
Putting everything together:
\begin{equation}\label{eq:aldous4}
\begin{aligned}
    &\left((j_{\mathcal{H}}^s,w_{\mathcal{H}})(v_{-1},v_0)-(j_{\mathcal{H}}^s,w_{\mathcal{H}}^s)(v_{1},v_0) \right)\mathbbm{1}_{v_{-1} \neq v_1} \\
    =&\left(((i^s_{\mathcal{H}},Z^s_{\mathcal{H}})(v_{-2},v_{-1})+(i^s_{\mathcal{H}},Z^s_{\mathcal{H}})(v_{-1},v_0)-(j_{\mathcal{H}}^s,w_{\mathcal{H}}^s)(v_{1},v_0)\right)\mathbbm{1}_{v_{-1} \neq v_1} \\
    =&\bigg((i^s_{\mathcal{H}},Z^s_{\mathcal{H}})(v_{-2},v_{-1})+(i^s_{\mathcal{H}},Z^s_{\mathcal{H}})(v_0,v_1) \\
    +&\bigg[(i^s_{\mathcal{H}},Z^s_{\mathcal{H}})(v_{-1},v_0)-(i^s_{\mathcal{H}},Z^s_{\mathcal{H}})(v_0,v_1)-(j_{\mathcal{H}}^s,w_{\mathcal{H}}^s)(v_{1},v_0)\bigg]\bigg)\mathbbm{1}_{v_{-1} \neq v_1} \\
    \overset{\lex}{\geq}&\left((i^s_{\mathcal{H}},Z^s_{\mathcal{H}})(v_{-2},v_{-1}) -(i^s_{\mathcal{H}},Z^s_{\mathcal{H}})(v_{0},v_{1})\right)\mathbbm{1}_{v_{-1} \neq v_1} .
\end{aligned}
\end{equation}
Now decompose the weights depending on whether $v_{-1}=v_0$ or $v_{1}=v_0$ and use the fact that $j_{\mathcal{H}}^s(u,v)=k$  if $u \neq v$:
\begin{align*}
    &\left((j_{\mathcal{H}}^s,w_{\mathcal{H}})(v_{-1},v_0)-(j_{\mathcal{H}}^s,w_{\mathcal{H}}^s)(v_{1},v_0) \right)\mathbbm{1}_{v_{-1} \neq v_1} \\
    =&(k,w(v_{-1},v_{0}))\mathbbm{1}_{v_{-1}\neq v_{0}}-(k,w(v_1,v_{0}))\mathbbm{1}_{v_1 \neq v_{0}} \\
    +&(j_{\mathcal{H}}^s,w_{\mathcal{H}}^s)(o,o)\mathbbm{1}_{v_{-1}=v_{0}}-(j_{\mathcal{H}}^s,w_{\mathcal{H}}^s)(o,o)\mathbbm{1}_{v_{1}=v_{0}}.
\end{align*}
Rearranging the different terms, and recalling that we are working on the event $v_1 \neq v_{-1}$, we obtain:
\begin{align*}
    &\left((\mathbbm{1}_{v_{-1}\neq v_0}-\mathbbm{1}_{v_1\neq v_0}),\left(w^s(v_{-1},v_0)\mathbbm{1}_{v_{-1}\neq v_0}-w^s(v_{1},v_0)\mathbbm{1}_{v_1 \neq v_0}\right) \right)\mathbbm{1}_{v_{-1}\neq v_1} \\
    &+ \left((j_{\mathcal{H}}^s, w_{\mathcal{H}}^s)(o,o)  \mathbbm{1}_{v_0=v_{-1}, v_0 \neq v_1}-( j_{\mathcal{H}}^s, w_{\mathcal{H}}^s)(o,o)  \mathbbm{1}_{v_0=v_{1}, v_0 \neq v_{-1}} \right) \\
    &\overset{\lex}{\geq} \left( (i^s_{\mathcal{H}},Z^s_{\mathcal{H}})(v_{-2},v_{-1}) -(i^s_{\mathcal{H}},Z^s_{\mathcal{H}})(v_{0},v_{1})\right)\mathbbm{1}_{v_{-1} \neq v_1}.
\end{align*}
So we set
\begin{align*}
    X'&:=\left[ (j_{\mathcal{H}}^s, w_{\mathcal{H}}^s)(o,o)  \mathbbm{1}_{o=n_{\mathcal{H}}(o),o\neq n(o)}- (j_{\mathcal{H}}^s, w_{\mathcal{H}}^s)(o,o)  \mathbbm{1}_{o=n_{\mathcal{H}}(o),o\neq n(o)}\right], \\
    Y&:=(i^s_{\mathcal{H}},Z^s_{\mathcal{H}})(v_{-2},v_{-1})\mathbbm{1}_{v_{-1} \neq v_1} , \\
    Y'&:=(i^s_{\mathcal{H}},Z^s_{\mathcal{H}})(v_{0},v_{1})\mathbbm{1}_{v_{-1} \neq v_1} .
\end{align*}

First let us show that the condition on $X'$ is satisfied.
Recall that $n_{\mathcal{H}}(o)=o$ if and only if $(j^s(o,o),w^s(o,o)) \overset{\lex}{<} (0,0)$.
\begin{enumerate}
\item If $n_{\mathcal{H}}(o)=o$ and $n(o) \neq o$ then $X'=(j_{\mathcal{H}}^s(o,o),w_{\mathcal{H}}^s)(o,o) \overset{\lex}{<} (0,0)$ as $n_{\mathcal{H}}(o)=o$ 
\item If $n(o)=o$ and $n_{\mathcal{H}}(o) \neq o$ then $X'=-(j_{\mathcal{H}}^s,w_{\mathcal{H}}^s)(o,o) \overset{\lex}{\leq }(0,0)$ as $n_{\mathcal{H}}(o) \neq o$
\end{enumerate}
In either case, $X' \overset{\lex}{\leq }(0,0)$. Since $\omega$ is continuous, we have that $\mathbb{P}((j_{\mathcal{H}}^s,w_{\mathcal{H}}^s)(o,o)=(0,0))=0$ so the inequality in the second case can be replaced by a strict one. Then, almost surely, $X'<(0,0)$ if we are in either of the previous case, so if we don't have $n_{\mathcal{H}}(o)=n(o)=o$.
If $n_{\mathcal{H}}(o)=n(o)=o$, then $X'$ trivially evaluates to $0$, which concludes.

Now we will use unimodularity to show that $Y \overset{\mathcal{L}}{=}Y'$.
Let $B$ be a measurable subset of $\mathbb{R}$. We will show that:
\begin{align*}
    &\mathbb{P}\left( (i^s_{\mathcal{H}},Z^s_{\mathcal{H}})(v_{0},v_{1}) \in B, v_{-1} \neq v_1 \right)=\mathbb{P}\left( (i^s_{\mathcal{H}},Z^s_{\mathcal{H}})(v_{-2},v_{-1}) \in B, v_{-1} \neq v_1 \right).
\end{align*}
Let us recall the definition of $v_k$ so the desired equality rewrites as:
\begin{align}\label{eq:lemmeoptimal}
\mathbb{P}\left((i^s_{\mathcal{H}},Z^s_{\mathcal{H}})(o,n(o)) \in B, n(o)\neq n_{\mathcal{H}}(o) \right) = \mathbb{P}\left(  (i^s_{\mathcal{H}},Z^s_{\mathcal{H}})(n(n_{\mathcal{H}}(o)),n_{\mathcal{H}}(o)) \in B, n(o) \neq n_{\mathcal{H}}(o) \right).
\end{align}
Define the $(\mathbb{T}^s,\mathbb{M}^s,(i^s_{\mathcal{H}},Z^s_{\mathcal{H}}))$ measurable function $f$ on the space of doubly rooted  decorated trees:
\[f(T,M,(i^s_{\mathcal{H}},Z^s_{\mathcal{H}}),a,b)=\mathbbm{1}_{ (i^s_{\mathcal{H}},Z^s_{\mathcal{H}})(a,b)\in B, n(a)\neq n_{\mathcal{H}}(a), b=n(a)   } .  \]
Applying mass-transport principle to $f$ we get:
\[ \mathbb{E} \left[  \sum_{v\in V} f(\mathbb{T}^s,\mathbb{M}^s,(i^s_{\mathcal{H}},Z^s_{\mathcal{H}}),o,v) \right] = \mathbb{E} \left[  \sum_{v\in V} f(\mathbb{T}^s,\mathbb{M}^s,(i^s_{\mathcal{H}},Z^s_{\mathcal{H}}),v,o) \right]  \]
Computing the first expectation yields:
\begin{align*}
\mathbb{E}\left[ \sum_{v\in V}\mathbbm{1}_{(i^s_{\mathcal{H}},Z^s_{\mathcal{H}})(o,v) \in B, n(o) \neq n_{\mathcal{H}}(o), v=n(o)}  \right] 
&= \mathbb{E}\left[ \mathbbm{1}_{(i^s_{\mathcal{H}},Z^s_{\mathcal{H}})(o,n(o)) \in B, n(o) \neq n_{\mathcal{H}}(o)} \right] \\
&= \mathbb{P}\left( (i^s_{\mathcal{H}},Z^s_{\mathcal{H}})(o,n(o)) \in B, n(o) \neq n_{\mathcal{H}}(o)  \right) .
 \end{align*}
 Computing the second expectation yields:
\begin{align*}
\mathbb{E}\left[ \sum_{v\in V}\mathbbm{1}_{((i^s_{\mathcal{H}},Z^s_{\mathcal{H}})(v,o) \in B, n(v) \neq n_{\mathcal{H}}(v), o=n(v)}  \right]  
&= \mathbb{E}\left[ \mathbbm{1}_{(i^s_{\mathcal{H}},Z^s_{\mathcal{H}})(n(o),o) \in B,  n(n(o)) \neq n_{\mathcal{H}}(n(o))} \right] \\
 &= \mathbb{P}\left( (i^s_{\mathcal{H}},Z^s_{\mathcal{H}})(n(o),o) \in B , n(o) \neq n_{\mathcal{H}}(o)  \right) 
 \end{align*}
where we used that $o=n(v)$ is equivalent to $v=n(o)$ and the fact that $n(n(o))\neq n_{\mathcal{H}}(n(o))$ is equivalent to $n(o) \neq n_{\mathcal{H}}(o)$.

We now define another measurable function $f'$ on the same space:
\[f'(T,M,(i^s_{\mathcal{H}},Z^s_{\mathcal{H}}),a,b)= \mathbbm{1}_{(i^s_{\mathcal{H}},Z^s_{\mathcal{H}})(n(a),a)\in B, n(a) \neq n_{\mathcal{H}}(a), b=n_{\mathcal{H}}(a)} .\]
Once again, applying the mass-transport principle to $f'$ we get:
\[ \mathbb{E} \left[  \sum_{v\in V} f'(\mathbb{T}^s,\mathbb{M}^s,(i^s_{\mathcal{H}},Z^s_{\mathcal{H}}),o,v) \right] = \mathbb{E} \left[  \sum_{v\in V} f'(\mathbb{T}^s,\mathbb{M}^s,(i^s_{\mathcal{H}},Z^s_{\mathcal{H}}),v,o) \right]  .\]
Computing the first expectation yields:
\begin{align*}
\mathbb{E}\left[ \sum_{v\in V}\mathbbm{1}_{(i^s_{\mathcal{H}},Z^s_{\mathcal{H}})(n(o),o) \in B, n(o) \neq n_{\mathcal{H}}(o), v=n_{\mathcal{H}}(o)}  \right] &= \mathbb{E}\left[ \mathbbm{1}_{(i^s_{\mathcal{H}},Z^s_{\mathcal{H}})(n(o),o)\in B, n(o) \neq n_{\mathcal{H}}(o)} \right] \\
 &= \mathbb{P}\left( (i^s_{\mathcal{H}},Z^s_{\mathcal{H}})(n(o),o) \in B , n(o) \neq n_{\mathcal{H}}(o)  \right) .
 \end{align*}
 Computing the second expectation yields:
 \begin{align*}
\mathbb{E}\left[ \sum_{v\in V}\mathbbm{1}_{(i^s_{\mathcal{H}},Z^s_{\mathcal{H}})(n(v),v) \in B, n(v) \neq n_{\mathcal{H}}(v), o=n_{\mathcal{H}}(v)}  \right] &= \mathbb{E}\left[ \mathbbm{1}_{(i^s_{\mathcal{H}},Z^s_{\mathcal{H}})(n(n_{\mathcal{H}}(o)),n_{\mathcal{H}}(o))\in B, n(n_{\mathcal{H}}(o)) \neq n_{\mathcal{H}}(n_{\mathcal{H}}(o))} \right] \\
 &= \mathbb{P}\left( (i^s_{\mathcal{H}},Z^s_{\mathcal{H}})(n(n_{\mathcal{H}}(o)),n_{\mathcal{H}}(o)) \in B , n(o) \neq n_{\mathcal{H}}(o)  \right) .
 \end{align*}
where we used that $o=n_{\mathcal{H}}(v)$ is equivalent to $v=n_{\mathcal{H}}(o)$ and the fact that $n(n_{\mathcal{H}}(o)) \neq (n_{\mathcal{H}}(n_{\mathcal{H}}(o)))$ is equivalent to $n(o) \neq n_{\mathcal{H}}(o)$. 

In conclusion, we showed
\[ \mathbb{P}\left(  (i^s_{\mathcal{H}},Z^s_{\mathcal{H}})(o,n(o))\in B, n(o)\neq n_{\mathcal{H}}(o) \right)=\mathbb{P}\left( (i^s_{\mathcal{H}},Z^s_{\mathcal{H}})(n(o),o) \in B , n(o) \neq n_{\mathcal{H}}(o)  \right)\]
\[=\mathbb{P}\left(  (i^s_{\mathcal{H}},Z^s_{\mathcal{H}})(n(n_{\mathcal{H}}(o)),n_{\mathcal{H}}(o))\in B, n(o) \neq n_{\mathcal{H}}(o) \right),\]
yielding \eqref{eq:lemmeoptimal} and finishing the Lemma.
\end{proof}

To conclude we just need the following proposition whose proof we postpone to Section~\ref{sec:Appendix}:
\begin{restatable}{prop}{proprusse}\label{prop:russe}
    Let $X,X',Y,Y'$ be random variables in $\mathbb{R}^2$, such that each marginal of $X$ has finite expectation, such that almost surely:
    \begin{align*}
        &X+X' \overset{\lex}{\geq} Y-Y', \\
        &X' \overset{\lex}{\leq } (0,0),
    \end{align*}
    and $Y$ and $Y'$ share the same law.
    Then:
    \begin{align*}
    \mathbb{E}[X] \overset{\lex}{\geq} (0,0)
    \end{align*}
    with equality implying $X'=(0,0)$ almost surely.
\end{restatable}
We can now turn to the proof of Proposition~\ref{prop:optimality}.
\begin{proof}[Proof of Proposition~\ref{prop:optimality}]
    Once again for $(\mathbb{T},\mathbb{M})$ a matching on $\mathbb{T}$, we can resample $\mathbb{M}_{\mathcal{H}}$ on top of $(\mathbb{T},\mathbb{M})$. We reuse the notations $n$ and $n_{\mathcal{H}}$ for the neighbour functions.
    Combining Lemma~\ref{lem:perfS}
    and Proposition~\ref{prop:russe}, we get
    \begin{align*}
    \mathbb{E}[  \left( \left( \mathbbm{1}_{o \neq n_{\mathcal{H}}(o)} - \mathbbm{1}_{o \neq n(o)}  \right) \mathbbm{1}_{n(o) \neq n_{\mathcal{H}}(o)}  ,\ \left( w(o,n_{\mathcal{H}}(o))+w(o,n(o))\right)\mathbbm{1}_{n(o) \neq n_{\mathcal{H}}(o)}   \right)] \overset{\lex}{\geq} 0 ,
    \end{align*}
    with equality if and only if $X'=(0,0)$ almost surely. By Lemma~\ref{lem:perfS}, this is equivalent to $n_{\mathcal{H}}(o)=n(o)=o$.

    But the expectation evaluates to exactly
    \begin{align*}
        \perf_V(\mathbb{M}_{\mathcal{H}})-\perf_V(\mathbb{M}).
    \end{align*}
    We have thus shown that $\mathbb{M}_{\mathcal{H}}$ is optimal, and that if $\mathbb{M}$ is also optimal, then one must have $X'=(0,0)$ and thus $n_{\mathcal{H}}(o)=n(o)=o$ which, by unimodularity, precisely means that both matchings share the same set of unmatched vertices.
\end{proof}

\subsection{Uniqueness}
In this section, we show that $(\mathbb{T},\mathbb{M}_{\mathcal{H}})$ is the unique (in law) optimal unimodular matching on $\mathbb{T}$.
\begin{prop}\label{prop:uniqueness}
    Let $\mathbb{M}$ be an optimal matching on $\mathbb{T}$ and let $(\mathbb{T},\mathbb{M},\mathbb{M}_{\mathcal{H}})$ be a coupling where $\mathbb{M}_{\mathcal{H}}$ is as in Proposition~\ref{prop:Zconstruction}.
    Then $\mathbb{M}=\mathbb{M}_{\mathcal{H}}$ almost surely.
\end{prop}

\begin{proof}
We will show that almost surely, $\mathbb{M}^s=\mathbb{M}_{\mathcal{H}}^s$ and deduce that almost surely, $\mathbb{M}=\mathbb{M}_{\mathcal{H}}$ through the projection that forgets self-loops. In the whole proof, let $A$ be the event when $n_{\mathcal{H}}(o)\neq n(o)$ and assume $\mathbb P(A) \neq 0$.

Since $\mathbb{M}$ is optimal, we know from the previous proof that $\mathbb{M}$ and $\mathbb{M}_{\mathcal{H}}$ share the same set of unmatched vertices. 
Recalling Equation~\eqref{eq:aldous4}, we get
\begin{align*}
&\left((j_{\mathcal{H}}^s,w_{\mathcal{H}}^s)(v_{-1},v_0)-(j_{\mathcal{H}}^s,w_{\mathcal{H}}^s)(v_{1},v_0) \right)\mathbbm{1}_{v_{-1} \neq v_1} \\
-&\left((i^s_{\mathcal{H}},Z^s_{\mathcal{H}})(v_{-2},v_{-1}) -(i^s_{\mathcal{H}},Z^s_{\mathcal{H}})(v_{0},v_{1})\right)\mathbbm{1}_{v_{-1} \neq v_1} \\
\overset{\lex}{\geq}&\left((j_{\mathcal{H}}^s,w_{\mathcal{H}}^s)(v_0,v_1)-(i^s_{\mathcal{H}},Z^s_{\mathcal{H}})(v_0,v_1)-(i^s_{\mathcal{H}},Z^s_{\mathcal{H}})(v_{-1},v_0)\right)\mathbbm{1}_{v_1 \neq v_{-1}} \\
\overset{\lex}{\geq} &\,0.
\end{align*}
The previous subsection also shows that
\[
\mathbb{E}[\left((j_{\mathcal{H}}^s,w_{\mathcal{H}}^s)(v_0,v_1)-(i^s_{\mathcal{H}},Z^s_{\mathcal{H}})(v_0,v_1)-(i^s_{\mathcal{H}},Z^s_{\mathcal{H}})(v_{-1},v_0)\right)\mathbbm{1}_{v_1 \neq v_{-1}}]
\]
exists and is equal to zero. Indeed, the expectation of the difference of performances is 
\[
\perfV(\mathbb{M}_{\mathcal{H}})-\perfV(\mathbb{M})=(0,0),
\]
and we deduce that $\left((j_{\mathcal{H}}^s,w_{\mathcal{H}}^s)(v_0,v_1)-(i^s_{\mathcal{H}},Z^s_{\mathcal{H}})(v_0,v_1)-(i^s_{\mathcal{H}},Z^s_{\mathcal{H}})(v_{-1},v_0)\right)\mathbbm{1}_{v_1 \neq v_{-1}}$ is integrable and that
\begin{align*}
\mathbb{E}\left[((j_{\mathcal{H}}^s,w_{\mathcal{H}}^s)(v_0,v_1)-(i^s_{\mathcal{H}},Z^s_{\mathcal{H}})(v_0,v_1)-(i^s_{\mathcal{H}},Z^s_{\mathcal{H}})(v_{-1},v_0))\mathbbm{1}_{v_1 \neq v_{-1}} \right]=(0,0) .
\end{align*}
Since Equation~\ref{eq:aldous2} implies that this variable is almost surely bigger than $(0,0)$, we deduce that it is equal to $(0,0)$ almost surely, hence that
\begin{align}\label{eq:Aldous5}
(i^s_{\mathcal{H}},Z^s_{\mathcal{H}})(v_{-1},o) =(j_{\mathcal{H}}^s,w_{\mathcal{H}}^s)(o,v_1)-(i^s_{\mathcal{H}},Z^s_{\mathcal{H}})(o,v_1).
\end{align}
By definition of $n^s_{\mathcal{H}}$,
\begin{align}\label{eq:Aldous6}
v_{-1}=n^s_{\mathcal{H}}(v_0) = \overset{\lex}{\argmax}_{y \sim v_0 }((j_{\mathcal{H}}^s,w_{\mathcal{H}}^s)(o,y)-(i^s_{\mathcal{H}},Z^s_{\mathcal{H}})(o,y)).
\end{align}
By the recursive equation of the messages: 
\begin{align}\label{eq:Aldous7}
    (i^s_{\mathcal{H}},Z^s_{\mathcal{H}})(v_{-1},o)=\maxlex_{\substack {y \sim o  y \neq v_{-1}}}((j_{\mathcal{H}}^s,w_{\mathcal{H}}^s)(o,y)-(i^s_{\mathcal{H}},Z^s_{\mathcal{H}})(o,y))    
\end{align}
Combining Equations~\eqref{eq:Aldous5} and \eqref{eq:Aldous7} then using Equation~\eqref{eq:Aldous6}, we get that $v_1$ achieves the maximum among the list of $(j_{\mathcal{H}}^s,w_{\mathcal{H}}^s)(o,y)-(i^s_{\mathcal{H}},Z^s_{\mathcal{H}})(o,y)$ stripped of its maximum. 
In other words, $v_1=n(o)$ achieves the second maximum of $(j_{\mathcal{H}}^s,w_{\mathcal{H}}^s)(o,y)-(i^s_{\mathcal{H}},Z^s_{\mathcal{H}})(o,y)$ when $y$ spans the neighbours of $o$. We will denote this quantity by ${\underset{y \sim o}{\argmaxxlex}}( (j_{\mathcal{H}}^s,w_{\mathcal{H}}^s)(o,y)-(i^s_{\mathcal{H}},Z^s_{\mathcal{H}})(o,y))$. 
Thus:
\[\mathbb{P}\left( n(o)=\underset{y \sim o}{\argmaxlex}((j_{\mathcal{H}}^s,w_{\mathcal{H}}^s)(o,y)-(i^s_{\mathcal{H}},Z^s_{\mathcal{H}})(o,y)) \text{ or } \underset{y \sim o}{\argmaxxlex}((j_{\mathcal{H}}^s,w_{\mathcal{H}}^s)(o,y)-(i^s_{\mathcal{H}},Z^s_{\mathcal{H}})(o,y))  \right)=1. \]
By unimodularity, this must be true for any vertex $u \in V(\mathbb{T})$. In particular, we obtain that
\[ n(u)\neq n_{\mathcal{H}}(u)  \implies n(u)= \bigg(\argmaxxlex_{v \sim u} (j_{\mathcal{H}}^s,w_{\mathcal{H}}^s)(u,v)-(i_{\mathcal{H}}^s,Z_{\mathcal{H}}^s)(u,v) \bigg) . \]
In addition, when an edge of the matching is selected according to the second largest value, this rule has to be involutive, meaning:
\begin{align} \label{eq:argmax2involutif}
    n(u)\neq n_{\mathcal{H}}(u)  \implies &\left\{  n(u)= \bigg(\argmaxxlex_{v \sim u} (j_{\mathcal{H}}^s,w_{\mathcal{H}}^s)(u,v)-(i_{\mathcal{H}}^s,Z_{\mathcal{H}}^s)(u,v) \bigg) \right. \notag \\
    & \qquad  \left. \text{ and } u= \bigg(\argmaxxlex_{v \sim n(u)} (j_{\mathcal{H}}^s,w_{\mathcal{H}}^s)(n(u),v)-(i_{\mathcal{H}}^s,Z_{\mathcal{H}}^s)(n(u),v) \bigg) \right\}.
\end{align}
The remainder of the argument will be based on the fact that it is impossible to use the second largest rule to define a matching in a unimodular fashion.

Let us set
\begin{align*}
\tilde{v}_0&:=o ,\\
\tilde{v}_{1}&:=\underset{y \sim o}{\argmaxxlex}((j_{\mathcal{H}}^s,w_{\mathcal{H}}^s)(o,y)-(i^s_{\mathcal{H}},Z^s_{\mathcal{H}})(o,y)) &&\text{if }o\text{ is not isolated} ,\\
\tilde{v}_{2}&:=\underset{y \sim \tilde{v}_1}{\argmaxlex}((j_{\mathcal{H}}^s,w_{\mathcal{H}}^s)(\tilde{v}_1,y)-(i^s_{\mathcal{H}},Z^s_{\mathcal{H}})(\tilde{v}_1,y)) &&\text{if }o\text{ is not isolated} ,\\
\tilde{v}_{3}&:=\underset{y \sim \tilde{v}_2}{\argmaxxlex}((j_{\mathcal{H}}^s,w_{\mathcal{H}}^s)(\tilde{v}_2,y)-(i^s_{\mathcal{H}},Z^s_{\mathcal{H}})(\tilde{v}_2,y))&&\text{if }o\text{ is not isolated} ,\\
\tilde{v}_{-1}&:=\underset{y \sim o}{\argmaxlex}((j_{\mathcal{H}}^s,w_{\mathcal{H}}^s)(o,y)-(i^s_{\mathcal{H}},Z^s_{\mathcal{H}})(o,y)) ,\\
\tilde{v}_{-2p}&:=\underset{y \sim w_{-2p+1}}{\argmaxxlex}((j_{\mathcal{H}}^s,w_{\mathcal{H}}^s)(\tilde{v}_{-2p+1},y)-(i^s_{\mathcal{H}},Z^s_{\mathcal{H}})(\tilde{v}_{-2p+1},y))  &&\forall p \in \mathbb{N}^*,\\
\tilde{v}_{-2p-1}&:=\underset{y \sim w_{-2p}}{\argmaxlex}((j_{\mathcal{H}}^s,w_{\mathcal{H}}^s)(\tilde{v}_{-2p},y)-(i^s_{\mathcal{H}},Z^s_{\mathcal{H}})(\tilde{v}_{-2p},y))  &&\forall k \in \mathbb{N}^* .
\end{align*}
See Figure~\ref{fig:infpath} for an illustration.
The advantage of the $\tilde{v}_p$ instead of the previously used $v_p$ is that they no longer depend on $\mathbb{M}$ which we do not have much control over. Instead, they depend on the messages directly.
\begin{figure}
    \centering
    \includegraphics[width=0.6\linewidth]{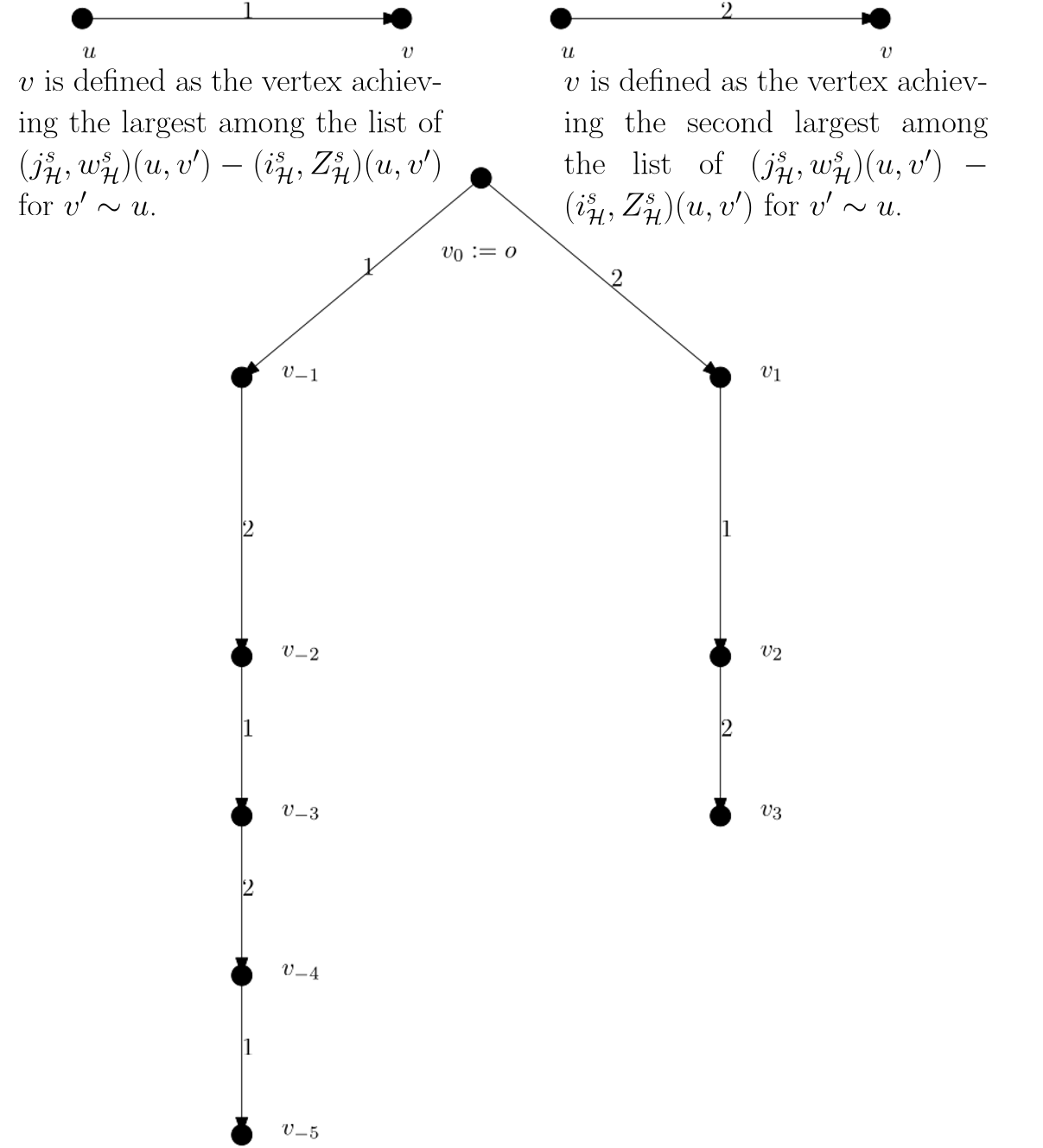}
    \caption{Definition of the path $(\tilde{v}_p)_{p\leq3}$.}
    \label{fig:infpath}
\end{figure}

We define the following events  for $p<0$ and $N < 0$
\begin{align}
A_{2p}^{(1)} &:= \left\{ \tilde{v}_{2p} \neq \tilde{v}_{2p+1} \neq \tilde{v}_{2p-1} \right\},\\
A_{2p}^{(2)} &:= \left\{  \tilde{v}_{2p+1} =\argmaxxlex_{u \sim \tilde{v}_{2p}}\bigg( (j_{\mathcal{H}}^s,w_{\mathcal{H}}^s)(\tilde{v}_{2p},u)- (i_{\mathcal{H}}^s,Z_{\mathcal{H}}^s)(\tilde{v}_{2p},u)\bigg) \right\},\\
A_{2p} &:= A_{2p}^{(1)} \cap A_{2p}^{(2)},\\
C_{N} &:=\bigcap_{N \leq 2p < 0} A_{2p} . 
\end{align}
The first event $A_{2p}^{(1)}$ corresponds to the fact that the portion of the path around $\tilde{v}_{2p}$ is injective and  $A_{2p}^{(2)}$ corresponds to \eqref{eq:argmax2involutif}. See figure~\ref{fig:uniquenessevent}.

\begin{figure}
    \centering
    \includegraphics[width=0.7\linewidth]{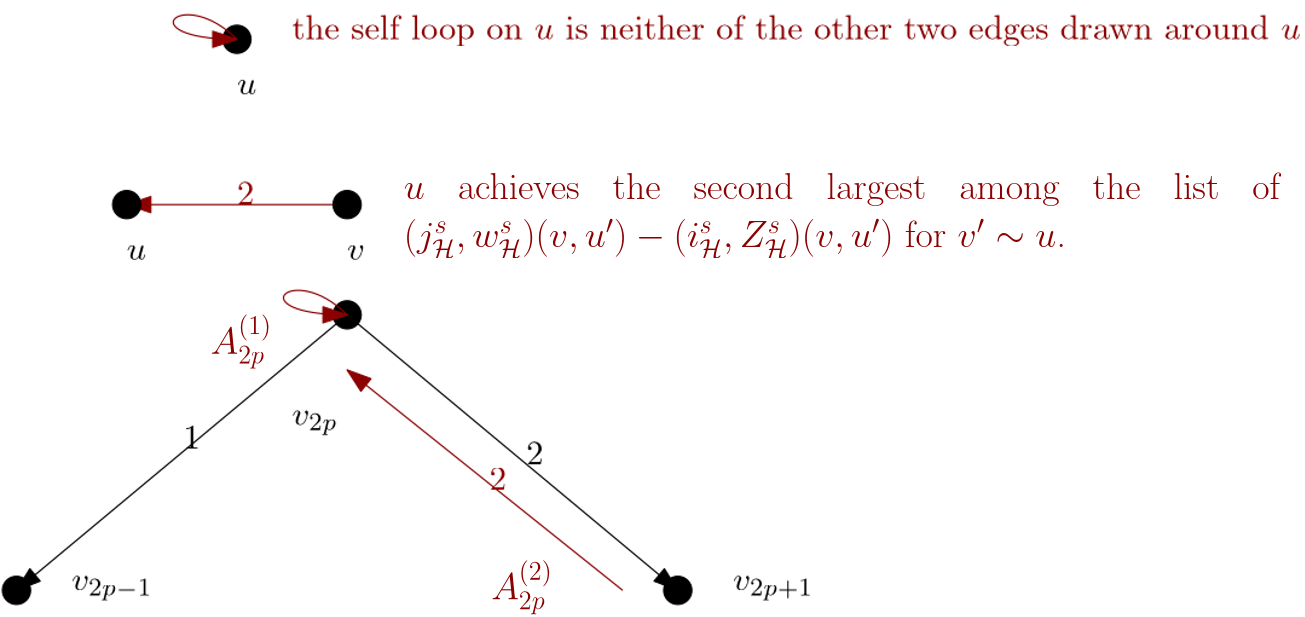}
    \caption{Definition of the event $A_{2p}$.}
    \label{fig:uniquenessevent}
\end{figure}

By construction, we have $(\tilde{v}_{-2p},\tilde{v}_{-2p-1})\in \mathbb{M}^s_{\mathcal{H}}$ for any $p\geq0$. Recall that we are interested in the event
\begin{equation}
    A := \{n_{\mathcal{H}}(o)\neq n(o)\}.
\end{equation}
Note that on the event $A$ we have
\[
\tilde{v}_{-2} = v_{-2}, \, \tilde{v}_{-1} = v_{-1}, \,\tilde{v}_{0} = v_0 \text{ and } \tilde{v}_1 = v_1.
\]
In addition we have
\[\mathbb{P}\left(\bigcup_{p \geq 0}\left\{ (\tilde{v}_{-2p-1},\tilde{v}_{-2p-2})\in \mathbb{M}\right\} \middle| A \right)=1. \]

Furthermore, by the second part of Proposition~\ref{prop:optimality}, almost surely, for any $v\in V$, $n_{\mathcal{H}}(v)=v$ if and only if $n(v)=v$. From this, we deduce that on $A$ we have $\tilde{v}_{-1}\neq \tilde{v}_0$ and, by induction, $  \forall k \in \mathbb{N}, \tilde{v}_{-k} \neq \tilde{v}_{-k-1}$ (no loops in the path) and also that $\forall k \in \mathbb{N}, \tilde{v}_{-k-2} \neq \tilde{v}_{-k}$ (the path cannot go back up). This holds similarly for $\tilde{v}_1$,$\tilde{v}_2$ and $\tilde{v}_3$. In short, we showed that $A$ is contained in the event where the path $(\tilde{v}_k)_{k\leq 3}$ is infinite and injective.

For the root and $\tilde{v}_2$, a direction reversal occurs, recall
\[ \tilde{v}_1=\argmaxxlex_{u \sim o} \bigg( (j_{\mathcal{H}}^s,w_{\mathcal{H}}^s)(o,u)- (i_{\mathcal{H}}^s,Z_{\mathcal{H}}^s)(o,u)\bigg)
.  \]
We will then set
\begin{align}
\tilde{A}_0 &:= \left\{ \tilde{v}_0=\argmaxxlex_{u \sim \tilde{v}_1} \bigg( (j_{\mathcal{H}}^s,w_{\mathcal{H}}^s)(\tilde{v}_1,u)- (i_{\mathcal{H}}^s,Z_{\mathcal{H}}^s)(\tilde{v}_1,u)\bigg), \tilde{v}_0\neq \tilde{v}_{-1}, \tilde{v}_ \neq \tilde{v}_1 
\right\} , \\
\tilde{A}_2 &:=\left\{ \tilde{v}_2=\argmaxxlex_{u \sim \tilde{v}_3} \bigg( (j_{\mathcal{H}}^s,w_{\mathcal{H}}^s)(\tilde{v}_3,u)- (i_{\mathcal{H}}^s,Z_{\mathcal{H}}^s)(\tilde{v}_3,u)\bigg) , \tilde{v}_2\neq \tilde{v}_1, \tilde{v}_3 \neq \tilde{v}_2 \right\} .
\end{align}

For every $N <0$ we have
\[ \mathbb{P}(A) \leq \mathbb{P}(\tilde A_0\cap \tilde A_2 \cap C_N) = \mathbb{P}\left( \tilde A_0 \cap \tilde A_2 \cap \bigcap_{N\leq 2p < 0} A_{2p} \right)  . \]
To conclude the argument, we want to show that the right hand side has probability zero when $N \rightarrow -\infty$. We will use the following Lemma, whose proof is postponed at the end of the section.
\begin{lemma}\label{lem:doublyinfinitepath}
Fix some $H>2$, then $\mathbb{P}(C_{-2H-2}\cap \tilde{A}_0)=\mathbb{P}(C_{-2H} \cap \tilde{A}_0 \cap \tilde{A}_2)$.
\end{lemma}

Now, if we assume that $ \lim_{N \rightarrow -\infty} \mathbb{P}(C_{N}\cap \tilde{A}_0 \cap \tilde{A_2})>0$, taking $H \rightarrow \infty$ in Lemma \ref{lem:doublyinfinitepath} yields
\begin{equation}\label{eq:infinitepath}
 \mathbb{P}\left(  \tilde{A}_{2}  \bigg| \tilde A_0 \cap \bigcap_{p' < 0} A_{2p'}  \right) =1.
\end{equation}

Assume that $\tilde{v}_2\neq \tilde{v}_3$. The only dependency between the messages and variables in the connected component $\mathbb{T}_{\tilde{v}_{3}}$ of $\tilde{v}_{3} \in \mathbb{T}\setminus \{\tilde{v}_{2} \}$ on the one hand, and its complement $\mathbb T_{\tilde{v}_{2}}$ on the other hand, is through the triplet
\[(i_{\mathcal{H}},Z_{\mathcal{H}})(\tilde{v}_{3},\tilde{v}_{2}),(i_{\mathcal{H}},Z_{\mathcal{H}})(\tilde{v}_{2},\tilde{v}_{3}) ,w(\tilde{v}_{3},\tilde{v}_{2}). \] 
From this we deduce that conditionally on the triplet
\[(i_{\mathcal{H}},Z_{\mathcal{H}})(\tilde{v}_{3},\tilde{v}_{2}),(i_{\mathcal{H}},Z_{\mathcal{H}})(\tilde{v}_{2},\tilde{v}_{3}) ,w(\tilde{v}_{3},\tilde{v}_{2}), \]
 the two decorated subtrees \[\left(\mathbb{T}_{\tilde{v}_{2}},\tilde{v}_{2},(i_{\mathcal{H}},Z_{\mathcal{H}})(u,v)_{(u,v)\in \overset{\rightarrow}{E}(\mathbb{T}_{\tilde{v}_{2}})}\right),\left(\mathbb{T}_{\tilde{v}_{3}},\tilde{v}_{3},(i_{\mathcal{H}},Z_{\mathcal{H}})(u,v)_{(u,v)\in \overset{\rightarrow}{E}(\mathbb{T}_{\tilde{v}_{3}})}\right) \] are independent. Furthermore, according to Equation~\eqref{eq:infinitepath}, conditionally on $\tilde{A_0}\cap \bigcap_{p'<0}A_{2p'}$, $\tilde{A}_{2}$ happens almost surely, which contains the event that $\tilde{v}_2 \neq \tilde{v}_3$.
 
As a consequence, we have that the law of
$\left(\mathbb{T}_{\tilde{v}_{3}},\tilde{v}_{3},(i_{\mathcal{H}},Z_{\mathcal{H}})(u,v)_{(u,v)\in \overset{\rightarrow}{E}(\mathbb{T}_{\tilde{v}_{3}})}\right)$
conditionally on $\tilde A_0 \cap \bigcap_{p' < 0}A_{2p'}$, which is measurable with respect to the other tree, can be condensed through the conditional distribution $\mathcal{P}$ of the triplet  \[(i_{\mathcal{H}},Z_{\mathcal{H}})(\tilde{v}_{3},\tilde{v}_{2}),(i_{\mathcal{H}},Z_{\mathcal{H}})(\tilde{v}_{2},\tilde{v}_{3}) ,w(\tilde{v}_{3},\tilde{v}_{2}) \] under $\tilde A_0 \cap \bigcap_{p' < 0}A_{2p'}$.

We will then show that no matter the values of this triplet, the event $\tilde A_{2}$ will not happen with probability $1$. We thus write that
\begin{align*}
    &\mathbb{P}\left(\overline{\tilde{A_{2}}} \bigg| \tilde A_0 \cap \bigcap_{p' < 0} A_{2p'}\right) \\
    &=\iiint \mathbb{P}\left(\overline{\tilde{A}_{2}} \bigg| \bigcap_{p' <0} A_{2p'}\cap \tilde{A}_0, (i_\mathcal{H},Z_{\mathcal{H}})(\tilde{v}_{2},\tilde{v}_{3})=(a,b), (i_\mathcal{H},Z_{\mathcal{H}})(\tilde{v}_{3},\tilde{v}_{2})=(a',b'), w(\tilde{v}_{2},\tilde{v}_{3})=\alpha  \right)\\
    &\qquad \qquad \qquad \mathrm{d}\mathcal{P}((a,b),(a',b'),\alpha) \\
    &=\iiint \mathbb{P}\left(\overline{\tilde{A}_{2}} \bigg| \tilde{v}_2 \neq \tilde{v}_3, (i_\mathcal{H},Z_{\mathcal{H}})(\tilde{v}_{2},\tilde{v}_{3})=(a,b), (i_\mathcal{H},Z_{\mathcal{H}})(\tilde{v}_{3},\tilde{v}_{2})=(a',b'), w(\tilde{v}_{2},\tilde{v}_{3})=\alpha  \right)\\
    &\qquad \qquad \qquad \mathrm{d}\mathcal{P}((a,b),(a',b'),\alpha )
\end{align*}
where $\mathcal{P}$ is the conditional law of the triplet $(i_{\mathcal{H}},Z_{\mathcal{H}})(\tilde{v}_{3},\tilde{v}_{2}),(i_{\mathcal{H}},Z_{\mathcal{H}})(\tilde{v}_{2},\tilde{v}_{3}) ,w(\tilde{v}_{3},\tilde{v}_{2})$ under $\tilde{A}_0\cap\bigcap_{p' < 0} A_{2p'}$. By definition, we must have that the support of $\mathcal{P}$ is included in the space $(a,b)+(a',b') \overset{\lex}{\geq }(k,\alpha)$ since $(\tilde{v}_{3},\tilde{v}_{2}) \notin \mathbb{M}_h^s$. As equality almost surely does not occur, we almost surely have $(a,b)+(a',b') \overset{\lex}{> }(k,\alpha)$ . The next lemma concludes the proof by showing that being in this space ensures that this conditional density is always positive.
\begin{lemma}\label{lem:Mhstop}
    For every $(a,b),(a',b') \in \mathrm{supp}(\zeta_{h}')$, $ \alpha \in \mathrm{supp} (\omega)$ such that $(a,b)+(a',b')\overset{\lex}{>}(k,\alpha)$, \[ \mathbb{P} \left( \overline{\tilde A_{2}} \middle| \tilde{v}_2 \neq \tilde{v}_3 , \,(i_\mathcal{H},Z_{\mathcal{H}})(\tilde{v}_{2},\tilde{v}_{3})=(a,b), (i_\mathcal{H},Z_{\mathcal{H}})(\tilde{v}_{3},\tilde{v}_{2})=(a',b'), w(\tilde{v}_{2},\tilde{v}_{3})=\alpha \right)    >0   . \]
\end{lemma}
As a conclusion, this shows a contradiction with the assumption that  \[  \mathbb{P}\left(  \bigcap_{p'<0}A_{2p'} \cap \tilde{A_0}\cap \tilde{A_2} \right)>0.\]
We deduce that this probability has to be zero, but since we showed that this probability is bigger than $\mathbb{P}(A)$, we get that $\mathbb{P}(A)=0$, which translates to
\[ \mathbb{P}(n(o)=n_{\mathcal{H}}(o))=1.  \]
By unimodularity, we deduce that almost surely, for all $v \in V(\mathbb{T})$, $n(v)=n_{\mathcal{H}}(v)$, hence $\mathbb{M}=\mathbb{M}_{\mathcal{H}}$.

\end{proof}

\bigskip

We now turn to the proof of our two Lemmas.

\begin{proof}[Proof of Lemma \ref{lem:doublyinfinitepath}]
   Refer to Figure~\ref{fig:uniquenesstransport} for an illustration.

   \begin{figure}
       \centering
       \includegraphics[width=0.6\linewidth]{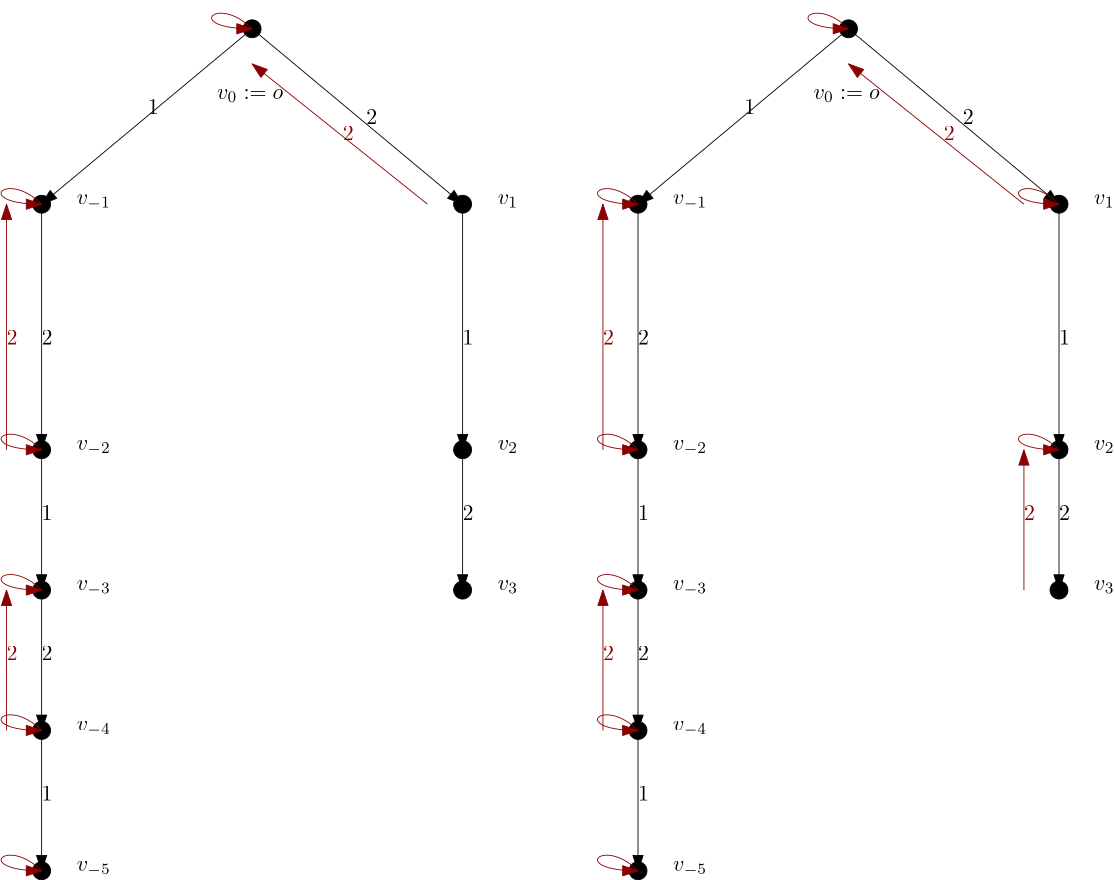}
       \caption{The events on the left (drawn in red) down to depth $2H+2$ have the same probability as the events on the right up to depth $2H$ by rerooting from $\tilde{v}_0$ to $\tilde{v}_2$.}
       \label{fig:uniquenesstransport}
   \end{figure}
    Define the alternating path starting from a vertex $u$ as:
    \begin{align*}
     u_{1}(u)&:=\argmaxxlex_{y \sim u}((j_{\mathcal{H}}^s,w_{\mathcal{H}}^s)(u,y)-(i^s_{\mathcal{H}},Z^s_{\mathcal{H}})(u,y)) \text{ if } u \text{ is not isolated} , \\
     u_{-1}(u)&:=\argmaxlex_{y \sim u}((j_{\mathcal{H}}^s,w_{\mathcal{H}}^s)(u,y)-(i^s_{\mathcal{H}},Z^s_{\mathcal{H}})(u,y)), \\
     u_{-2}(u)&:=u_1(u_{-1}(u))=\argmaxxlex_{y \sim u_{-1}(u)}((j_{\mathcal{H}}^s,w_{\mathcal{H}}^s)(u_{-1}(u),y)-(i^s_{\mathcal{H}},Z^s_{\mathcal{H}})(u_{-1}(u),y)),
     \end{align*}
    and for $n \geq 1,$
    \begin{align*}
    u_{-2n-1}(u)&:=u_{-1}(u_{-2n}(u))=\argmaxlex_{y \sim u_{-2n}(u)}((j_{\mathcal{H}}^s,w_{\mathcal{H}}^s)(u_{-2n}(u),y)-(i^s_{\mathcal{H}},Z^s_{\mathcal{H}})(u_{-2n}(u),y)), \\ 
    u_{-2n-2}(u)&:=u_1(u_{-2n-1}(u))=\argmaxxlex_{y \sim u_{-2n-1}(u)}((j_{\mathcal{H}}^s,w_{\mathcal{H}}^s)(u_{-2n-1}(u),y)-(i^s_{\mathcal{H}},Z^s_{\mathcal{H}})(u_{-2n-1}(u),y)).
    \end{align*}
    By convention, set $u_0=\text{Id}_{V}$, in other words, $u_0(u)=u$.
    For any vertex $u$, define the event \[A(u) := \left\{u=\argmaxxlex_{y \sim u_{1}(u)}((j_{\mathcal{H}}^s,w_{\mathcal{H}}^s)(u_{1}(u),y)-(i^s_{\mathcal{H}},Z^s_{\mathcal{H}})(u_{1}(u),y)), u \text{ is not isolated}, u_1(u) \neq u, u_{-1}(u)\neq u \right\}.
      \]
    In particular, we have that $A_{-2p}=A(u_{-2p}(o))$ for any $p>1$.
    
    For an outcome of decorated tree $T$ of $(\mathbb{T},(i_\mathcal{H},Z_\mathcal{H}))$ , $H\in \mathbb{N}$ and $a,b \in V(T)$, define the event:
    \[ \tilde{C}_H(T,(i^s_{\mathcal{H}},Z^s_{\mathcal{H}}),a,b) :=  \left(  \bigcap_{0 \geq n \geq 2H} A(u_{-n}(b)) \right) \cap \left( \bigcap_{0 \geq n \geq 2H} A(u_{-n}(a)) \right) .\]
    Let us define the mass transport function 
    \[f_{H}(T,M,(i^s_{\mathcal{H}},Z^s_{\mathcal{H}}),a,b):=\mathbbm{1}_{\tilde{C}_H(T,(i^s_{\mathcal{H}},Z^s_{\mathcal{H}}),a,b) \cap \{b=u_{-2}(a)\}  }\].
    
    On the one hand, a simple computation yields
    \begin{align*}
     \mathbb{E}\left[\sum_{v \in V}f_{H}(T,M,(i^s_{\mathcal{H}},Z^s_{\mathcal{H}}),o,v)\right] &=\mathbb{P}\left( \bigcap_{0 \leq n \leq 2H} A(u_{-n}(o)) \cap  \bigcap_{0 \leq n \leq 2H} A(u_{-n}(u_{-2}(o)))  \right) \\ 
     &= \mathbb{P}\left( \bigcap_{0 \leq n \leq 2H+2 }A(u_{-n}(o))  \right)\\
     &= \mathbb{P}\left(C_{-2H-2} \cap \tilde{A}_0\right).
     \end{align*}
     
     On the other hand, the other expectation is more delicate:
    \begin{align*}
     \mathbb{E}\left[\sum_{v \in V}f_{H}(T,M,(i^s_{\mathcal{H}},Z^s_{\mathcal{H}}),v,o)\right] = &\mathbb{E}\left[ \sum_{v': u_{-2}(v')=o} \mathbbm{1}_{ \bigcap_{0 \leq n \leq 2H} A(u_{-n}(v')) \cap \bigcap_{0 \leq n \leq 2H} A(u_{-n}( u_{-2}(v')))} \right] \\
     \\
     =&\mathbb{E}\left[ \sum_{v': u_{-2}(v')=o} \mathbbm{1}_{ \bigcap_{0 \leq n \leq 2H} A(u_{-n}(v')) \cap \bigcap_{0 \leq n \leq 2H} A(u_{-n}( o))} \right] \\
     =&\mathbb{E}\left[ \sum_{v': u_{-2}(v')=o} \mathbbm{1}_{ A(v') \cap A(u_{-1}(v')) \cap \bigcap_{0 \leq n \leq 2H} A(u_{-n}( o))} \right].   
     \end{align*}
     It is not \emph{a priori} obvious that there exists a unique $v'\in V$ such that $u_{-2}(v')=o$, in fact, this is not true in general.
     Nevertheless, we will show that being in $A(o)\cap A(u_{-1}(v'))$ will automatically imply that $v'=u_{-1}(u_{1}(o)))=\tilde{v}_2$.
    To proceed, since the event $\bigcap_{0 \leq n \leq 2H} A(u_{-n}( o))$ appearing in the sum does not depend on the choice of $v'$, we rewrite using a conditional expectation:
    \begin{align*}
        &\mathbb{E}\left[ \sum_{v': u_{-2}(v')=o} \mathbbm{1}_{ A(v') \cap A(u_{-1}(v')) \cap \bigcap_{0 \leq n \leq 2H} A(u_{-n}( o))} \right]\\
        =&\mathbb{E}\left[ \sum_{v': u_{-2}(v')=o} \mathbbm{1}_{ A(v') \cap A(u_{-1}(v')) } \bigg| \bigcap_{0 \leq n \leq 2H} A(u_{-n}( o)) \right] \mathbb{P}\left(\bigcap_{0 \leq n \leq 2H} A(u_{-n}( o))\right)
    \end{align*}
    The conditioning event contains $A(o)$, under which $\tilde{v}_2$ exists and $u_{-2}(\tilde{v}_2)=\tilde{v}_0$. This guarantees that the summation set of $v'$ is not empty. By definition, we have that $\tilde{v}_1=u_1(o)$.
    Take $v'$ such that $u_{-2}(v')=o$, we want to show that only $v'=\tilde{v}_2$ contributes. We have that 
    \begin{align*}
    A(u_{-1}(v'))=\bigg\{&u_{-1}(v')=\argmaxxlex_{y \sim u_{1}(u_{-1}(v'))}((j_{\mathcal{H}}^s,w_{\mathcal{H}}^s)(u_{1}(u_{-1}(v')),y)-(i^s_{\mathcal{H}},Z^s_{\mathcal{H}})(u_{1}(u_{-1}(v')),y)),\\
    &u_{1}(u_{-1}(v')) \neq u_{-1}(v'), v' \neq u_{-1}(v') \bigg\} \\
    =\bigg\{&u_{-1}(v')=\argmaxxlex_{y \sim o}((j_{\mathcal{H}}^s,w_{\mathcal{H}}^s)(o,y)-(i^s_{\mathcal{H}},Z^s_{\mathcal{H}})(o,y)), o \neq u_{-1}(v') ,v'\neq \tilde{v}_{-1}(v') \bigg\} \\
    \subseteq\bigg\{&u_{-1}(v')=u_1(o)=\tilde{v}_1  \bigg\} \\
    =\bigg\{&v'=u_{-1}(\tilde{v}_1)=\tilde{v}_2\bigg\}.
    \end{align*}
    where we simply injected the expression of $o$ and used that $u_{-1}$ is an involution.
    This shows that any other $v'\neq \tilde{v}_2$ cannot satisfy the condition, hence:
    \begin{align*}
    \mathbb{E}\left[ \sum_{v': u_{-2}(v')=o} \mathbbm{1}_{ A(v') \cap A(u_{-1}(v')) } \bigg| \bigcap_{0 \leq n \leq 2H} A(u_{-n}( o)) \right] 
    &= \mathbb{E} \left[    \mathbbm{1}_{ A(\tilde{v}_2) \cap A(\tilde{v}_1) } \bigg| \bigcap_{0 \leq n \leq 2H} A(u_{-n}( o))     \right] \\
    &= \mathbb{P} \left( A(\tilde{v}_2)\cap A(\tilde{v}_1) \bigg|  \bigcap_{0 \leq n \leq 2H} A(u_{-n}( o))      \right).
    \end{align*}
    Putting everything together, we have shown that:
    \begin{align*}
    \mathbb{E}\left[\sum_{v \in V}f_{H}(T,M,(i^s_{\mathcal{H}},Z^s_{\mathcal{H}}),o,v)\right] &=\mathbb{P} \left( A(\tilde{v}_2)\cap A(\tilde{v}_1) \bigg|  \bigcap_{0 \leq n \leq 2H} A(u_{-n}( o))      \right)\mathbb{P} \left(   \bigcap_{0 \leq n \leq 2H} A(u_{-n}( o)) \right) \\
    &=\mathbb{P} \left( A(\tilde{v}_2)\cap A(\tilde{v}_1) \cap  \bigcap_{0 \leq n \leq 2H} A(u_{-n}( o))      \right) \\
    &= \mathbb{P}( \tilde{A}_2 \cap \tilde{A}_0 \cap C_{-2H-2}).
    \end{align*}
    Applying mass-transport principle to $f_{H}$ yields equality between the two expectations, hence the result.
\end{proof}

\begin{proof}[Proof of Lemma \ref{lem:Mhstop}]
Let $N_1:=\deg(\tilde{v}_{3})-1$, write $u_j$ to be the neighbours of $\tilde{v}_{3}$ in the tree that are not $\tilde{v}_{2}$, $(i_j, Z_j)$ to be $(i_{\mathcal{H}},Z_{\mathcal{H}})(\tilde{v}_{3},u_j)$ and $w_j=w(\tilde{v}_{3},u_j)$. For an illustration of the notations and the overall idea of this proof, see Figure~\ref{fig:aldouserie}.
\begin{figure}
    \centering
    \includegraphics[width=0.6\linewidth]{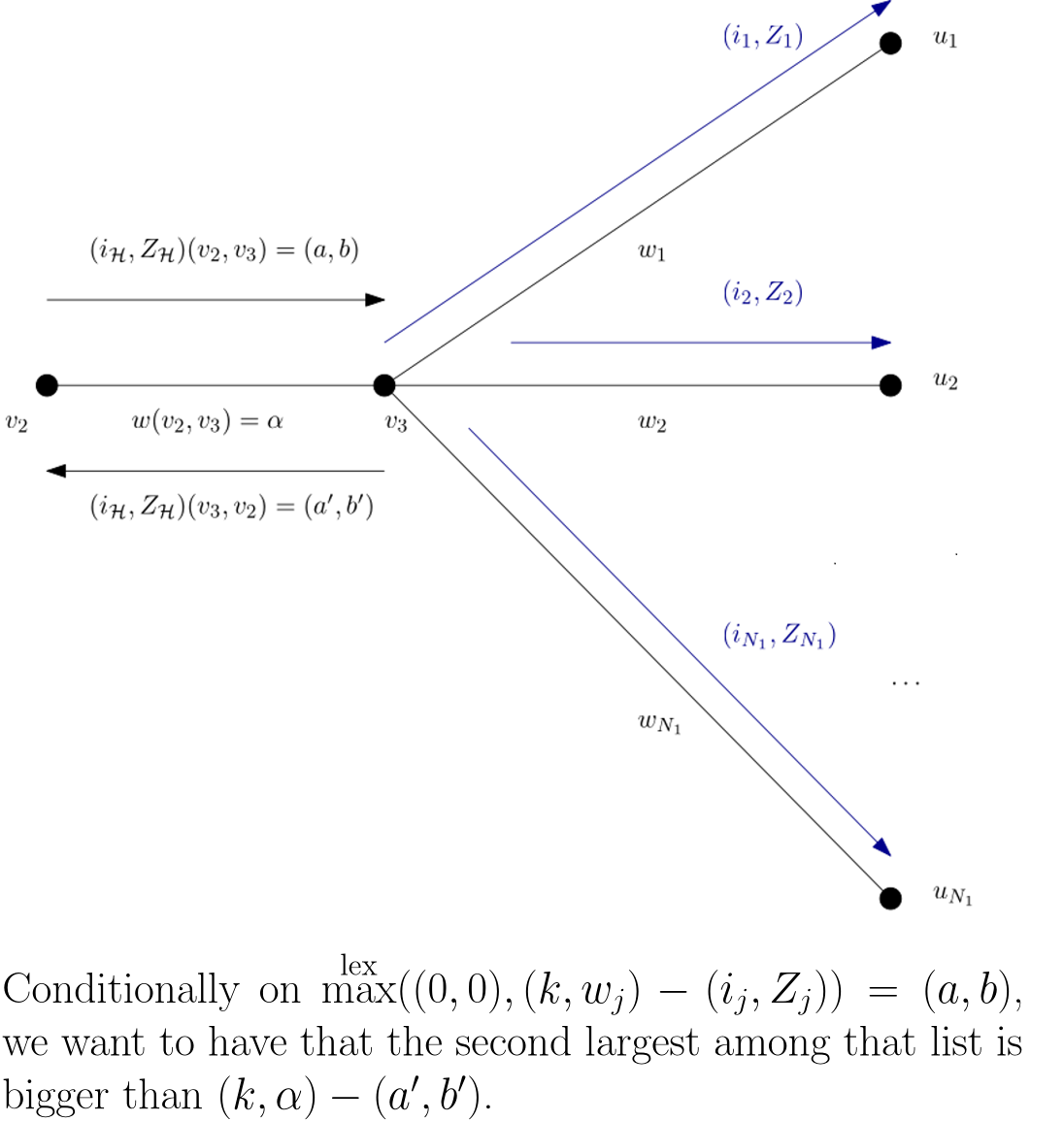}
    \caption{The event $\overline{\tilde{A}_2}$ is equivalent to the condition on the bottom.}
    \label{fig:aldouserie}
\end{figure}

On the event
\[ \tilde{v}_2\neq \tilde{v}_3, (i_{\mathcal{H}},Z_{\mathcal{H}})(\tilde{v}_{2},\tilde{v}_{3})=(a,b),(i_{\mathcal{H}},Z_{\mathcal{H}})(\tilde{v}_{3},\tilde{v}_{2})=(a',b'),w(\tilde{v}_{2},\tilde{v}_{3})=\alpha ,  \]
the event $ \overline{\tilde{A}_2}$ can be rewritten as
\[ \tilde{v}_{2} \neq u_1(\tilde{v}_{3})= \argmaxxlex_{y \sim \tilde{v}_{3}} \left( (j_{\mathcal{H}}^s,w_{\mathcal{H}}^s)(\tilde{v}_{3},j)-(i_{\mathcal{H}}^s,Z_{\mathcal{H}}^s)(\tilde{v}_{3},j) \right).    \]

From the recursion, the conditioning $(i_{\mathcal{H}},Z_{\mathcal{H}})(\tilde{v}_{2},\tilde{v}_{3})=(a,b)$ is equivalent to
\[ \maxlex( (0,0), \maxlex_{1\leq j \leq N_1} (k,w_j)-(i_j,Z_j))=(a,b)   .\]
Let us write $\maxxlex$ to be the operator that returns the second largest value in lexicographic order. Then the event $\overline{\tilde{A}_2}$ can be rewritten as \[ \maxxlex_{1 \leq j \leq N_1} (k, w_j)-(i_j,Z_j) \overset{\lex}{>} (k,\alpha)-(a',b') \text{ or } (k,\alpha)-(a',b') \overset{\lex}{<}(0,0)\overset{\lex}{<}(a,b).\]
Indeed, if $(k,\alpha)-(a',b')\overset{\lex}{>}(0,0)$, then it is clear. The event $(k,\alpha)-(a',b')=(0,0)$ happens with probability zero so it is irrelevant. Finally, if $(k,\alpha)-(a',b')\overset{\lex}{<}(0,0)\overset{\lex}{<}(a,b)$, then either $\maxxlex_{1 \leq j \leq N_1} (k, w_j)-(i_j,Z_j)\overset{\lex}{<}(0,0)$ and $u_1(\tilde{v}_3)=\tilde{v}_3$ or $\maxxlex_{1 \leq j \leq N_1} (k, w_j)-(i_j,Z_j)\overset{\lex}{>}(0,0)$ and $u_1(\tilde{v}_3)=\argmaxxlex_{1 \leq j \leq N_1} (k, w_j)-(i_j,Z_j)$.

Additionally, $(k,\alpha)-(a,b)\overset{\lex}{<} (0,0) \overset{\lex}{<} (a',b')=\max_{u \sim \tilde{v}_2}((j_{\mathcal{H}}^s,w_{\mathcal{H}}^s)(\tilde{v}_2,u)-(i_{\mathcal{H}}^s,Z_{\mathcal{H}}^s)(\tilde{v_2},u)$ would imply that $n_1(\tilde{v}_2)=\tilde{v}_2$, so it is incompatible with the given conditioning.

Putting everything together, the event \[\overline{\tilde{A}_2} \cap  \big\{\tilde{v}_2\neq \tilde{v}_3, (i_{\mathcal{H}},Z_{\mathcal{H}})(\tilde{v}_{2},\tilde{v}_{3})=(a,b),(i_{\mathcal{H}},Z_{\mathcal{H}})(\tilde{v}_{3},\tilde{v}_{2})=(a',b'),w(\tilde{v}_{2},\tilde{v}_{3})=\alpha \bigg\}  \]  can be rewritten as
\begin{align*}
\left\{ \maxxlex_{1 \leq j \leq N_1} (k, w_j)-(i_j,Z_j) \overset{\lex}{>} (k,\alpha)-(a',b') \right\} \cap
\left\{ (a,b)= \maxlex( (0,0), \maxlex_{1\leq j \leq N_1}\left( (k,w_j)-(i_j,Z_j) \right)
\right\}
\end{align*}
where $(i_j,Z_j),w_j,N_1$ are independent variables, each $(i_j,Z_j)$ has distribution $\zeta_h'$, each $w_j$ has distribution $\omega$ and $N_1$ has distribution $\hat{\pi}$.

To end the proof, we just need to check that the support of $(k,w_1)-(i_1, Z_1)$ intersects the lexicographic interval $\bigg((k,\alpha)-(a',b'),(a,b) \bigg)_{\lex}$. The support of $(k,w_1)-(i_1,Z_1)$ is precisely the set of differences between elements in the support of $\{0,\dots,k\} \times \mathrm{supp}(\omega)$ and of the support of $\zeta_{\mathcal{H}}'$. The quantity $(k,\alpha)-(a',b')$ also lies in the same support. We deduce that it is sufficient to show that the support of $\zeta_h'$ does not contain holes, in the sense that there does not exist pairs $(c,d)\overset{\lex}{<}(c',d')$ in $\mathrm{supp}(\zeta_h')$ such that $\bigg( (c,d),(c',d')\bigg)_{\lex} \cap \mathrm{supp}(\zeta_h') = \emptyset$. By focusing on the second component, if $c\neq 0$ or $k$, this can be seen from iterating the recursive equation which shows that the second component of the support of $\zeta_h'$ is additively stable with the span of alternating sums of the form $w^{\star}_1-w_2^{\star}+w_3^{\star}-w_4^{\star} + \cdots$ where each $w^{\star} \in \mathrm{supp}(\omega)$. A similar argument holds for $c=0$ but the support is restricted to $\mathbb{R}_+$ in this case. Since we chose $\mathrm{supp}(\omega)$ to contain some open interval, these alternating sums span $\mathbb{R}$, which concludes.
\end{proof}

\subsection{Proofs of Theorem~\ref{maintheorem} and of Theorem~\ref{maintheorem2}}\label{sec:mainresults}

In this section, we combine previous results in order to prove Theorems~\ref{maintheorem} and~\ref{maintheorem2}.
\begin{proof}[Proof of Theorem~\ref{maintheorem}]
The Theorem follows from directly combining Proposition~\ref{prop:optimality} and Proposition~\ref{prop:uniqueness} which respectively states that the matching distribution constructed from the messages is optimal, then that sampling the messages make the optimal matching almost surely unique.
\end{proof}

\begin{proof}[Proof of Theorem~\ref{maintheorem2}]
Applying \cite[Lemma~4.2]{enriquez2024optimalunimodularmatching},
we know that the sequence $(G_n,o_n,\mathbb{M}_{\mathrm{opt}})$ is tight for local convergence.
Take any subsequence $k_n$ such that $(G_{k_n},o_{k_n},\mathbb{M}_{\mathrm{opt}}(G_{k_n}))$ converges locally to some unimodular matched graph $(\mathbb{T},o,\mathbb{M})$, since size and performance are all expectations of $1-$local functions, we have that:
\begin{align*}
    \lim_{n \rightarrow \infty} \mathbb{P}(o_n \in \mathbb{M}_{\mathrm{opt}}(G_{k_n})))&=\mathbb{P}(o\in \mathbb{M}), \\
    \forall x>0, \lim_{n \rightarrow \infty} \mathbb{E}[w_n(o_n)\mathbbm{1}_{w_n(o_n)\leq x, o_n \in M_{\mathrm{opt}(G_n)}}] &= \mathbb{E}[w(o)\mathbbm{1}_{w(o)\leq x, o \in M_\mathrm{opt}(G_{k_n})}]. \\    
\end{align*}
Since $\perfE(\mathbb{T},\mathbb{M}_{\mathcal{H}})$ is lexicographically optimal, we have that the vector
\[  \left(\mathbb{P}(o \in \mathbb{M}), \mathbb{E}[w(o)\mathbbm{1}_{ o \in \mathbb{M})}] \right) \]
is lexicographically bounded by $\perfE(\mathbb{T},\mathbb{M}_{\mathcal{H}})$.
By monotone convergence, we have 
\begin{align*}
 \lim_{x \rightarrow +\infty} \mathbb{E}\left[w_n(o_n)\mathbbm{1}_{w_n(o_n)\leq x, o_n \in M_{\mathrm{opt}(G_n)}}\right] &=\mathbb{E}\left[(w_n(o_n)\mathbbm{1}_{o_n \in M_{\mathrm{opt}}(G_n)}\right] \\
 \lim_{x \rightarrow +\infty} \mathbb{E}\left[w(o)\mathbbm{1}_{w(o)\leq x, o_n \in \mathbb{M}}\right], &=\mathbb{E}\left[(w(o_n)\mathbbm{1}_{o \in \mathbb{M}}\right]. 
\end{align*}
Therefore, to conclude, we just need to justify that
\[ \lim_{x \rightarrow +\infty}\lim_{n \rightarrow \infty} \mathbb{E}[w_n(o_n)\mathbbm{1}_{w_n(o_n)\leq x, o_n \in M_{\mathrm{opt}(G_n)}}] =\lim_{n \rightarrow +\infty}\lim_{x \rightarrow \infty} \mathbb{E}[w_n(o_n)\mathbbm{1}_{w_n(o_n)\leq x, o_n \in M_{\mathrm{opt}(G_n)}}] \]
To this end, we bound the random variable
$w_n(o_n)\mathbbm{1}_{w_n(o_n)\leq x, o_n \in M_{\mathrm{opt}(G_n)}}$ by $w_n(o_n)$, which is uniformly integrable by hypothesis, and the definition of uniformly integrability allows us to switch the order of the limits.

Finally, to find a subsequence that converges locally to $\mathbb{M}_{\mathcal{H}}$, we apply the proof of \cite[Section~4]{enriquez2024optimalunimodularmatching}  and the corresponding matching law $\mathbb{M}_{\mathcal{H}}$ on $\mathbb{T}$. We will build approximations of $\mathbb{M}_{\mathcal{H}}$ on the finite graph $G_{n}$ with increasingly small error and arbitrarily close size and performance of $\mathbb{M}_{\mathcal{H}}$.
We use the same proof as \cite[Proposition~4.2]{enriquez2024optimalunimodularmatching}, except that we apply \cite[Proposition~4.10]{enriquez2024optimalunimodularmatching} local functions 
\begin{align*}
    f(T,M)&=\mathbbm{1}_{o \in M} ,\\
    g(T,M)&=w(o)\mathbbm{1}_{o \in M}
\end{align*}
that ensure asymptotic convergence of both size and performance.
This allows us to construct matchings $\tilde{M}_n$ on $G_n$ that have asymptotically optimal size and weight. But since there is a unique matching $\mathbb{M}_{\mathcal{H}}$ on the limit, that maximises first the size and then the weights, we deduce that the only Benjamini-Schramm subsequent limit of $(G_n,\tilde{M}_n)$ is $(\mathbb{T},o,w,\mathbb{M}_{\mathcal{H}})$, hence convergence.
\end{proof}

\section{The law of the messages}\label{sec:marginali}
In this section, we prove that under some technical conditions the distribution of the macroscopic marginal $i$ is explicit, and that the distribution of $(i,Z)$ is unique. We start by stating the theorem. Recall that $\{0,\dots,k\}$ is the support of $i$.
\begin{theorem}\label{uniqueness}
    Consider the map 
    \[
    F_\pi: x \in [0,1] \mapsto \phi(1-x)+\phi(1-\hat{\phi}(1-x))+\phi'(1)x\hat{\phi}(1-x)
    \]
    and suppose that $\argmax F_\pi$ has at most two elements.
    Then the distribution of $i$ is described by:
    \begin{enumerate}
        \item If $\argmax F_\pi$ is reduced to a single point $\gamma$, then $k=1$ and the distribution of $i$ is $\gamma \delta_0 + (1-\gamma) \delta_1$.
        \item If $\argmax F_\pi$ has two distinct points $0<\underline{\gamma}< \overline{\gamma}<1$, then $k=2$, and  the distribution of $i$ is $\underline{\gamma}\delta_0+(\overline{\gamma}-\underline{\gamma})\delta_1+(1-\overline{\gamma})\delta_2$.
        \item If $\argmax F_\pi$ is not as in one the previous two items, then it is equal to $\{0,1\}$, in which case $k=0$ and $i=0$ almost surely.
    \end{enumerate}
    Furthermore, let $b=\sup( \mathrm{supp}(\omega))$, assume that there exists $a<b$ such that the interval $(a,b)$ is included in $\mathrm{supp}(\omega)$, then there is a unique distribution $\zeta'$ solution to the system of equations~\eqref{eq:systemzoomh0singular}-\eqref{eq:systemzoomh1/2singular}.
\end{theorem}

\subsection{Proof of the shape of the macroscopic messages in Theorem~\ref{uniqueness}}
In this subsection we prove that every solution to the system~\eqref{eq:system} is of the form stated by Theorem~\ref{uniqueness}.

Let $\mathcal{H}(k,h_{0},....,h_{k})$ be any solution to the system~\eqref{eq:system} and define  for $0<j\leq k$ \[l_{\mathcal{H},j} := \mathbb{P}(i_{\mathcal{H}}<j)=\lim_{t \rightarrow -\infty}h_j(t) .\]
Note that $\hat{\phi}(1-l_{\mathcal{H},k-j+1})=l_{\mathcal{H},j}$. We will show the following proposition in the next subsection, which states that for every $j$, the value of $F_\pi(l_{\mathcal{H},j})$ does not change:

\begin{prop}\label{prop:bordsgeneral}
The quantity $F_\pi(l_{\mathcal{H},j})$ does not depend on $j$.
\end{prop}

Next, we will show that any $F_\pi$ is maximal at the points $l_{\mathcal{H},j}$.  To this end, we will show the following Proposition~\ref{coro:formulasize}:

\begin{restatable}{prop}{formulasize}[Fixed point size formula]\label{coro:formulasize}
    We have the following formula linking the size of optimal maximal matchings $S_E(\pi)$ to $l_{\mathcal{H},1}$ and $l_{\mathcal{H},k}$:
    \begin{equation}\label{eq:formulasize}
    \begin{aligned}
        S_E(\pi)&=\frac{2-F_\pi(l_{\mathcal{H},1})}{\phi'(1)}.
    \end{aligned}
    \end{equation}
    \end{restatable}

Combining our Theorem~\ref{maintheorem2} and Theorem 1 of C. Bordenave, M. Lelarge and J. Salez, \cite{salez2011cavity}, we know that

\[
S_E(\pi) = \frac{2-\max_{[0,1]}F_\pi}{\phi'(1)}.
\]
Therefore, for all $0<j\leq k$ we have $l_{\mathcal{H},j} \in \argmax F_\pi$ giving the first part of Theorem~\ref{uniqueness}.

\subsection{Proofs of Proposition~\ref{prop:bordsgeneral} and Proposition~\ref{prop:chgmtpdv}}
In this section we prove Propositions~\ref{prop:bordsgeneral} and~\ref{coro:formulasize}. We start with three useful lemmas.

\begin{restatable}{lemma}{lemmetaille}[Formula for the size of maximum matchings]\label{lem:taille}
    For  any solution ${\mathcal{H}}$ to the system~\eqref{eq:system}, set $\beta_{\mathcal{H}}:=\mathbb{P}_{(i_{\mathcal{H}},Z_{\mathcal{H}})\sim \zeta_{\mathcal{H}}'}((i_{\mathcal{H}},Z_{\mathcal{H}})=(0,0))$.
 Then $\beta:=\beta_{\mathcal{H}}$ does not depend on the choice of $\mathcal{H}$ and is linked to the size of optimal unimodular matchings on $\mathbb{T}$ by:
    \begin{equation}\label{eq:sizezeroatom}
        S_E(\pi):= \mathbb{P}(\overset{\rightarrow}{o}\in \mathbb{M}_{\mathrm{opt}}) = \frac{1-\phi(\hat{\phi}^{-1}(\beta))}{\hat{\phi}'(1)}= \beta(1-\hat{\phi}^{-1}(\beta)) + \int_{\beta}^{1}(1-\hat{\phi}^{-1}(u))\mathrm{d}u.
    \end{equation}
    In particular, $\beta=0$ corresponds to perfect matchings and can only occur when $\hat{\phi}(0)=0$, which is the case where the tree has no leaves.
\end{restatable}

\begin{restatable}{lemma}{atomzero}[Conservation equations]\label{lem:atomzero}

    For any solution $\mathcal{H}=(k,(h_0,h_{1},....,h_{k-1},h_k))$ to the system~\eqref{eq:system},  recall the boundary condition $l_{\mathcal{H},j}$ of $h_j$:
    \begin{align*}
        \forall 0 < j \leq k, \quad l_{\mathcal{H},j}:= \lim_{t \rightarrow -\infty} h_j(t)=\mathbb{P}(i_\mathcal{H}<j).
    \end{align*}
    Then the boundary conditions $l_{\mathcal{H},j}$ satisfy the following conservation equations:
    \begin{align}
    \beta (1-\hat{\phi}^{-1}(\beta))+\int_{\beta}^{l_{\mathcal{H},1}}(1-\hat{\phi}^{-1}(u))\mathrm{d}u &=l_{\mathcal{H},k}l_{\mathcal{H},1} + \int_{l_{\mathcal{H},k}}^{1}(1-\hat{\phi}^{-1}(u))\mathrm{d}u,\label{eq:atomzerobords} \\
        l_{\mathcal{H},j}l_{\mathcal{H},k-j+1}+ \int_{l_{\mathcal{H},j}}^{l_{\mathcal{H},j+1}}(1-\hat{\phi}^{-1}(u))\mathrm{d}u &= l_{\mathcal{H},j+1}l_{\mathcal{H},k-j} + \int_{l_{\mathcal{H},k-j}}^{l_{\mathcal{H},k-j+1}}(1-\hat{\phi}^{-1}(u))\mathrm{d}u. \label{eq:energysave}
    \end{align}
\end{restatable}

\begin{proof}[Proof of Lemmas~\ref{lem:taille}, \ref{lem:atomzero}]
    The idea behind all three lemmas is to introduce a mixing integral between $h_{j}$ and $h_{k-j}$ and use the functional equation to express the integral either with the boundary conditions of $h_{j}$ or the boundary conditions of $h_{k-j}$.

    Let $Z=(Z_0,....Z_k) \sim \zeta_{\mathcal{H}}$ and $Z'=(Z_0',....,Z_k') \sim \zeta_{\mathcal{H}}$ such that $Z$ and $Z'$ are independent, let $W \sim \omega$ that is independent from $(Z,Z')$.
    Define for $0\leq j \leq k$ the following probability:
    \begin{equation}\label{eq:integralemagique}
        I_j=\mathbb{P}(Z_j+Z_{k-j}'<W)
    \end{equation}
    with the convention $+\infty + (-\infty)=+\infty$.
    Clearly, we have that $I_j=I_{k-j}$. Now, let us express $I_j$ with respect to $\lim_{t \rightarrow \pm \infty } h_j(t) $.
   We start with $I_0$ and $I_k$, and integrate with respect to the value of $Z_0$:
    \begin{align*}
     I_0&= \int_{\mathbb{R}}\mathbb{P}(Z_{k}'<W-x)\mathrm{d}\mathbb{P}_{Z_0}(x).   
    \end{align*}
    From Equation~\eqref{eq:systemzoomhzero}, we see that 
    \begin{align*}
    \mathbb{P}(Z_0\leq x)&=\mathbbm{1}_{x \geq 0} \sum_{l=0}^{+\infty}\hat{\pi}_l \mathbb{P}(W-Z_k\leq x)^{l} \\
    &=\mathbbm{1}_{x \geq 0} \hat{\phi}(\mathbb{P}(Z_k \geq W-x)) \\
    &=\mathbbm{1}_{x \geq 0} \hat{\phi}(1-\mathbb{P}(Z_k < W-x)).
    \end{align*}
    Hence \[1-\mathbb{P}(Z_{k}'<W-x)= \hat{\phi}^{-1}(\mathbb{P}(Z_0 \leq x)).\]
    Define $\mu_0$ the push-forward measure of $\mathbb{P}_{Z_0}$ by $\mathbb{P}(Z_0 \leq  \cdot)=h_0$, for any measurable subset $A$ of $\mathbb{R}$. $\mu_0(A)=\mathbb{P}(Z_0 \in \{x: \, \mathbb{P}(Z_0 \leq  x) \in A\})$ . We deduce that:
    \begin{align*}
        I_0&=\int_{\mathbb{R}}(1-\hat{\phi}^{-1})(\mathbb{P}(Z_0\leq x))\mathrm{d}\mathbb{P}_{Z_0}(x) \\
        &= \int_{\mathbb{R}}(1-\hat{\phi}^{-1})(x)\mathrm{d}\mu_0(x).
    \end{align*}
    Everywhere where $\mathbb{P}_{Z_0}$ is continuous, we simply get $\mu_0(\mathrm{d}x)=\mathrm{d}x$. Since there is an atom of mass $\beta$ at $0$, we end up with:
    \begin{align*}
        I_0=\int_{\beta}^{l_{\mathcal{H},1}}(1-\hat{\phi}^{-1})(x)\mathrm{d}x + \beta(1-\hat{\phi}^{-1}(\beta)).
    \end{align*}

    Now let us do a similar computation for $I_k$ where we integrate with respect to the value of $Z_k$:
    \begin{align*}
        I_k=\int_{\mathbb{R}}\mathbb{P}(Z_0'<W-x)\mathrm{d}\mathbb{P}_{Z_k}(x) + \mathbb{P}(Z_k=-\infty, Z_0< \infty) .
    \end{align*}
    Once again from Equation~\eqref{eq:systemzoomh}, we see that:
    \begin{align*}
        \mathbb{P}(Z_k<x)&= \sum_{m=0}^{+\infty}\hat{\pi}_l\mathbb{P}(Z_0'<W-x)^m =\hat{\phi}(1-\mathbb{P}(Z_0'<W-x)).
    \end{align*}
    Hence:
    \begin{align*}
    1-\mathbb{P}(Z_0'<W-x)=\hat{\phi}^{-1}(\mathbb{P}(Z_k \leq x)).
    \end{align*}
    Define $\mu_k$ the push-forward measure of $\mathbb{P}_{Z_k}$ by $\mathbb{P}(Z_k\leq \cdot)= h_k$, then:
    \begin{align*}
        I_k&=\int_{\mathbb{R}}(1-\hat{\phi}^{-1})(\mathbb{P}(Z_k\leq x))\mathrm{d}\mathbb{P}_{Z_k}(x) + \mathbb{P}(Z_k=-\infty)\mathbb{P}(Z_0<\infty) \\
        &=\int_{\mathbb{R}}(1-\hat{\phi}^{-1})(x)d\mu_k(x) + l_{\mathcal{H},k}l_{\mathcal{H},1}.
    \end{align*}
    Similarly as above, $\mathrm{d}\mu_k(x)=\mathrm{d}x \mathbbm{1}_{l_{\mathcal{H},k}<x<1}$, hence:
    \begin{align*}
        I_k=\int_{l_{\mathcal{H},k}}^{1}(1-\hat{\phi}^{-1})(x)\mathrm{d}x + l_{\mathcal{H},1}l_{\mathcal{H},k}.
    \end{align*}
    The equality $I_0=I_k$ which is true by symmetry yields Lemma~\ref{lem:atomzero}.

\bigskip
    
    Now let us do the same calculations for $I_j$ and $I_{k-j}$ for $0<j<k$.
    We have that:
    \begin{align*}
        I_j&=\int_{\mathbb{R}}\mathbb{P}(Z_{k-j}'<W-x)\mathrm{d}\mathbb{P}_{Z_j}(x)+\mathbb{P}(Z_j=-\infty,Z_{k-j}<\infty) \\
        &=\int_{l_{\mathcal{H},j}}^{l_{\mathcal{H},j+1}}(1-\hat{\phi}^{-1})(x)\mathrm{d}x + l_{\mathcal{H},j}l_{\mathcal{H},k-j+1},
    \end{align*}
    by repeating the same arguments as before. Writing the equality $I_j=I_{k-j}$ yields:
    \begin{align*}
        \int_{l_{\mathcal{H},j}}^{l_{\mathcal{H},j+1}}(1-\hat{\phi}^{-1})(x)\mathrm{d}x + l_{\mathcal{H},j}l_{\mathcal{H},k-j+1}=\int_{l_{\mathcal{H},k-j}}^{l_{\mathcal{H},k-j+1}}(1-\hat{\phi}^{-1})(x)\mathrm{d}x + l_{\mathcal{H},k-j}l_{\mathcal{H},j+1}
    \end{align*}
    which is exactly Equation~\eqref{eq:energysave}.

\bigskip
    
    Finally, the size of the matching can be expressed as follows:
    \begin{align*}
    S_E(\pi,{\mathcal{H}})=\mathbb{P}(\exists  j\in\llbracket 0, k\rrbracket: \, Z_j+Z_{k-j}' < W )  
    \end{align*}
    with the convention $+\infty + (-\infty)=+\infty$.
    We will compute this probability by decomposing with respect to the index of $Z_j$:
    \begin{align*}
        S_E(\pi,{\mathcal{H}})&=\mathbb{P}\left( \bigcup_{0 \leq j \leq k} \left\{ Z_j+Z_{k-j}'<W , -\infty < Z_j < +\infty   \right\}  \right) \\
        &=\sum_{j=0}^{k} \mathbb{P}\left( Z_j+Z_{k-j}'<W, -\infty < Z_j < +\infty  \right).
    \end{align*}
    If one does not pay attention, one could be tempted to write $S_E(\pi,{\mathcal{H}})=\underset{j=0}{\overset{k}{\sum}} I_j$, but in $I_j$ we do not forbid $Z_i=-\infty$, so the computation will be slightly different.
    In particular, $I_j=\mathbb{P}(Z_j+Z_{k-j}'<W, Z_j<+\infty)$ so 
    \begin{align*}
        I_j=\mathbb{P}\left( Z_i+Z_{k-i}'<W, -\infty < Z_i < +\infty  \right) +\mathbb{P}(Z_i=-\infty, Z_{k-i}' < \infty).
    \end{align*}
    Since $\mathbb{P}(Z_j=-\infty,Z_{k-j}'<\infty)=l_j l_{k-j+1}$ if $i\neq 0$ and $0$ if $j=0$ injecting the expression previously found yields:
    \begin{align*}
        S_E(\pi,{\mathcal{H}})&=\sum_{i=j}^{k}(I_j-l_j l_{k-j+1}) + I_0 \\
        &= \sum_{j=1}^{k}\int_{l_{\mathcal{H},j}}^{l_{\mathcal{H},j+1}}(1-\hat{\phi}^{-1})(x)\mathrm{d}x +  \int_{\beta}^{l_{\mathcal{H},1}}(1-\hat{\phi}^{-1})(x)\mathrm{d}x + \beta(1-\hat{\phi}^{-1}(\beta)) \\
        &=\beta(1-\hat{\phi}^{-1}(\beta)) + \int_{\beta}^{1}(1-\hat{\phi}^{-1}(x)\mathrm{d}x .
    \end{align*}
    This is precisely the result of Lemma~\ref{lem:taille}.
\end{proof}
\begin{proof}[Proof of Proposition~\ref{coro:formulasize}]
Combining the first Equation~\ref{eq:sizezeroatom} and Equation~\ref{eq:atomzerobords} yields:

\begin{align*}
S_E(\pi)&=\beta(1-\hat{\phi}^{-1}(\beta)) + \int_{\beta}^{1}(1-\hat{\phi}^{-1}(u))\mathrm{d}u \\
&= l_{\mathcal{H},k}l_{\mathcal{H},1} + \int_{l_{\mathcal{H},k}}^{1}(1-\hat{\phi}^{-1}(u))\mathrm{d}u.+\int_{l_{\mathcal{H},1}}^1 (1-\hat{\phi}^{-1}(u))\mathrm{d}u.
\end{align*}
Define $\mathcal{F}: u \mapsto \hat{\phi}(1-u)$. We have that $1-\hat{\phi}^{-1}(u)=\mathcal{F}^{-1}(u)$ and that $l_{\mathcal{H},1}=\mathcal{F}(l_{\mathcal{H},k})=\mathcal{F}^{-1}(l_{\mathcal{H},k})$.
Using integral of reciprocal functions, we get that
\begin{align*}
    S_E(\pi)&=l_{\mathcal{H},k}l_{\mathcal{H},1} + \int_{\mathcal{F}(l_{\mathcal{H},1})}^{\mathcal{F}(0)} \mathcal{F}^{-1}(u)\mathrm{d}u+\int_{\mathcal{F}(l_{\mathcal{H},k})}^{\mathcal{F}(0)} \mathcal{F}^{-1}(u)\mathrm{d}u \\
    &= l_{\mathcal{H},k}l_{\mathcal{H},1}  - \int_{l_{\mathcal{H},1}}^0 \mathcal{F}(u)\mathrm{d}u - \int_{l_{\mathcal{H},k}}^{0}\mathcal{F}(u)\mathrm{d}u-l_{\mathcal{H},1}\mathcal{F}(l_{\mathcal{H},1})-l_{\mathcal{H},k}\mathcal{F}(l_{\mathcal{H},k}).
\end{align*}
As $\mathcal{F}(u)=\frac{\phi'(1-u)}{\phi'(1)}$ , a primitive of $\mathcal{F}$ is $\mathcal{G}(u)=-\frac{\phi(1-u)}{\phi'(1)}$, hence
\begin{align*}
    S_E(\pi)&= -l_{\mathcal{H},k}l_{\mathcal{H},1} -2\mathcal{G}(0)+\mathcal{G}(l_{\mathcal{H},1})+\mathcal{G}(l_{\mathcal{H},k}) \\
    &=\frac{2-\phi(1-l_{\mathcal{H},1})-\phi(1-l_{\mathcal{H},k})}{\phi'(1)}-l_{\mathcal{H},1}l_{\mathcal{H},k} \\
    &= \frac{2-F_\pi(l_{\mathcal{H},1})}{\phi'(1)}
\end{align*}
giving the Proposition.
\end{proof}

\begin{proof}[Proof of Proposition~\ref{prop:bordsgeneral}]
    Recall Equation~\eqref{eq:energysave}. For all $j=1,..., k$, 
     \begin{equation*}
        l_{\mathcal{H},j}l_{\mathcal{H},k-j+1}+ \int_{l_{\mathcal{H},j}}^{l_{\mathcal{H},j+1}}(1-\hat{\phi}^{-1}(u))\mathrm{d}u = l_{\mathcal{H},j+1}l_{\mathcal{H},k-j} + \int_{l_{\mathcal{H},k-j}}^{l_{\mathcal{H},k-j+1}}(1-\hat{\phi}^{-1}(u))\mathrm{d}u. 
    \end{equation*}
    Similarly as before, recall that for any $ 0 <j\leq k$, $l_{\mathcal{H},j}=\mathcal{F}(l_{\mathcal{H},k-j+1})=\mathcal{F}^{-1}(l_{\mathcal{H},k-j+1})$. Using integral of reciprocal functions, we get that:
    \begin{align*}
    \int_{l_{\mathcal{H},j}}^{l_{\mathcal{H},j+1}}\mathcal{F}^{-1}(u)\mathrm{d}u &= \int_{\mathcal{F}(l_{\mathcal{H},k-j+1})}^{\mathcal{F}(l_{\mathcal{H},k-j})} \mathcal{F}^{-1}(u)\mathrm{d}u \\
    &=-\int_{l_{\mathcal{H},k-j+1}}^{l_{\mathcal{H},k-j}}\mathcal{F}(u)\mathrm{d}u+l_{\mathcal{H},j+1}l_{\mathcal{H},k-j}-l_{\mathcal{H},j}l_{\mathcal{H},k-j+1}.
    \end{align*}
    Recall that a primitive of $\mathcal{F}$ is $\mathcal{G}(u)=-\frac{\phi(1-u)}{\phi'(1)}$ , so we have that:
    \begin{align*}
    \int_{l_{\mathcal{H},j}}^{l_{\mathcal{H},j+1}}\mathcal{F}^{-1}(u)\mathrm{d}u &= \left[ \frac{\phi(1-u)}{\phi'(1)} \right]_{l_{\mathcal{H},k-j+1}}^{l_{\mathcal{H},k-j}} +l_{\mathcal{H},j+1}l_{\mathcal{H},k-j}-l_{\mathcal{H},j}l_{\mathcal{H},k-j+1} \\
    &= \frac{\phi(1-l_{\mathcal{H},k-j})-\phi(1-l_{\mathcal{H},k-j+1})}{\phi'(1)}+l_{\mathcal{H},j+1}l_{\mathcal{H},k-j}-l_{\mathcal{H},j}l_{\mathcal{H},k-j+1}. 
    \end{align*}
Evaluating this equality at $j$ and $k-j$ and injecting this into the first expression, we obtain:
\begin{align*}
    \frac{\phi(1-l_{\mathcal{H},k-j})-\phi(1-l_{\mathcal{H},k-j+1})}{\phi'(1)} + l_{\mathcal{H},j+1}l_{\mathcal{H},k-j} = \frac{\phi(1-l_{\mathcal{H},j})-\phi(1-l_{\mathcal{H},j+1})}{\phi'(1)} +l_{\mathcal{H},k-j+1}l_{\mathcal{H},j}.
\end{align*}
Hence,
\begin{align*}
    \frac{\phi(1-l_{\mathcal{H},k-j})+\phi(1-l_{\mathcal{H},j+1})}{\phi'(1)} + l_{\mathcal{H},j+1}l_{\mathcal{H},k-j} = \frac{\phi(1-l_{\mathcal{H},j})+\phi(1-l_{\mathcal{H},k-j+1})}{\phi'(1)} +l_{\mathcal{H},k-j+1}l_{\mathcal{H},j}.
\end{align*}
This is precisely the statement that the desired quantity is preserved when changing $j$ by $j+1$. The conclusion follows by induction.
\end{proof}

\subsection{Proof of uniqueness in Theorem~\ref{uniqueness}}

    We will treat the case $k=2$, we will leave out the case $k=1,0$ as the proof is nearly identical.
    In the remainder of the proof, we will thus assume that all solutions ${\mathcal{H}}$ are of the form $(2,(h_0,h_1,h_2))$. 
    
    Our aim is to find a family of statistics on the matchings that will suffice to reconstruct the law $\zeta_{\mathcal{H}}$ from $\mathbb{M}_{\mathcal{H}}$. We will then use the uniqueness of $\mathbb{M}_{\mathcal{H}}$  given by Theorem~\ref{maintheorem} to conclude. 
    The first part of the proof follows nearly verbatim from the proof of Theorem~2 in \cite{enriquez2024optimalunimodularmatching}, we will reproduce the main arguments of that proof here for sake of completeness.
    
    Assume $\hat{\pi}_1>0$ for simplicity (the case $\hat{\pi}_1=0$ will be discussed later). Let $b=\mathrm{supp}(\omega)$ and fix $a$ such that $(a,b) \subset \mathrm{supp}(\omega)$. We will condition on the event that the weight of the root edge is $w_0\in (a,b)$ and that the $+$ side of the root is a \textbf{simple path} $v_1,v_2,....,v_H$ of length $H$ and with weights $w_1,....,w_H \in (a,b)$.
    We will compute the probability that $v_H$ is unmatched in the matching $\mathbb{M}_{\mathcal{H}}$. The total event of conditioning on the sequence of weights and $v_H$ being unmatched is measurable with respect to the matching and graph. Figure \ref{fig:huniqueness} gives a depiction of the situation.

    \begin{figure}
        \centering
        \includegraphics[width=0.7\linewidth]{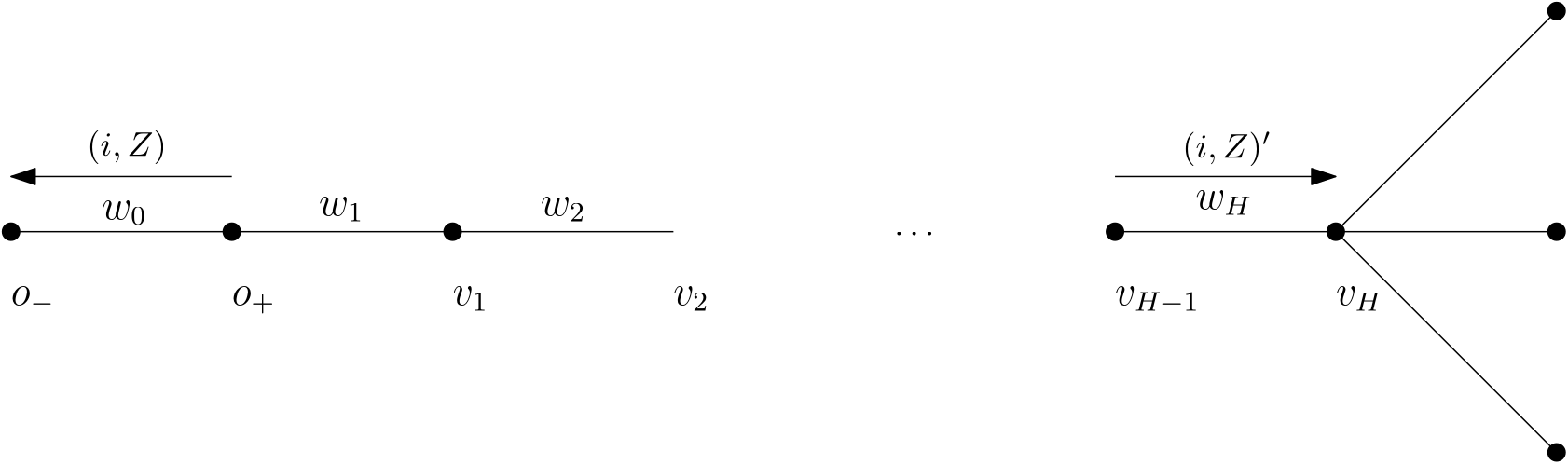}
        \caption{Illustration of the notations and conditioning in the proof. The first event considered is that $v_{H}$ is unmatched. The second event considered is that $v_{H}$ is not matched to $v_{H-1}$.}
        \label{fig:huniqueness}
    \end{figure}
    
    Now writing $(i,Z)=(i_{\mathcal{H}},Z_{\mathcal{H}})(o_+,o_-)$ and $(i,Z)'=(i_{\mathcal{H}},Z_{\mathcal{H}})(v_{H-1},v_H)$, $v_H$ being unmatched is equivalent to:
    \begin{align*}
        &\left\{ \maxlex_{u \sim v_H}((2,w)(v_H,u)-(i_{\mathcal{H}},Z_{\mathcal{H}})(v_H,u)) \overset{\lex}{\leq} (0,0)         \right\} \\
        & \qquad =\left\{ (i_{\mathcal{H}},Z_{\mathcal{H}})'=\maxlex\left( (0,0),\maxlex_{u \sim v_H, u \neq v_{H-1}}((2,w(v_H,u))-(i_{\mathcal{H}},Z_{\mathcal{H}})(v_H,u))\right)=(0,0) , \right. \\
        & \qquad \qquad \left.  (2,w(v_H,v_{H-1}))-(i_{\mathcal{H}},Z_{\mathcal{H}})(v_H,v_{H-1}) \overset{\lex}{\leq} (0,0)  \vphantom{\maxlex(0,\max_{u \sim v_H, u \neq v_{H-1}}} \right\} \\
        & \qquad =\left\{ (i,Z)'=(0,0), (2,w_H) \overset{\lex}{\leq} (i_{\mathcal{H}},Z_{\mathcal{H}})(v_H,v_{H-1})   \right\}.
    \end{align*}
    The variables $(i,Z)$ and $(i,Z)(v_{H},v_{H-1})$ are independent. Furthermore we already know from Lemma~\ref{lem:taille} that $\mathbb{P}((i_{\mathcal{H}},Z_{\mathcal{H}})=(0,0))$ does not depend on the solution ${\mathcal{H}}$ to \eqref{eq:system}. We deduce that the following quantity as a function of $w_0,...,w_H$ is invariant over ${\mathcal{H}}$ solutions to \eqref{eq:system}:
    \begin{align*}
        p(w_0,...,w_H) &:=\mathbb{P}((2,w_H) \overset{\lex}{\leq} (i_{\mathcal{H}},Z_{\mathcal{H}})(v_H,v_{H-1})) \\
        &=\mathbb{P}\left( (2,w_H) \overset{\lex}{\leq} \maxlex\left((0,0),(2,w_{H-1})-\maxlex\left((0,0),(2,w_{H-2})- \vphantom{\maxlex} \right. \right. \right.\\
        &\left.\left.\left. \qquad \cdots -\maxlex\left((0,0),(2,w_1)-\maxlex\left((0,0),(2,w_0)-(i,Z)\right) \right)... \right)  \right)    \right).
    \end{align*}

Fix $x_0 \in \mathbb{R}_+$, we will only consider $H$ even.
Fix a sequence $(w_0,\ldots,w_H) \in (a,b)^{H+1}$ such that the sequence $\underset{l=0}{\overset{2j}{\sum}}(-1)^l w_l$ is decreasing in $j$.
Furthermore we can force $\underset{l=0}{\overset{H}{\sum}}(-1)^l w_{H}=x_0$ and that for every $j<H$, $\underset{j=l}{\overset{H}{\sum}}(-1)^l w_l \neq 0$. For a proof of this fact, see \cite[Lemma 5.1.]{enriquez2024optimalunimodularmatching}.

Let us discuss what happens to the nested maximums depending on the values of $(i,Z)$:

\begin{enumerate}
    \item If $i=0$, since $H$ is even, then we end up with a condition of the form $(2,w_H) \overset{\lex}{\leq} (0,u)$ for $u$ some function of $w_0,...,w_H,Z$ which cannot be satisfied.
    \item If $i=1$, then the successive indexes stay at $1$ so the inequality cannot be satisfied.
    \item Finally, if $i=2$, then we end up with a condition of the form $(2,w_H) \overset{\lex}{\leq} (2,u)$ for $u$ some function of $w_0,.....,w_H,Z$ which may or may not be satisfied.
\end{enumerate}

In the third case, the nested maximums will never evaluate to $(0,0)$ on the odd terms since it will be of the form $\maxlex((0,0),(2,x))$ for some $x \in \mathbb{R}$. 
Let us consider the function
\begin{align*}
    f(w_0,.....,w_H,y)=\mathbbm{1}_{w_H \leq w_{H-1}-\max(0, w_{H-2}-w_{H-3}-\max(0,... \max(0,w_4-w_3+ \max(0,w_2 -w_1+\max(0,w_0-y)))...))}
\end{align*}
such that from the previous discussion, the quantity $p$ rewrites as:
\begin{align*}
    p(w_0,...,w_H)=\int_{\mathbb{R}} f(w_0,...,w_H,u) \mathrm{d}h_2(u),
\end{align*}
indeed, the conditioning does not affect the law of $Z$ at the root here.

We will now show three facts:
\begin{enumerate}
    \item If $u\leq  \underset{l=0}{\overset{H}{\sum}}(-1)^l w_l$, then $f(w_0,...,w_H,u)=1$.
    \item If $\underset{l=0}{\overset{H}{\sum}}(-1)^l w_l \leq  u < \underset{l=0}{\overset{H-2}{\sum}}(-1)^l w_l$, then $f(w_0,...,w_H,u)=0$.
    \item On every other interval, if we modify $w_H$ by $\eps< \underset{0<j<H}{\min}(\underset{l=j}{\overset{H}{\sum}}(-1)^l w_l)$, then the value of $f(w_0,....,w_H,u)$ does not change.
\end{enumerate}
For both 1. and 2. , we have that for every $0\leq j \leq \frac{H}{2}-1$, 
\begin{align*}
    u \leq \underset{l=0}{\overset{2j}{\sum}}(-1)^l w_l.
\end{align*}
For $j=0$ this is exactly when the first maximum evaluates to $w_0-u$. Repeating this argument, we find that the nested maximum evaluates to 
\begin{align*}
    \underset{l=0}{\overset{H-2}{\sum}}(-1)^l w_l-u
\end{align*}
such that $f(w_0,...,w_H,u)=1$ becomes equivalent to:
\begin{align*}
    \left\{ w_H \leq w_{H-1}-\underset{l=0}{\overset{H-2}{\sum}}(-1)^l w_l+u \right\}
    \Leftrightarrow  \left\{ u \geq \underset{l=0}{\overset{H}{\sum}}(-1)^l w_l \right\}.
\end{align*}
    The opposite $f(w_0,...,w_H,u)=0$ just reverses this condition. So $1.$ and $2.$ are verified.

    If $u$ does not verify these inequalities, the condition for $f=1$ is of the form
\begin{align*}
    w_H<w_{H-1} - \sum_{l=l_0}^{H-2}(-1)^lw_l 
\end{align*}
    where $l_0$ depends on $u,w_0,.....,w_{H-2}$ but not on $w_H$. In any case, we do not modify this inequality by modifying $w_H$ by the prescribed $\eps$ so 3. is also verified.
    
    Combining everything, we get the following equality:
\begin{align*}
    p(w_0,....,w_{H-1},w_H)-p(w_0,....,w_{H-1},w_H+\eps) = \mathbb{P}\left( (2,x_0) \overset{\lex}{\leq} (i,Z) \overset{\lex}{<} (2,x_0+\eps) \right).
\end{align*}
Since the quantity above must be independent of ${\mathcal{H}}$, the function $h_2$ is unique. Since $h_0$ can be expressed as a function of $h_2$, the function $h_0$ is also unique.

To get access to $i=1$, we will now reuse the conditioning above and compute the probability that $(v_{H-1},v_H) \notin \mathbb{M}_{\mathcal{H}}$ instead. Let us assume $H$ even for the sake of simplicity and note that $l=\underset{t \rightarrow \infty}{h_0(t)}$, $l'=\underset{t \rightarrow -\infty}{h_2(t)}$  are invariant over ${\mathcal{H}}$. This yields the following invariant over ${\mathcal{H}}$:
\begin{align*}
    p'(w_0,....,w_H) &:=\mathbb{P}\left( (2,w_H)\overset{\lex}{\leq} (i_{\mathcal{H}},Z_{\mathcal{H}})(v_H,v_{H-1})+(i_{\mathcal{H}},Z_{\mathcal{H}})(v_{H-1},v_H) \right) \\
    &=\mathbb{P}\left( (2,w_H) \overset{\lex}{\leq} \maxlex\left((0,0),(2,w_{H-1})-\maxlex\left((0,0),(2,w_{H-2})- \vphantom{\maxlex} \right. \right. \right.\\
        & \quad \left.\left.\left. \cdots -\maxlex\left((0,0),(2,w_1)-\maxlex\left((0,0),(2,w_0)-(i,Z)\right) \right)... \right)  \right)  +(i,Z)'  \right).
\end{align*}
We will decompose this probability with respect to the value of $(i,Z)'$. For $(x,y) \in \mathbb{R}^2$, let us define 
\begin{align*}
F(x,y) & := \mathbb{P}\left( (x,y) \overset{\lex}{\leq} \maxlex\left((0,0),(2,w_{H-1})-\maxlex\left((0,0),(2,w_{H-2})- \vphantom{\maxlex} \right. \right. \right.\\
        & \qquad \qquad \left.\left.\left. \cdots -\maxlex\left((0,0),(2,w_1)-\maxlex\left((0,0),(2,w_0)-(i,Z)\right) \right)... \right)  \right)\right).
\end{align*}
The function $F$ implicitly depends on $w_0,...,w_{H-1}$. First decomposing with respect to the value of $i'$ where $(i,Z)'=(i',Z')$ we get
\begin{align*}
    p'(w_0,...,w_H)&= \int_{\mathbb{R}}F(2,w_H-y) \mathrm{d}h_0(y)
    +\int_{\mathbb{R}}F(1,w_H-y)\mathrm{d}h_1(y)
    +\int_{\mathbb{R}}F(0,w_H-y)\mathrm{d}h_2(y).
\end{align*}
Now let us generalise the previously defined function $f$ with three functions $f_1$,$f_2$ and $f_1$ that depend implicitly on $w_0,....,w_{H-1}$ :
\begin{align*}
    f_2(x,y)&=\mathbbm{1}_{x \leq w_{H-1}-\max(0, w_{H-2}-w_{H-3}+\max(0,... \max(0,w_4-w_3+ \max(0,w_2 -w_1+\max(0,w_0-y)))...))} ,\\
    f_0(x,y)&=\mathbbm{1}_{x \leq \max(0,w_{H-1}-w_{H-2}-\max(0,w_{H-3}-w_{H-4}-\max(0,... - \max(0,w_3-w_2-\max(0, w_1-w_0+y))...))}, \\
    f_1(x,y)&= \mathbbm{1}_{x \leq \sum_{l=1}^{H-1}(-1)^lw_{H-l}+y}.
\end{align*}
such that $f(y)=f_2(w_H,y)$.
Now, in $F(x,y)$, we can distinguish three cases:
\begin{enumerate}
    \item If $x=0$, then we must have $i=0$ and we do a similar decomposition as previously with $f_2(w_H-y,\cdot)$ instead of $f(\cdot)$.
    \item If $x=1$ and $i=2$, the inequality is automatically satisfied since it will be of the form $(1,y) \overset{\lex}{\leq}(2,a)$ for some $a$.
    \item If $x=1$ and $i=1$, the nested maximums will never evaluate to $(0,0)$ and the condition is equivalent to $f_1(w_H-y,Z)=1$.
    \item If $x=2$, then the inequality is automatically satisfied for $i=1,2$, and a similar argument as the first case will hold with $f_0(w_H-y,\cdot)$ instead.
\end{enumerate}
As a result, we can decompose the values of $F$ as integrals, denoting $l=\lim_{t \rightarrow +\infty}h_0(t)$ and $l'=\hat{\phi}(1-l)=\lim_{t \rightarrow -\infty}h_2(t)$ :
\begin{align*}
    p'(w_0,....,w_H)&= \int_{\mathbb{R}}\int_{\mathbb{R}}\left(f_2(w_H-y,u)\mathrm{d}h_0(u)\right)\mathrm{d}h_2(y) \\
    & \qquad + \int_{\mathbb{R}}\int_{\mathbb{R}}\left(f_1(w_H-y,u)\mathrm{d}h_1(u) + (1-l') \right)\mathrm{d}h_1(y) \\
    & \qquad + \int_{\mathbb{R}}\int_{\mathbb{R}}\left(f_0(w_H-y,u)\mathrm{d}h_0(u)+ (1-l)   \vphantom{h_\frac{1}{2}} \right)\mathrm{d}h_2(y).
\end{align*}
We know that $h_0$ and $h_2$ are invariants over ${\mathcal{H}}$. Using this, we have that
\begin{align*}
    p''(w_0,....,w_H) =\int_{\mathbb{R}}\int_{\mathbb{R}}f_1(w_H-y,u)\mathrm{d}h_1(u)\mathrm{d}h_1(y) 
\end{align*}
is an invariant for $h_1$.

Over all $H$ and all sequences $(w_i)\in (a,b)^{\mathbb{N}}$, the range of $\sum_{j=0}^{H}(-1)^lw_l$ is $\mathbb{R}$. Differentiating $p''$ with respect to $w_H$ yields that:
\begin{align*}
    \mathbb{P}_{((i,Z),(i',Z')) \sim \zeta_{\mathcal{H}}'\otimes \zeta_{\mathcal{H}}'}\left( Z+Z'=x , i=i'=1 \right)
\end{align*}
is an invariant over ${\mathcal{H}}$.
Taking characteristic functions, this implies that $(i,Z) | i=1$ have the same law for $(i,Z) \sim \zeta_{\mathcal{H}}$. Hence $h_1$ does not depend on $\mathcal H$, which concludes the proof when $\hat{\pi}_1 \neq 0$.

\bigskip

Finally if $\hat{\pi}_1=0$, pick any $l>0$ such that $\hat{\pi}_l>0$ and condition on every vertex on the right hand side up to depth $H$ of the root being of degree $l+1$ instead of having a simple path. Further condition on the weights being the same on every path with the same matching events at the end of every path. A similar proof with these invariants holds.
\subsection{Proof of Theorem~\ref{thm:convergenceZ}}
In this subsection we combine the previous results in order to prove Theorem~\ref{thm:convergenceZ}.

From \cite[Theorem 1]{enriquez2024optimalunimodularmatching}, we know that the optimal matching distribution $(\mathbb{T},o,\mathbb{M}_{\max}^{(\eps)})$ is constructed as a deterministic function of a decorated tree $(\mathbb{T},o,\mathbf{Z}^{(\eps)})$ whose distribution is characterised by:
\begin{itemize}
    \item The random messages $(\mathbf{Z}^{\eps}(u,v))_{(u,v) \in \overset{\rightarrow}{T}}$ verify
    \begin{equation}
        \forall (u,v) \in \overset{\rightarrow}{E}(\mathbb{T}), \, \mathbf{Z}^{\eps} (u,v)= \max(0, \max_{\substack{u' \sim v \\ u' \neq u}} (1+\eps w(v,u')-\mathbf{Z}^{\eps}(v,u')).
    \end{equation}
    \item For any finite $H>0$, the outgoing messages on the boundary of the ball $B_H(\mathbb{T},o)$ are i.i.d.~with stationary distribution $\zeta^{(\eps)}$ verifying the corresponding recursive distributional equation.
\end{itemize}
Let us treat the case $k=2$ for the sake of simplicity, the others being nearly identical.
We will prove that \[\left(\mathbb{T},o, \left(\frac{\mathbf{Z}^{\eps}}{\eps},\frac{\mathbf{Z}^{\eps}-\frac{1}{2}}{\eps},\frac{\mathbf{Z}^{\eps}-1}{\eps} \right) \right)\] converges in law to $(\mathbb{T},o, (Z_0,Z_1,Z_2))$. Since the matchings $\mathbb{M}_{\max}^{(\eps)}$
 can be represented as a deterministic function of the first variables and $\mathbb{M}_{\mathrm{opt}}$ as well as a deterministic function of the last variables, we obtain the convergence of the matchings as a by-product.

 We thus begin with the following lemma
\begin{lemma}\label{lem:Zconvergerenorm}
    Let $\mathbf{Z}^{\eps}$ a random variable of distribution $\zeta^{(\eps)}$.
    Consider $\tilde{\zeta}^{(\eps)}$  the distribution on $\mathbb{R}^3$ defined by the random variable \[\left(\frac{\mathbf{Z}^{\eps}}{\eps},\frac{\mathbf{Z}^{\eps}-\frac{1}{2}}{\eps},\frac{\mathbf{Z}^{\eps}-1}{\eps} \right) .\]
    Then $\tilde{\zeta}^{(\eps)}$ converges weakly to $\zeta_{\mathcal{H}}$ as $\eps \to 0$.
\end{lemma}

\begin{proof}[Proof of Lemma~\ref{lem:Zconvergerenorm}]
    In the proof of Lemma~\ref{prop:distributionzoom}, the distribution of $Z_0$ was obtained precisely as a subsequent distribution of $\frac{\mathbf{Z}^{\eps}}{\eps}$ and similarly for $Z_2$. We are also in the condition where the last renormalisation occurs around $\frac{1}{2}$, in which case we picked the distribution of $Z_1$ to be a subsequent limiting distribution of $\frac{\mathbf{Z}^{(\eps)}-\frac{1}{2}}{\eps}$.

    As Theorem~\ref{uniqueness} implies that there is only one subsequent limit, we obtain that full convergence holds.
\end{proof}

\begin{proof}[Proof of Theorem~\ref{thm:convergenceZ}]
    We start by taking $(\mathbb{T},o,\mathbb{M})$ to be a subsequent local limit of $(\mathbb{T},o,
    \mathbb{M}_{\max}^{(\eps)})$ as $\eps \to 0$.
    Let us write \[ \mathbf{Z}_0^{(\eps)},\mathbf{Z}_1^{(\eps)},\mathbf{Z}_2^{(\eps)}:=  \left(\frac{\mathbf{Z}^{\eps}}{\eps},\frac{\mathbf{Z}^{\eps}-\frac{1}{2}}{\eps},\frac{\mathbf{Z}^{\eps}-1}{\eps} \right)\]
    Fix $H>0$, the distribution of $B_H(\mathbb{T},o, \mathbb{M}_{\max}^{(\eps)} )$ is characterised by \[ B_H\left(\mathbb{T}, \left(\mathbf{Z}_0^{(\eps)},\mathbf{Z}_1^{(\eps)},\mathbf{Z}_2^{(\eps)} \right) ,o\right)   \]
    through
    \begin{align*}
     (u,v) \in \mathbb{M}_{\max}^{(\eps)} &\iff \mathbf{Z}^{(\eps)}(u,v)+\mathbf{Z}^{(\eps)}(v,u)<1+ \eps w(u,v)  \\
     &\iff \frac{\mathbf{Z}^{(\eps)}(u,v)+\mathbf{Z}^{(\eps)}(v,u)-1}{\eps}<  w(u,v) \\
     &\iff \exists i \in \{0,1,2\}, \mathbf{Z}_i^{(\eps)}(u,v)+\mathbf{Z}_{2-i}^{(\eps)}(v,u)<w(u,v).
     \end{align*}

     By construction, as they are the exterior variables on the boundary of $\partial B_H(\mathbb{T},o)$, the family of messages $\left(\left(\mathbf{Z}_0^{(\eps)},\mathbf{Z}_1^{(\eps)},\mathbf{Z}_2^{(\eps)} \right)(u,v)\right)_{(u,v) \in \partial B_H (\mathbb{T},o)}$ are i.i.d. We also showed in Lemma~\ref{lem:Zconvergerenorm} that they each converge in law to a random variable of distribution $\zeta_{\mathcal{H}}$. We deduce that they converge to a i.i.d. family $((Z_0,Z_1,Z_2)(u,v))_{(u,v) \in \partial B_H (\mathbb{T},o)}$ of distribution $\zeta_{\mathcal{H}}$.
    
     Next, we check that the family inside the ball must also converge, and they verify Recursion~\eqref{eq:recursionZlist2} which is the equivalent list form of Recursion~\eqref{eq:systemrecursionzoombis} in the case $k=2$.

     For any $\eps>0$, the family verifies, for any $(u,v)$ in the ball,
     \begin{equation}
         \begin{aligned}
             \mathbf{Z}_0^{(\eps)}(u,v)&=\max\left(0, \max_{\substack{u' \sim v \\ u' \neq u}} (w(v,u')-\mathbf{Z}_2^{(\eps)}(v,u')) \right), \\
             \mathbf{Z}_1^{(\eps)}(u,v)&=\max\left(\frac{-1}{2\eps}, \max_{\substack{u' \sim v \\ u' \neq u}} (w(v,u')-\mathbf{Z}_1^{(\eps)}(v,u')) \right), \\
             \mathbf{Z}_2^{(\eps)}(u,v)&=\max\left(\frac{-1}{\eps}, \max_{\substack{u' \sim v \\ u' \neq u}} (w(v,u')-\mathbf{Z}_0^{(\eps)}(v,u')) \right).
         \end{aligned}
     \end{equation}
     Taking $\eps \to 0$ in this system yields:
     \begin{equation}
         \begin{aligned}
             Z_0(u,v)&=\max\left(0, \max_{\substack{u' \sim v \\ u' \neq u}} (w(v,u')-Z_2^{}(v,u')) \right), \\
             Z_1(u,v)&=\max\left(-\infty, \max_{\substack{u' \sim v \\ u' \neq u}} (w(v,u')-Z_1(v,u')) \right), \\
             Z_2(u,v)&=\max\left(-\infty, \max_{\substack{u' \sim v \\ u' \neq u}} (w(v,u')-Z_0(v,u')) \right).
         \end{aligned}
     \end{equation}
    We can now iterate this recursion inside the ball to both convergence and that this equation is verified inside.
    We have thus shown that the limit $B_H(\mathbb{T},o,(Z_0,Z_1,Z_2))$ is characterised by 
     \begin{itemize}
         \item A family of exterior variables on the boundary of i.i.d. distribution $\zeta_{\mathcal{H}}$.
         \item The variables verify Equation~\eqref{eq:recursionZlist2} inside the ball.
     \end{itemize}
    We have thus obtained the convergence of the tree decorated with the messages.
    To conclude, we need to show that the same holds for the matching as a function of the messages.

    By construction, for any $\eps >0$, we have that:
    \[ \{u,v\} \in \mathbb{M}_{\max}^{(\eps)}\iff \exists i \in \{0,1,2\}, \, \mathbf{Z}_i^{(\eps)}(u,v)+\mathbf{Z}_{2-i}^{(\eps)}(v,u) < w(u,v) .\]
    Taking $\eps \to 0$, we have that
    \[ \{u,v\} \in \mathbb{M}\implies  \exists i \in \{0,1,2\}, \, Z_i(u,v)+ Z_{2-i}(v,u)\leq w(u,v).   \]
    Since the vector $(Z_0,Z_1,Z_2)$ has no atoms, we can almost surely replace the inequality by a strict one.
    Hence:
    \[ \{u,v\} \in \mathbb{M} \implies  \exists i \in \{0,1,2\}, \, Z_i(u,v)+ Z_{2-i}(v,u)\ < w(u,v).   \]
    The reverse implication clearly holds:
    \begin{align*}
     &\exists i \in \{0,1,2\},\,\,  Z_i(u,v)+ Z_{2-i}(v,u)\ < w(u,v)  \implies \\
     &\exists \eps_0>0, \, \forall \eps < \eps_0,  \, \exists i \in \{0,1,2\}, \,\, \mathbf{Z}_i^{(\eps)}(u,v)+\mathbf{Z}_{2-i}^{(\eps)}(v,u) < w(u,v).  \end{align*}
     Hence,
    \[  \{u,v\} \in \mathbb{M} \iff  \exists  \in \{0,1,2\}, \, Z_i(u,v)+ Z_{2-i}(v,u)\ < w(u,v). \]
    So, $\mathbb{M}$ agrees with $\mathbb{M}_{\mathrm{opt}}$. As $\mathbb{M}$ was chosen as any subsequent limit of $\mathbb{M}_{\max}^{(\eps)}$, we have thus obtained that
    \[ \left( \mathbb{T},o,\mathbb{M}_{\max}^{(\eps)}\right) \underset{\eps \to 0}{\longrightarrow} \left( \mathbb{T},o,\mathbb{M}_{\mathrm{opt}} \right)   \]
    as required.
\end{proof}

\section{The subcritical regimes}
\label{sec:subcritical}

Recall that a reproduction law  is subcritical for optimal matchings if it satisfies, for every random variable $X$ with values in $[0,1]$,
\begin{equation}
      \mathbb{E}\left[ \hat{\phi}'(1-X)\right] \hat{\phi}'\left(1-\mathbb{E}\left[\hat{\phi}(1-X)\right]\right) < 1    \tag{\ref{eq:subcriticalcond}}.
\end{equation}
In the first subsection, under this sub-criticality condition, we strengthen the approximation in law of optimal matchings for the local topology obtained in Theorem~\ref{maintheorem2} to a convergence in total variation of {\it finite} optimal matchings. On the way we prove that, in this regime, the messages and the state of edges in the optimal matching have uniform exponential \emph{point to set} correlation decay.

Similarly, we say that a reproduction law is subcritical for mandatory and blocking edges if it satisfies
\begin{equation}
    \exists!t \in (0,1): t=\hat{\phi}(1-\hat\phi(1-t)) \tag{\ref{eq:uniquefixedpoint}}.
\end{equation}
In the second subsection, we prove that in this regime, there is local convergence in total variation of mandatory and blocking edges along with local convergence in total variation of a sequence of random graphs to the corresponding UBGW tree.

Note that in both regimes, we must have a unique fixed point $\gamma= \hat{\phi}\left( 1- \hat{\phi}(1-\gamma)\right)$ that also verifies $\gamma = \hat{\phi}(1-\gamma)$, hence $k=1$ and $i \sim \gamma \delta_0 + (1-\gamma)\delta_1$.

\bigskip

The main arguments in both settings can be broken down as follows:
\begin{enumerate}
    \item First, for a given graph $G$ and $v \in V(G)$ such that $B_H(G_n,v)$ is a tree, we can represent the influence of the optimal matching outside the ball $B_H(G_n,v)$ by a well-chosen boundary condition on the messages on the ball. Namely, take a boundary vertex $u \in \partial B_H(G,v)$ and denote by $u'$ its unique neighbour in $B_H(G,v)$. If the vertex $u$ is matched to an exterior vertex, we can set $(i,Z)(u',u) = (1,  + \infty)$, if on the contrary $u$ is not matched to the exterior of $B_H(G,v)$, we can set $(i,Z)(u',u) = (0,0)$. Then, the restriction of the optimal matching to $B_H(G,v)$ coincides with the matching induced by the messages inside $B_H(G,v)$ defined by these boundary conditions.
    \item Another general fact is that inside $B_H(G,v)$, the messages around the root vertex $v$ are deterministically increasing with respect to the boundary conditions when $H$ is even, and decreasing when $H$ is odd. So as $H \rightarrow \infty$, we get that the messages are almost surely bounded by the limit of two extremal families built by considering the $(0,0)$ boundary condition and $(1, +\infty)$ boundary condition. However, in general, these two extremal families do not converge to the same limit so we are unable to say anything about the limit for a mixed boundary condition that is built from the first point.
    \item This motivates the search for a regime where these two extremal families converge almost surely to the same limit, we will call this regime subcritical. For Theorems~\ref{th:reciproquegamarnik} and~\ref{th:decay}, Condition~\ref{eq:subcriticalcond} is a sufficient condition for sub-criticality. To prove Theorem~\ref{thm:mandatory}, we only need a subcritical regime for the macroscopic component $i$ and Condition~\ref{eq:uniquefixedpoint} is a necessary and sufficient condition for sub-criticality on the macroscopic messages. In both cases, we get that the correlation between the messages around a vertex $v$ and $G_n \setminus B_H(G_n,v)$ decay exponentially and uniformly in $H$. 
\end{enumerate}

\subsection{Proof of Theorem~\ref{th:reciproquegamarnik}, Theorem~\ref{th:decay} and Corollary~\ref{coro:convergenceepsilonn} }
In this section, we pick a sequence of weighted random graphs $(G_n,o_n)$ such that:
\begin{enumerate}[label=\roman*)]
\item For any $\eps, H >0$, there exists for $n$ large enough a coupling $(\tilde{G}_n,\tilde{o}_n,\tilde{ \mathbb T},\tilde{o})$ such that
\begin{equation*}
    \mathbb{P}\left( B_H(\tilde{G}_n,\tilde{o}_n)=B_H( \tilde{\mathbb{T}},\tilde{o}) \right) > 1-\eps.
\end{equation*}
\item The reproduction law of $(\mathbb{T},o)$ is subcritical in the sense that it satisfies condition~\eqref{eq:subcriticalcond}.
\end{enumerate}

Let us give a slightly more precise presentation of our arguments in this setting.
For every $H,\varepsilon >0$, our goal is to prove that if $M_{\mathrm{opt}}(G_n)$ is an optimal matching on $G_n$, then there exists a rank $N_{\eps,H}'>0$ such that for every $n \geq N'_{\eps,H}$, there exists a coupling $(\tilde{G}_n,\tilde{o_n},\tilde{M}_{\mathrm{opt}}(G_n),\tilde{T},\tilde{o},\tilde{\mathbb{M}}_{\mathrm{opt}})$ of $(G_n,o_n,M_{\mathrm{opt}}(G_n))$ and $(\mathbb{T},o,\mathbb{M}_{\mathrm{opt}})$ such that:
\begin{equation}
    \mathbb{P}\left( B_H(\tilde{G}_n,\tilde{o_n},\tilde{M}_{\mathrm{opt}}(G_n))=B_H(\tilde{T},\tilde{o},\tilde{\mathbb{M}}_{\mathrm{opt}})  \right)>1-\eps.
\end{equation}

The main idea is that under condition~\eqref{eq:subcriticalcond}, the messages behave analogously to the subcritical regime of a Gibbs measure in the following sense:

For any $H_0,H>0$, one can consider the family of messages in $B_{H_0+H}(\mathbb{T},o, (i,Z))$ conditionally on a boundary condition $\partial_{H_0+H} Z$ of outgoing messages on the boundary of the ball. When condition~\ref{eq:subcriticalcond} is satisfied, we will show that the messages in the fixed smaller ball $B_{H_0} (\mathbb{T},o, (i,Z))$ become independent of the boundary conditions as $H \rightarrow \infty$ and converge in total variation to the stationary distribution corresponding to $\mathcal{H}$ from the previous sections. This is proved by using two facts:
\begin{itemize}
    \item Consider two families of extremal messages which are constructed by setting boundary messages to be $(0,0)$ and $(1,+\infty)$. Then these two families bound any family of messages within the ball deterministically.
    \item In this subcritical regime given by condition~\eqref{eq:subcriticalcond}, the two extremal families of messages are almost surely monotonous when incrementing $H$ by $2$ and converge in total variation to the same random variable.
\end{itemize}
This is analogous to the uniqueness of extremal Gibbs measures in the sense that we use families of extremal variables. However, it is different in the sense that we use a deterministic monotonicity of extremal measures with respect to the depth. Note that this also implies that there almost surely exists a unique family of messages $(i,Z)$ on $\mathbb{T}$ verifying the recursive Equation~\eqref{eq:systemrecursionzoombis}.

Next, using i), outside of an event of probability less than $\eps$, $B_{H_0+H}(\tilde{G}_n,\tilde{o}_n)$ is isomorphic to $B_{H_0+H}(\tilde{\mathbb{T}},\tilde{o})$. We will then show that the influence of $M_{\mathrm{opt}}(G_n)$ outside of this ball can be condensed into a boundary condition taking values in $(i,Z)=(0,0)$ or $(1,+\infty)$ and that $M_{\mathrm{opt}}$ in $B_{H_0+H}(\tilde{G}_n,\tilde{o}_n)$ agrees with the matching given by a family of messages $(i,Z)$ with the boundary condition given by $M_{\mathrm{opt}}(G_n)$. From the previous paragraph, between the boundaries at $H_0+H$ and $H_0$, the messages forget this boundary condition and 
converge uniformly to the stationary distribution. This implies that, with arbitrary small error, the matching $\tilde{M}_{\mathrm{opt}}(G_n)$ coincides on $B_{H_0}(\tilde{G}_n,\tilde{o}_n)$ with a matching given by a family of messages that is arbitrarily close to the stationary distribution, which concludes the argument. See Figure~\ref{fig:gamarnik} for an illustration of the situation.

\begin{figure}[t!]
    \centering
    \includegraphics[width=0.7\linewidth]{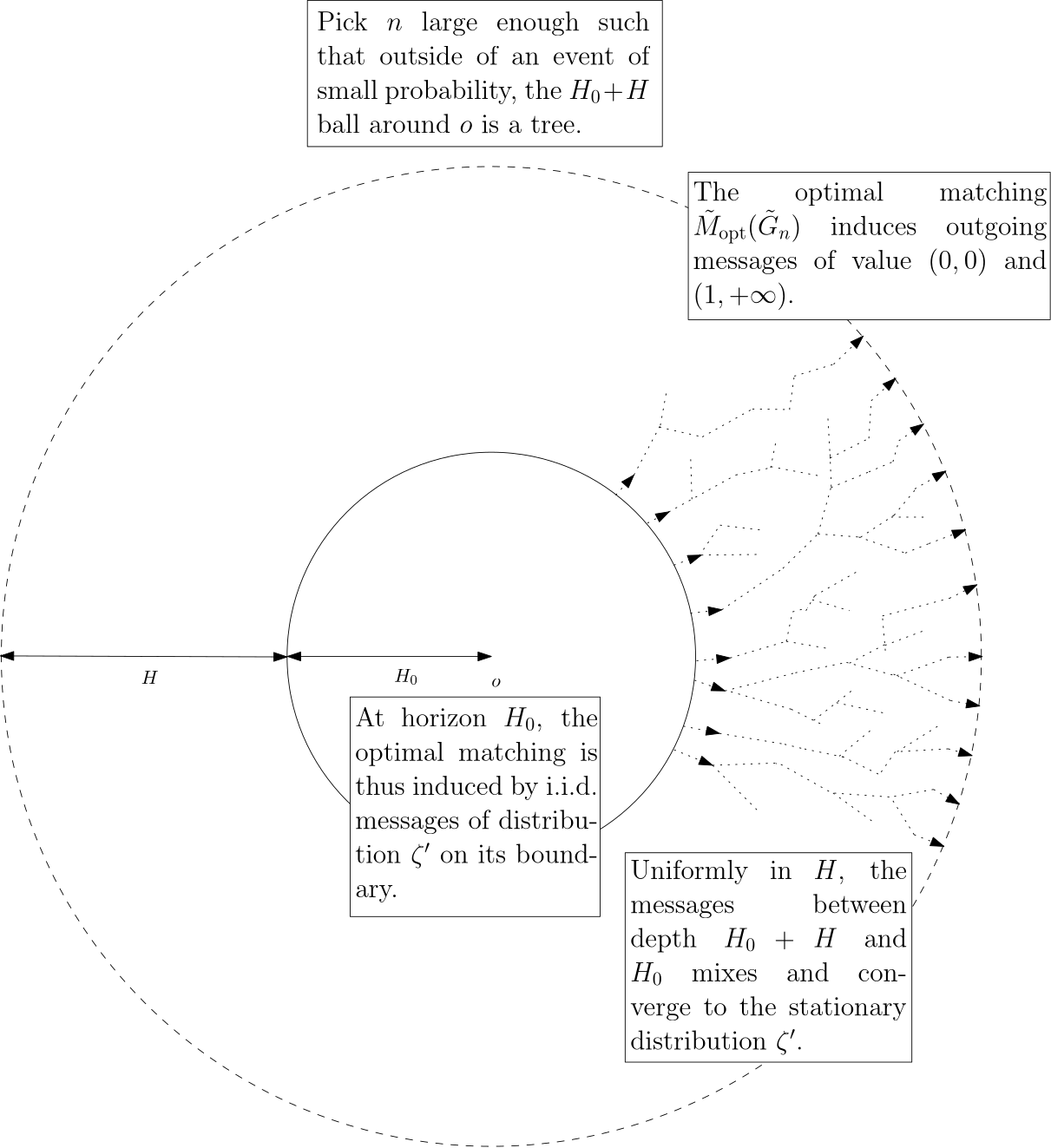}
    \caption{Illustration of the situation in Theorem~\ref{th:reciproquegamarnik}. The text is to be read clockwise.}
    \label{fig:gamarnik}
\end{figure}

\bigskip

To proceed, let us start with the following lemma that treats the special case $H_0 = 1$:
\begin{lemma}\label{lem:subcritical}
    Define the families of truncated decorated trees $\left(\mathbb{T}_H,o,(i_{-}^{H},Z_{-}^H),(i_{+}^{H},Z_{+}^H)\right)_{H \geq 0}$ as follows:
    Sample $(\mathbb{T},o)$, set $\mathbb{T}_H=B_H(\mathbb{T},o)$ and   $(i_{-}^{H},Z_{-}^H)(u,v)=(0,0)$ for every edge of $\mathbb{T}_H$ pointing towards the $H-$boundary. Then use Recursion~\eqref{eq:systemrecursionzoombis} to define $(i_{-}^{H},Z_{-}^H)$ on the remainder of $\mathbb{T}_H$. We proceed similarly for $(i_{+}^{H},Z_{+}^H)$ by setting $(i_{+}^{H},Z_{+}^H)=(1,+\infty)$ on the boundary instead of $(0,0)$.

    Then under condition~\eqref{eq:subcriticalcond}, both families $((i_{+}^{H},Z_{+}^H)(o,v))_{v \sim o}$ and $((i_{-}^{H},Z_{-}^H)(o,v))_{v \sim o}$ converge almost surely, as $H$ goes to infinity, to the same family $(i_l,Z_l)_{1 \leq l \leq N}$ where $(i_l,Z_l)$ have i.i.d distribution $\zeta'$ and are independent of $N \sim \pi$.
\end{lemma}
\begin{remark}
Note that $(\mathbb{T}_H,o,(i_{-}^{H},Z_{-}^H),(i_{+}^{H},Z_{+}^H))$ does not define a consistent family in $H$. 
\end{remark}
\begin{proof}[Proof of Lemma~\ref{lem:subcritical}]
We will use the fact that the random variables $(i_-^{2H},Z_-^{2H}) (o,v)$  (resp. $(i_+^{2H},Z_+^{2H}) (o,v))$ are deterministically increasing (resp. decreasing) as $H$ increases.
Indeed, for a fixed family of weights $(w_l)_{l \geq 0}$, if we pick two families $(a_l,b_l)$ and $(a_l',b_l')$ such that for every $l$, $(a_l,b_l) \overset{\lex}{\leq} (a_l',b_l')$, then for any integer $N \geq 0$:
\[ \maxlex\left( (0,0) ,\maxlex_{l \leq N} (1,w_l)-(a_l,b_l)\right) \overset{\lex}{\geq} \maxlex\left( (0,0), \maxlex_{l \leq N} (1,w_l)-(a_l',b_l')\right). \]

In terms of random variables, the following almost sure limits thus exist when looking at even depths:
\begin{align*}
    (i_{-}^{\text{even}},Z_{-}^{\text{even}}) &= \lim_{H \rightarrow +\infty} (i_{-}^{2H},Z_{-}^{2H}) ,\\
    (i_{+}^{\text{even}},Z_{+}^{\text{even}}) &= \lim_{H \rightarrow +\infty} (i_{+}^{2H},Z_{+}^{2H}).
\end{align*}

For the odd depths, the analogous limits also exist:
\begin{align*}
    (i_{-}^{
    \text{odd}},Z_{-}^{
    \text{odd}}) &= \lim_{H \rightarrow +\infty} (i_-^{2H+1},Z_{-}^{2H+1}), \\
    (i_{+}^{
    \text{odd}},Z_{+}^{\text{odd}}) &= \lim_{H \rightarrow +\infty} (i_{+}^{2H+1}, Z_{+}^{2H+1}).
\end{align*}

Furthermore, we end up with the following almost sure  chain of inequalities
\begin{align*} 
&(i_-^{2H},Z_{-}^{2H}) \overset{\lex}{\leq} (i_{+}^{2H+1},Z_{+}^{2H+1}) \overset{\lex}{\leq } (i_-^{2H+2},Z_{-}^{2H+2}) \\ \overset{\lex}{\leq} &(i_+^{2H+2},Z_{+}^{2H+2}) \overset{\lex}{\leq}(i_-^{2H+1},Z_{-}^{2H+1}) \overset{\lex}{\leq}(i_+^{2H},Z_{+}^{2H}).    
\end{align*}
Going from the outer edge to the centre, the first inequalities are obtained by looking at the boundary conditions of both variables at depth $H$, the next inequalities by looking at their boundary conditions at depth $H+1$, and the centre inequality by looking at their boundary conditions at depth $H+2$. Taking $H$ to infinity implies that:
\[  (i_{-}^{\text{even}},Z_{-}^{\text{even}}) \overset{\lex}{\leq} (i_{+}^{
    \text{odd}},Z_{+}^{\text{odd}})\overset{\lex}{\leq}  (i_{-}^{\text{even}},Z_{-}^{\text{even}}) \overset{\lex}{\leq} (i_{+}^{\text{even}},Z_{+}^{\text{even}})\overset{\lex}{\leq} (i_{-}^{
    \text{odd}},Z_{-}^{
    \text{odd}})\overset{\lex}{\leq} (i_{+}^{\text{even}},Z_{+}^{\text{even}}) .  \]
which condenses into
\[ (i_{-}^{\text{even}},Z_{-}^{\text{even}}) = (i_{+}^{
    \text{odd}},Z_{+}^{\text{odd}})\overset{\lex}{\leq} (i_{+}^{\text{even}},Z_{+}^{\text{even}})=(i_{-}^{
    \text{odd}},Z_{-}^{
    \text{odd}}).   \]
To show the conclusion, we only need to show that the centre inequality is in fact an equality. For this purpose, it is sufficient to show that the distribution of the two variables on the left is the same as the distribution of the two variables on the right.
We will use a contraction argument for the distributions of those variables for the total variation distance.
Let $\mathcal X$ be the space of pairs of c\`adl\`ag increasing functions $(f_0,f_1)$ taking values in $[0,1]$ such that:
    \begin{align*}
        \lim_{t \rightarrow +\infty} f_0(t)&= \lim_{t \rightarrow -\infty} f_1(t).
    \end{align*}

We interpret $(f_0,f_1)$ as a random variable taking values in $\{0,1\} \times \mathbb{R}$ through the list representation.
Let us consider the operator $F: \mathcal X \mapsto \mathcal X$ such that if one writes $F(f_0,f_1)=(F_0(f_0,f_1),F_0(f_0,f_1))$, we have:
\begin{align*}
    \forall t \in \mathbb{R}, F_1(f_0,f_1)(t)&= \mathbbm{1}_{t \geq 0} \hat{\phi}\left(1-\mathbb{E}_{W \sim \omega}[f_1(W-t)]\right),  \\
    \forall t \in \mathbb{R}, F_2(f_0,f_1)(t)&= \hat{\phi}\left(1-\mathbb{E}_{W\sim \omega}[f_0(W-t)] \right).
\end{align*}
The first observation is then that the distribution of $(i_{-}^{H},Z_{-}^H)(o,v)$ for any $v \sim o$ is exactly $F^{\circ H}(\mathbbm{1}_{[0,+\infty[},1)$ (where $1$ is to be understood at the constant function equal to $1$) and similarly, the distribution of $(i_{+}^{H},Z_{+}^H)(o,v)$
is $F^{\circ H}(1,1)$.
The second observation is that $F$ is stochastically decreasing, in the sense that for pairs of functions $(f_0,f_1)$ and $(f_0',f_1')$ one has 
\[
\left( f_0 \leq f_0' , f_1 \leq f_1' \right)
\Longrightarrow 
\left( F_0(f_0,f_1) \geq F^{\circ(2)}_0(f_0',f_1'), F^{\circ(2)}_1(f_0,f_1) \geq F_1(f_0',f_1')\right).
\]
This can be deduced either directly from the expression of $F$.
In either case, we obtain that $F^{\circ 2}$ is stochastically increasing, namely
\[f_0 \leq f_0' , f_1 \leq f_1' \Longrightarrow F^{\circ(2)}_0(f_0,f_1) \leq F^{\circ(2)}_0(f_0',f_1'), F^{\circ(2)}_1(f_0,f_1) \leq F^{\circ(2)}_1(f_0',f_1'). \]
We will shorten $(f_0,f_1) \leq (f_0',f_1')$ to mean that $f_0 \leq f_0'$ and $f_1 \leq f_1'$.
    
In terms of the operator $F$, it thus suffices to show that $F^{\circ 2}$ admits a unique fixed point to obtain the conclusion, since the distributions of $(i_{+}^{\text{odd}},Z_+^{\text{odd}})$ and the distribution of $(i_{+}^{\text{even}},Z_+^{\text{even}})$ are both fixed points of $F^{\circ 2}$.

To this end, we equip $X$ with the total variation distance defined as 
\begin{equation}\label{def:TVfunc}
    d_{TV}((f_0,f_1),(f_0',f_1'))=\max( \sup_\mathbb{R} |f_0-f_0'|,\sup_\mathbb{R}|f_1-f_1'|  ).
\end{equation}
We will then show that $F^{\circ 2}$ is a local contraction of $(X,d_{TV})$ in the sense that there exists $K<1$ and $\eps>0$ such that for any pair $(f_0,f_1),(f_0',f_1') \in X$,
\[ d_{TV}((f_0,f_1),(f_0',f_1'))< \eps \Longrightarrow d_{TV}(F^{\circ 2}(f_0,f_1),F^{\circ 2}(f_0',f_1'))<K d_{TV}((f_0,f_1),(f_0',f_1')). \]
Since $X$ is convex, this will imply that $F^{\circ 2}$ is a global contraction.
First, we notice that $F^{\circ 2}$ can be reduced to a single dimensional mapping, as \[F^{\circ 2}_0(f_0,f_1)(t)=\mathbbm{1}_{t \geq 0} \hat{\phi}(1-\mathbb{E}_{W\sim \omega}[\hat{\phi}(1-\mathbb{E}_{W' \sim \omega}[f_0(W'-W+t])]\] and similarly for $F^{\circ 2}_1$.
We will thus only show that $F^{\circ 2}_0$ is a contraction on $f_0$, the $f_1$ case being nearly identical.
We thus write that
\begin{align*}
&|F^{\circ 2}_0(f_0,f_1)(t)-F^{\circ 2}_0(f_0',f_1')(t)|   = \\
&\bigg|\hat{\phi}(1-\mathbb{E}_{W\sim \omega}[\hat{\phi}\left(1-\mathbb{E}_{W' \sim \omega}\left[f_0(W'-W+t\right)])]\right) - \\ &\hat{\phi}\left(1-\mathbb{E}_{W\sim \omega}[\hat{\phi}(1-\mathbb{E}_{W' \sim \omega}\left[f_0'(W'-W+t])\right]\right)\bigg|
\end{align*}
We can upper-bound the difference by the case where $f_0'-f_0=\pm \sup |f_0'-f_0|$. Let us write $\eta=\sup |f_0'-f_0|$, using a mean value inequality, we get that:
\begin{align*}
    &|F^{\circ 2}_0(f_0,f_1)(t)-F^{\circ 2}_0(f_0',f_1')(t)| \\
    \leq &\eta \sup_{u \in [-\eta,\eta]} \mathbb{E}_W\left[\hat{\phi}'(1-\mathbb{E}_{W'}[f_0(W'-W+t)+u]\right]\hat{\phi}'\left(1-\mathbb{E}_W[\hat{\phi}(1-\mathbb{E}_{W'}[f_0(W'-W+t)+u])]\right).
\end{align*}
Let us denote by $X(W)$ the random variable $\mathbb{E}_{W'}[f_0(W'-W+t)]$, then the upper bound rewrites as
\[ \eta \sup_{u \in [-\eta,\eta]} \mathbb{E}_X \left[\hat{\phi}'(1-X+u) \right]\hat{\phi'}\left(1-\mathbb{E}_X \left[ \hat{\phi}(1-X+u)\right]\right). \]
Dominated convergence theorem shows that $u \mapsto \mathbb{E}_X \left[\hat{\phi}'(1-X+u) \right]\hat{\phi'}\left(1-\mathbb{E}_X \left[ \hat{\phi}(1-X+u)\right]\right)$ is continuous, and condition~\ref{eq:subcriticalcond} states that evaluating at $u=0$ yields a number less than $1$.  We deduce that there exists $\eps>0$ such that
\[ \sup_{u \in [-\eps,\eps]} \mathbb{E}_X[\hat{\phi}'(1-X+u)]\hat{\phi}'\left( 1-\mathbb{E}_X\left[\hat{\phi}(1-X+u)\right]\right)  <1.\]
Set $K=\sup_{u \in [-\eps,\eps]} \mathbb{E}_X[\hat{\phi}'(1-X+u)]\hat{\phi}'\left( 1-\mathbb{E}_X\left[\hat{\phi}(1-X+u)\right]\right)$, we thus deduce that for $\sup |f_0-f_1|= \eta < \eps$, we have:
\begin{align*}
    |F^{\circ 2}_0(f_0,f_1)(t)-F^{\circ 2}_0(f_0',f_1')(t)| 
    \leq K\eta =K\sup|f_0-f_1|.
\end{align*}
Taking the supremum over $t$ yields the desired result.
\end{proof}
\begin{remark}
    In the previous proof, we can always impose that $X+u \in [0,1]$, such that the constant $K$ appearing in the proof does not depend on the value of $X$. We deduce that the contraction happens with coefficient 
    \[ \rho= 
      \sup_{X \text{ r.v. in} [0,1]} \mathbb{E}\left[ \hat{\phi}'(1-X)\right] \hat{\phi}'\left(1-\mathbb{E}\left[\hat{\phi}(1-X)\right]\right)  .\]
    on $(\mathcal{X},\mathrm{d}_{\mathrm{TV}})$.
\end{remark}
The next argument, also for the special case $H_0=1$, is the following:

\begin{lemma}\label{lem:matchingbord}
    Let $(G,o,w)$ be a finite weighted graph with a unique optimal matching $M_{\mathrm{opt}}$. Let $H>0$ and assume that $G^H:=B_H(G,o)$ is a tree. 
    Define the decorated graph $(G^H,o, (i_M,Z_M))$ as follows:
    \begin{enumerate}
        \item Set $(i_M,Z_M)(u,v)=(0,0)$ if $v$ is on the boundary and not matched by $M_{\mathrm{opt}}$ to the exterior of $G^H$. In other words, $(i_M,Z_M)(u,v)=(0,0)$ if $d(v,o)=H+1$ and $(v,u') \notin M_{\mathrm{opt}}$ for $u' \sim v, u' \neq u$.
        \item Set $(i_M,Z_M)(u,v)=(1,+\infty)$  if $v$ is on the boundary and matched by $M_{\mathrm{opt}}$ to the exterior of $G^H$.
        \item Use Recursion~\eqref{eq:systemrecursionzoombis} to define $(i_M,Z_M)$ on the remainder of $(G^H,o)$.
    \end{enumerate}
    Then an edge $(u,v)\in G^H$ is in $M_{\mathrm{opt}}$ if and only if
     \begin{align*}
         (i_M,Z_M)(u,v)+(i_M,Z_M)(v,u)\overset{\lex}{<} (1,w(u,v)).
     \end{align*}
\end{lemma}
\begin{proof}[Proof of Lemma~\ref{lem:matchingbord}]
The idea is that the optimal matching in $G^H$ behaves as if every vertex on the boundary matched to the exterior in $M_{\mathrm{opt}}$ is removed and untouched otherwise. Setting $(i_M,Z_M)(u,v)=(0,0)$ has the same effect in $G^H$ as if $v$ was a leaf. Similarly, setting $(i_M,Z_M)(u,v)=(1,+\infty)$ forbids $(u,v)$ to be in the matching. Indeed, we always have  $(1,+\infty)\overset{\lex}{>} (1,w(u,v))$ and this removes any effect that $(u,v)$ has in the rest of $G^H$ since the value appearing in Recursion~\ref{eq:systemrecursionzoombis} in the next generation's from $(u,v)$ is $(1,w(u,v))-(1,+\infty)=(0,-\infty)$ which is always smaller than $(0,0)$. We know that the $(i,Z)$ message formalism is exact on finite trees so the rule \[(i_M,Z_M)(u,v)+(i_M,Z_M)(v,u) \overset{\lex}{<} (1,w(u,v)) \] constructs optimal matchings on $G^{H,\star}:=G^H \setminus \{v, \exists u \in G \setminus G^H, (u,v)\in M_{\mathrm{opt}}\}$. 
\end{proof}
We can now combine the two lemmas to obtain Theorem~\ref{th:reciproquegamarnik}.
\begin{proof}[Proof of Theorem~\ref{th:reciproquegamarnik}.]
Fix $\eps>0$ and $H_0>0$.
For any $\eps'>0$, use Lemma~\ref{lem:subcritical} to find $H>0$ such that \[d_{TV}((i_-^{H},Z_-^{H})(o,v)_{v \sim o}, (i_+^{H},Z_+^{H})(o,v)_{v \sim o})<1-\eps'.\]
Next, use condition i) and find a rank $N_{\eps',H_0+H}$ such that for any $n \geq N_{\eps',H_0+H}$, there exists a coupling of $(\tilde{G}_n,\tilde{o}_n,\tilde{\mathbb{T}},\tilde{o})$ of $(G_n,o_n)$ and $(\mathbb{T},o)$ such that $B_{H_0+H}(\tilde{G}_n,\tilde{o}_n)=B_{H_0+H}(\tilde{\mathbb{T}},\tilde{o})$ outside of an event of probability at most $\eps'$.
Take any optimal matching $M_{\mathrm{opt}}(G_n)$ on $G_n$ and its corresponding matching $\tilde{M}_{\mathrm{opt}}(G_n)$ on $\tilde{G}_n$. Using Lemma~\ref{lem:matchingbord}, outside of an event of probability at most $\eps'$, we know that $B_{H_0+H}(\tilde{G}_n,o)$ is a tree with i.i.d weights so there is a unique matching on \[\tilde{G}_n^{H_0+H,\star}:=B_{H_0+H}(\tilde{G}_n,\tilde{o}_n) \setminus \{v, \exists u \in \tilde{G}_n \setminus B_{H_0+H}(\tilde{G}_n,\tilde{o}_n), (u,v) \in \tilde{M}_n(\tilde{G}_n))\] so the lemma says that this optimal matching agrees with a family of messages $(i_M,Z_M)$ constructed with boundary conditions specified inside the lemma. 
Define the set of directed edges on the $H_0-$boundary to be $\partial_{H_0}(\tilde{G}_n,\tilde{o}):= \{(u,v), d(u,\tilde{o})=H_0, d(v,\tilde{o})=H_0+1\}$.
For every $(u,v) \in \partial_{H_0}(\tilde{G}_n,\tilde{o}_n)$, define the subtree $T_{u,v}^{H}$ to be the subtree between depth $H_0$ and $H_0+H$ that $(u,v)$ is pointing towards. The family $(T_{u,v}^{H},v)_{(u,v) \in \partial_{H_0}(\tilde{G}_n,\tilde{o})}$ forms a disjoint family of truncated trees. 
We can define on each of these trees a family $(i_-^{H,(u,v)},Z_-^{H,(u,v)})$ and $(i_+^{H,(u,v)},Z_+^{H,(u,v)})$ with the same procedure as in Lemma~\ref{lem:subcritical}.
Recall the almost sure messages are decreasing with respect to the boundary conditions at odd depth and increasing at even depth with \textbf{fixed} weights. 

Any mixed boundary condition is lower bounded by the pure $(0,0)$ boundary condition and upper bounded by the pure $(1,+\infty)$ boundary condition. Furthermore, the families 
\begin{align*}
&\left((i_-^{H,(u,v)},Z_-^{(u,v)})(u',v')_{(u',v') \in T_{u,v}^{H}}\right), \\ &\left( (i_+^{H,(u,v)},Z_+^{(u,v)})(u',v')_{(u',v') \in T_{u,v}^{H}} \right)  \text{ and}\\ &\left((i_M,Z_M)(u',v')_{ (u',v')  \in T_{(u,v)}^{H}} \right) 
\end{align*} use the same weights for their recursions. Thus we get that 
each $(i_M,Z_M)(u,v)$ for $(u,v) \in \partial_{H_0}(\tilde{G}_n,\tilde{o}_n)$ lies between
\begin{align*}
    &\maxlex \left( (0,0),\maxlex_{u' \sim v, u' \in T_{u,v}^{H}} (1,w(v,u')-(i_-^{H,(u,v)},Z_-^{H,(u,v)})(v,u')  )\right)\\ \text{ and } \\
    &\maxlex \left( (0,0),\maxlex_{u' \sim v, u' \in T_{u,v}^{H}} (1,w(v,u')-(i_+^{H,(u,v)},Z_+^{H,(u,v)})(v,u') ) \right),
\end{align*}
where the ordering of the upper and lower bounds depend on whether $H$ is even or odd.
But Lemma~\ref{lem:subcritical} states that the families \[\left((i_-^{H,(u,v)},Z_-^{H,(u,v)})(v,u')\right)_{u' \sim v} \text{and }\left((i_+^{H,(u,v)},Z_+^{H,(u,v)})(v,u')\right)_{u' \sim v}\] are $\eps'$ close in total variation to a i.i.d family of random variables of distribution $\zeta'$.
This shows that $(i_M,Z_M)(u,v)$ is $\eps'$ close to 
\[\maxlex \left( (0,0),\maxlex_{u' \sim v, u' \in T_{u,v}^{H}} ((1,w(v,u'))- (i_{u'},Z_{u'} ) \right) \]
where $(i_{u'},Z_{u'})$ are i.i.d of law $\zeta'$. By stationarity of $\zeta'$ we deduce that each $(i_M,Z_M)(u,v)$ is also $\eps'$ close to a random variable of law $\zeta'$.

Since each random family $\left((i_-^{H,(u,v)},Z_-^{H,(u,v)})(v,u')\right)_{u' \sim v}$ and $\left((i_+^{H,(u,v)},Z_+^{H,(u,v)})(v,u')\right)_{u' \sim v}$ does not depend on the graph outside of the corresponding tree $(T_{u,v}^{H},v)$, which are independent in $\mathbb{T}$, we furthermore deduce that the family $(i_M,Z_M)(u,v)_{(u,v) \in \partial_H(\tilde{G}_n, \tilde{o}_n)}$ is $ |\partial_{H_0}(\tilde{G}_n,\tilde{o}_n) | \eps'$ close in total variation to a i.i.d family of distribution $\zeta'$. 

In conclusion, we showed that outside of an event of probability $\eps'$, $\tilde{M}_{\mathrm{opt}}(\tilde{G}_n)$ agrees on $B_{H_0}(\tilde{G}_n, \tilde{o}_n)$ with a matching generated by:
\begin{enumerate}
    \item Outside variables $(i_M,Z_M)(u,v)_{(u,v) \in \partial_{H_0}(\tilde{G}_n,\tilde{o}_n)}$ that are $|\partial_{H_0}(\tilde{G}_n,\tilde{o}_n) | \eps'$ close in total variation to an i.i.d family of variables of law $\zeta'$.
    \item The family $(i_M,Z_M)$ inside $B_{H_0}(\tilde{G}_n,\tilde{o}_n)$ generated by Recursion~\eqref{eq:systemrecursionzoombis}.
\end{enumerate}
Since ${H_0}$ is finite, one can pick $\eps'$ small enough so that this implies that the entire family $(B_{H_0}(\tilde{G}_n),\tilde{o}_n,(i_M,Z_M)_{(u,v) \in B_{H_0}(\tilde{G}_n,\tilde{o}_n)})$ agrees with $B_{H_0}(\mathbb{T},o,(i,Z)(u,v)_{(u,v) \in \overset{\rightarrow}{E}})$ outside of an event of probability $1-\tfrac{\eps}{2}$. Forcing $\eps'<\tfrac{\eps}{2}$ thus implies that we have found a rank $N_{\eps',H_0}$ for which $(\tilde{G}_n,\tilde{o}_n,\tilde{M}_{\mathrm{opt}}(\tilde{G}_n))$ agrees with $(\tilde{\mathbb{T}},\tilde{o},\tilde{\mathbb{M}}_{\mathrm{opt}})$ outside of an event of probability $\eps$, proving Theorem~\ref{th:reciproquegamarnik}.
\end{proof}

\bigskip

The proofs of Theorem~\ref{th:decay} and of Corollary~\ref{coro:indepasympt} are basically contained in the previous proofs of this section:

\begin{proof}[Proof of Theorem~\ref{th:decay}]
Theorem~\ref{th:decay} is essentially a quantified version of Lemma~\ref{lem:subcritical}. We have shown in the proof of Lemma~\ref{lem:subcritical} that the family \[ (i_M,Z_M)(o,v)_{(o,v)\in \mathbb{T}_{H})} \] converges geometrically fast to the stationary distribution with $H$, and that the contraction coefficient is precisely $\rho_\pi$.
\end{proof}
\begin{proof}[Proof of Corollary~\ref{coro:indepasympt}]
If the distance between two uniformly chosen vertices does not go almost surely to infinity, then it is clear that $G_n$ seen from two independently chosen uniform vertices cannot be independent copies of $(\mathbb{T},o)$. We thus deduce that the distance between two uniformly chosen vertices must go to infinity.

We can then redo the same proof with two $H_0$-balls around $o_n$ and $o_n'$.
Theorem~\ref{th:decay} states that the influence of the matching on one ball towards the other decay geometrically fast in their distance. The conclusion follows by combining this with the fact that the distance between two uniformly chosen vertices (and hence the two balls) goes to infinity.
\end{proof}

\begin{proof}[Proof of Corollary~\ref{coro:convergenceepsilonn}]
    The proof is essentially the same, with the added fact that the contraction argument holds simultaneously for every $(G_n,o_n,M_{\max}^{(\eps)}(G_n))$ over every $\eps>0$ with the same contraction coefficient given by condition~\eqref{eq:subcriticalcond}. We can write the same proof while considering the entire family of matchings parametrised by $\eps$.
\end{proof}

\bigskip

We end this section with a quick calculation to show that Poisson distribution with parameter $c$ for $c<e$, satisfies condition~\eqref{eq:subcriticalcond}.
We want to show that for any $[0,1]-$valued random variable, we have that
\begin{align*}
    \mathbb{E}\left[\hat{\phi}'(1-X)\right]\hat{\phi}'\left(1-\mathbb{E}\left[\hat{\phi}(1-X)\right]\right) <1
\end{align*}
where $\hat{\phi}(1-t)=e^{-ct}$.
This rewrites as
\begin{align*}
    c^2\mathbb{E}[e^{-cX}]e^{-c\mathbb{E}[e^{-cX}]}.
\end{align*}
If we write $y=\mathbb{E}[e^{-cX}] \in [e^{-c},1]$, then the quantity rewrites as $c^2ye^{-cy}$.
The map $t \mapsto te^{-ct}$ reaches its maximum for $y=\frac{1}{c}$. We thus have that the quantity of interest is bounded by
$\frac{c}{e}$
which is smaller than $1$ if and only if $c<e$, as announced.
\subsection{Proof of Theorem~\ref{thm:mandatory}}
We now turn to the proof of Theorem~\ref{thm:mandatory}. The overarching idea behind this theorem is that the hypothesis on $t \mapsto \hat{\phi}(1-\hat{\phi}(1-t))$ having a unique fixed point is sufficient to prove the subcritical behaviour for messages as stated in Lemma~\ref{lem:subcritical}, but only for the first marginals $i_-,i_+$. It thus states that the messages $i_M$ forget the boundary uniformly fast over all possible boundary conditions. We then write that the set of maximum size matchings induces a set of possible boundary conditions. Combining these facts gives that the messages forget the set of all maximal matchings outside the larger ball, and thus behaves like the set of matchings on a finite tree in the fixed neighbourhood. The final ingredient that we use is that the support of maximum weight maximum size matchings on a finite graph over the set of all possible edge weights is precisely the set of maximal matchings. We can then show that if an edge $e$ satisfies $i_M(\overset{\rightarrow}{e})+i_M(\overset{\leftarrow}{e})=1$, then it is possible to locally modify the weights to find a maximum weight maximum size matching that includes the edge $e$ and another that does not.

We thus start with the modified version of Lemma~\ref{lem:subcritical}. Under assumption~\eqref{eq:uniquefixedpoint}, denote by $\gamma$ the unique fixed point of $t \mapsto \hat{\phi} (1- \hat{\phi}(1-t))$.
\begin{lemma}\label{lem:subcriticalproj}
    Define the families of truncated decorated trees $\left(\mathbb{T}_H,o,i_{-}^{H},i_{+}^{H}\right)_{H \geq 0}$ as follows:
    Sample $(\mathbb{T},o)$ an unweighted Bienaymé-Galton-Watson tree of reproduction law $\pi$, set $\mathbb{T}_H=B_H(\mathbb{T},o)$ and   $(i_{-}^{H})(u,v)=0$ for every edge of $\mathbb{T}_H$ pointing towards the $H-$boundary. Then use Recursion~\eqref{eq:systemrecursionzoombis} to define $i_{-}^{H}$ on the remainder of $\mathbb{T}_H$. Proceed similarly for $i_{+}^{H}$ by setting $i_{+}^{H}=1$ on the boundary instead of $0$.

    Then under assuming that $t \mapsto \hat{\phi}(1-\hat{\phi}(1-t))$ possesses a unique fixed point, both families $(i_{+}^{H}(o,v))_{v \sim o}$ and $(i_{-}^{H}(o,v))_{v \sim o}$ converges almost surely to a family $(i_l)_{l \leq N}$ where $(i_l)$ have i.i.d distribution $\gamma\delta_0+(1-\gamma)\delta_1$ and independent of $N$ of distribution $\pi$.
\end{lemma}
\begin{proof}[Proof of Lemma~\ref{lem:subcriticalproj}]
    The proof is in fact simpler than the two dimensional one. We skip the discussion about almost sure monotonicity and go straight into the proof of the contraction. The space of laws $\mu$ taking values in $\{0,1\}$ can be parametrised by a real $p\in [0,1]$ such that $\mathbb{P}_{Y \sim \mu}(Y=0)=p$ (such that $\mu \sim B(1-p)$). The corresponding operator is then $F:[0,1]\mapsto [0,1]$ with
    \[ F(p)=\hat{\phi}(1-p). \]
    By hypothesis, $F^{\circ 2}$ possesses a unique fixed point in $[0,1]$, and $F^{\circ 2}$ is increasing, we thus deduce that $F^{\circ 2H}$ converge uniformly to the unique fixed point $\gamma$ of $F^{\circ 2}$. We deduce that both $F^{\circ 2H}$ and $F^{\circ (2H+1)}$ converge uniformly to the constant $\gamma$ as $H \rightarrow +\infty$. We then observe as before that the distribution of $i^{H}_-(o,v)$ for any $v \sim o $ is exactly $F^{\circ H}(1)$ and the distribution of $i_+^H(o,v)$ is $F^{\circ H}(0)$. So they must almost surely converge to random variables whose distribution is the unique fixed point of $F^{\circ 2}$, which is exactly $\gamma \delta_0 + (1-\gamma) \delta_1$.
\end{proof}

An important, and immediate, consequence of Lemma~\ref{lem:subcriticalproj} and its proof is the following.

\begin{prop} \label{prop:ilocal}
    Under condition~\ref{eq:uniquefixedpoint}, for any edge weight distribution, the macroscopic marginal $i$ in $\left(\mathbb{T},o,w,(i,Z)(u,v)_{(u,v)\in \overset{\rightarrow}{\mathbb{T}}}\right)$ from Proposition~\ref{prop:Zconstruction} is measurable in the weightless tree $(\mathbb{T},o)$.  In addition, $(\mathbb{T},o)$-\emph{a.e.}, there exists a unique family $(i(u,v))_{(u,v)\in\overset{\rightarrow}{E}(\mathbb{T})}$ verifying
    \begin{equation}\label{eq:Irecursion}
        \forall (u,v) \in \overset{\rightarrow}{E}(\mathbb{T}) , \, i(u,v)=\max\left(0,\max_{\substack{u' \sim v \\ u'\neq u }} (1-i(v,u')) \right).
    \end{equation}
\end{prop}
\begin{proof}
    The proof is immediate as there is almost surely a unique family of macroscopic marginals verifying the macroscopic recursion~\eqref{eq:Irecursion}.
\end{proof}

\bigskip

The next ingredient we require is the fact that the span of maximum weight maximum size matchings over all possible weights is precisely the set of all maximal matchings.
\begin{lemma}\label{lem:iidmax}
    Let $G$ be a finite graph. Consider $M_{\mathrm{opt}}$ to be the maximum weight maximum size matching on $G$ when we pick a family $w(e)_{e \in E(G)}$ of weights. Formally, define $W^{\mathrm{bad}}\subseteq \mathbb{R}_+^{E(G)}$ to be the set of weights such that there are ties, and $\mathcal{M}_{\mathrm{max}}$ to be the set of maximal matchings on $G$. Then $M_{\mathrm{opt}}$ is the map defined as
    \begin{align*}
        \mathbb{R}_+^{E(G)} \setminus W^{\mathrm{bad}} &\to \mathcal{M}_{\max} \\
        w(e)_{e \in E(G)} &\mapsto M_{\mathrm{opt}}(w(e)_{e \in E(G)}).
    \end{align*}
     The range of $M_{\mathrm{opt}}$ is precisely $\mathcal{M}_{\max}$.
\end{lemma}
\begin{remark}
    It is important to stress that this statement holds for \textbf{finite} graphs, and that there is no randomness involved.
\end{remark}
\begin{proof}
The proof is immediate. For a given maximal matching $M$, simply assign large weights on $M$ and small weights elsewhere. The set $W^{\mathrm{bad}}$ does not cause any issue as it is included in the subset of $\mathbb{Z}-$dependent weights that is negligible.
\end{proof}

\bigskip

We are now ready to prove our last main result.

\begin{proof}[Proof of Theorem~\ref{thm:mandatory}]
    Fix $\eps,H_0>0$. The goal is to find an integer $N$ such that for every $n > N$, there exists a coupling between    \[ B_{H_0}\left((G_n, \mathbbm{1}_{e \in \bigcap_{M \in \mathcal{M}_{n,\max}}M},\mathbbm{1}_{e \in \bigcap_{M \in \mathcal{M}_{n,\max}} M^{\complement}} ),o_n\right),   \]
    and
    \[ B_{H_0}\left((\mathbb{T}, \mathbbm{1}_{i(\overset{\rightarrow}{e})=i(\overset{\leftarrow}{e})=0}, \mathbbm{1}_{i(\overset{\rightarrow}{e})= i(\overset{\leftarrow}{e})=1}),o\right) \]
    where $(\mathbb T,o,i)$ is defined in Proposition~\ref{prop:ilocal},
    that differ on an event of probability at most $\eps$.
    As before, pick $\eps'$ and $H$ from which the total variation distance given in Lemma~\ref{lem:subcriticalproj} between the two families $i_+^{H}(o,v)_{v \sim o}$ and $i_-^{H}(o,v)_{v \sim o}$ is bounded by $\eps'$.
   
    Using local convergence in total variation, we know that it is possible to find $N$ such that for every $n \geq N$, we can couple $(G_n,o_n)$ and $(\mathbb{T},o)$ such that
    \[ \mathbb{P}\left(B_{H_0+H}(G_n,o_n) = B_{H_0+H}(\mathbb{T},o) \right)>1-\eps'. \]

    Let us write $\Omega^{\mathrm{good}}$ to be the event where the previous equality holds, and $\Omega^{\mathrm{bad}}=(\Omega^{\mathrm{good}})^{\complement}$. Using Lemma~\ref{lem:matchingbord}, we get that in $\Omega^{\mathrm{good}}_G$, for any family of weights that are $\mathbb{Z}-$independent, one can write the optimal matching in $B_H(G_n,o_n)$ as the set of edges such that $(i_M,Z_M)(\overset{\rightarrow}{e})+(i_M,Z_M)(\overset{\leftarrow}{e})<(1,w(e))$ where the family $(i_M,Z_M)$ is specified inside the Lemma.
    Define once again the boundary outward edges by
    \[\partial_{H_0}(G_n,o):= \{(u,v), d(u,o)=H_0, d(v,o)=H_0+1\}.\]
    Use Lemma~\ref{lem:subcriticalproj} and an identical proof as in Theorem~\ref{th:reciproquegamarnik}  to obtain that outside of an event $\Omega^{\mathrm{bad},\star}$ of probability bounded by $|\partial_{H_0}(G_n,o_n)|\eps'$, the family of  random variables $(i_M)(u,v)_{(u,v) \in  \partial_{H_0}(G_n,o)}$ is equal to a family $(i(u,v))_{(u,v) \in \partial_{H_0}(G_n,o)}$ of i.i.d. variables with distribution $\gamma \delta_0 +(1-\gamma)\delta_1$. Such a family depends only on the subtree $(T^{H}_{u,v},v)$ of depth $H$ that $(u,v)$ points to.
    Therefore, outside of $\Omega^{\mathrm{bad}}\cup\Omega^{\mathrm{bad},\star}$, no matter the family of weights, we have that $M_{\mathrm{opt}}$ on $B_{H_0+H}(G_n,o_n)$ agrees with a family of messages $(i_M,Z_M)$ where $i_M$ is equal to an i.i.d family of messages $i$ of distribution $\gamma\delta_0+(1-\gamma)\delta_1$ on the boundary.

    Fix a rooted graph in $\Omega$. Lemma~\ref{lem:iidmax} allows us to state that for every edge $e \in G_n$, $ e \in \bigcap_{M \in \mathcal{M}_{n,\max}M}$ if and only if $e \in M_{\mathrm{opt}}$ for every family of weights picked on $G_n$.
    We will now work on $ \Omega^{\mathrm{very\,good}}:= (\Omega^{\mathrm{bad}}\cup \Omega^{\mathrm{bad},\star})^{\complement}$. We would like to show that for $e \in B_{H_0}(G_n,o_n)$,
    \begin{enumerate}
        \item if $i(\overset{\rightarrow}{e})+i(\overset{\leftarrow}{e})<1$, then deterministically on every family of weights, $e \in M_{\mathrm{opt}}$.
        \item if $i(\overset{\rightarrow}{e})+i(\overset{\leftarrow}{e})=1$, then we can find one family $w(e)_{e \in E(G)}$ such that $e \in M_{\mathrm{opt}}(w(e)_{e \in M(G_n)})$ and another family $w'(e)_{e \in E(G)}$ such that $e \notin M_{\mathrm{opt}}(w'(e)_{e \in E(G_n)})$.
        \item if $i(\overset{\rightarrow}{e})+i(\overset{\leftarrow}{e})>1$, then there is no family of weights such that $e \in M_{\mathrm{opt}}$.
    \end{enumerate}
     The first and last cases are the easy ones. Indeed, on $\Omega^{\mathrm{very\,good}}$, we have that $i_M=i$ no matter the family of weights and $e \in M_{\mathrm{opt}}$ if and only if $(i_M,Z_M)(\overset{\rightarrow}{e})+(i_M,Z_M)(\overset{\leftarrow}{e})\overset{\lex}{<}(1,w(e))$.
     If $i(\overset{\rightarrow}{e})+i(\overset{\leftarrow}{e})=i_M(\overset{\rightarrow}{e})+i_M(\overset{\leftarrow}{e})<1$, then for any family of weights we have that the condition above is satisfied by looking at the first coordinate alone, hence the conclusion. We proceed similarly for the last case.

     The second case is more delicate. The idea is that in this case, there must exist $e_2 \sim e_1$ such that $i(\overset{\rightarrow}{e_2})+i(\overset{\leftarrow}{e_2})=1$, and that fixing $w(e_1)$ and letting $w(e_2)\rightarrow +\infty$ will force $e_2 \in M_{\mathrm{opt}}$ whereas fixing $w(e_2)$ and letting $w(e_1) \rightarrow +\infty$ will force $e_1 \in M_{\mathrm{opt}}$. Doing such modifications may have an impact on the boundary conditions of the messages in $B_{H_0+H}(G_n,o_n)$ that we are not able to control. Nevertheless, we can once again uniformly control the behaviour of the messages at radius $H$ by bounding $(i_M,Z_M)$  with the families $(i_-,Z_-)$ and $(i_+,Z_+)$. Our last argument relies on the following lemma whose proof is postponed to the end of this section.
     \begin{lemma}\label{lem:chirurgieloc}
         Fix a rooted graph in $\Omega^{\mathrm{very\, good}}$ then fix a family of weights $w(e)_{e \in E(G_n)}$ and pick $e_1 \in B_H(G_n,o_n)$.
         Assume that $i(\overset{\rightarrow}{e_1})+i(\overset{\leftarrow}{e_1})=1$.
         Then there exists $e_2 \sim e_1$ such that $i(\overset{\rightarrow}{e_2})+i(\overset{\leftarrow}{e_2})=1$ and:
         \begin{itemize}
             \item If $e_1\in M_{\mathrm{opt}}$ then there exists $a>0$ such that setting $w'(e_2)=a$ and $w'=w$ elsewhere gives $\max(Z_-(\overset{\rightarrow}{e_2}),Z_+(\overset{\rightarrow}{e_2}))+\max(Z_-(\overset{\leftarrow}{e_2}),Z_+(\overset{\leftarrow}{e_2}))<w'(e_2)=a.$
             \item If $e_1 \notin M_{\mathrm{opt}}$, then there exists $b>0$ such that setting $w'(e_1)=b$ and $w'=w$ elsewhere gives $\max(Z_-(\overset{\rightarrow}{e_1}),Z_+(\overset{\rightarrow}{e_1}))+\max(Z_-(\overset{\leftarrow}{e_1}),Z_+(\overset{\leftarrow}{e_1}))<w'(e_1)=b$.
         \end{itemize}
     \end{lemma}
 
     We can now apply the inequality $(i_M,Z_M)\overset{\lex}{\leq} \maxlex((i_+,Z_+),(i_-,Z_-))$ along with equality $i_+=i_-=i=i_M$ to obtain that we have found one family of weights such that $(i_M,Z_M)(\overset{\rightarrow}{e}_1)+ (i_M,Z_M)\overset{\leftarrow}({e_1})<w'(e_1)$ so $e_1 \in M_{\mathrm{opt}}$, and another family of weights that induces  $e_2 \in M_{\mathrm{opt}}$. We have found one maximum size matching that includes $e_1$ and another that includes $e_2$ instead. In particular, this implies that the edge $e_1$ is neither mandatory nor blocking.
\end{proof}

\begin{proof}[Proof of Lemma~\ref{lem:chirurgieloc}.]
        Let us first show the existence of $e_2$.
        Write $\overset{\rightarrow}{e_1}=(e_{-,1},e_{+,1})$ and assume without loss of generality that $i(e_{-,1},e_{+,1})=0$ and $i(e_{+,1},e_{-,1})=1$.
        Since $i(e_{+,1},e_{-,1})=1$, this implies that there exists $u \sim e_{-,1}$ such that $i(e_{-,1},u)=0$.
        Now, using the fact that $i(e_{-,1},e_{+,1})=0$ we must have that $i(u,e_{-,1})=1$ so $e_2=\{u,e_{-,1}\}$ satisfies the required condition.
        
        It seems natural that in either case, picking $a=b>\sum_{e \in E(G_n) \setminus e_j}w(e)$ should suffice for $j=1,2$. For simplicity sake, we will take $a=b=2\sum_{e \in E(G_n)} w(e)$.
        Let us treat the first case as the other one is symmetrical. Setting $w'(e_2)=a$ and $w'=w$ elsewhere does not modify the value of $Z_{+}(\overset{\rightarrow}{e_2})$, $Z_{+}(\overset{\leftarrow}{e_2})$, $Z_{-}(\overset{\rightarrow}{e_2})$ and $Z_{-}(\overset{\leftarrow}{e_2})$  since these do not depend on the boundary conditions but only on a subset of the weighted graph that do not include $e_2$.
        The following argument works for all 4 variables in question so we will shorten the notation to $Z_{\pm}(\overset{\leftrightarrow}{e_2})$ to mean any of the variables. Using Recursion~\eqref{eq:systemrecursionzoombis}, we know that there exists a (directed) path $e_2,e_3,\dots,e_{m}$ in $B_H(G_n,o_n)$ from $\overset{\leftrightarrow}{e_2}$ to $\overset{\rightarrow}{e_{m}}$ such that  
        \begin{align*}
            Z_{\pm}(\overset{\leftrightarrow}{e_2})&=\sum_{j=0}^{m-3} (-1)^j w(e_{j+3})  .
        \end{align*}
        This allows us to write the upper bound
        \begin{align}
            Z_{\pm}(\overset{\leftrightarrow}{e_2})\leq \sum_{e\in E(G_n)\setminus e_2}w(e)\leq \sum_{e\in E(G_n)}w(e) = \frac{a}{2} .
        \end{align}
        The same reasoning for all four variables work identically, which gives the required inequality.
    \end{proof}

\section{Discussion and related problems}\label{sec:final}

\subsection{Generalisation to lexicographic weights, min cost max gain  matching}
In this subsection, we introduce the formalism that is the natural generalisation of the previous results to the min cost max gain matching problem. As we will see, the situation is very similar to this article and our methods should apply. We only present the general framework and will not go into a detailed analysis.

The main idea is that we treated the case where we lexicographically maximize the sum of $(1,w(u,v))$ on every edge $(u,v)$, instead one could in general lexicographically maximize $(w_1(u,v),w_2(u,v))$ with $((w_1,w_2)(u,v))_{(u,v) \in E}$ i.i.d and $w_2$ having a continuous distribution. If one picked $w_2$ to be negative variables and interpreted $w_1$ as the gains and $w_2$ as the costs, then this matching minimises the cost under the constraint of maximising the gain, hence the terminology. The corresponding RDE on variables $Z=(Z_1,Z_2)$  becomes:
\begin{equation}\label{eq:generalisationlexico}
    \begin{aligned}
        (Z_1,Z_2)(u,v)=\maxlex \left( (0,0) , \maxlex _{\substack{ u' \sim v \\ u' \neq u}}\left( (w_1,w_2)(v,u')  -(Z_1,Z_2)(v,u') \right) \right)
    \end{aligned}
\end{equation}
By projecting on the first component, the variable $Z_1$ satisfies the recursion of the classic unidimensional RDE when one only seeks to optimize one weight. However, without specifying anything on the distribution $w_1$, one cannot hope to say more on solutions of Equation~\eqref{eq:generalisationlexico}. Furthermore, there are problems when trying to solve the corresponding RDE with existence theorems using compactness since $\mathbb{R}^2$ equipped with the lexicographical topology is not a Polish space. One needs to proceed by studying the maximum matching under weights $w_1+\eps   w_2$ and let $\eps \rightarrow 0$. This crucial step is not straightforward and is not a consequence of the present work. Nevertheless, if one managed to obtain the existence of those limits, the rest of the analysis would be similar with $(i_{\mathcal{H}},Z_{\mathcal{H}})$ replaced by $(Z_{1,{\mathcal{H}}},Z_{2,{\mathcal{H}}})$.

\subsection{Relationship between maximum weight maximum size matching and uniform maximum size matchings}

In this subsection we show that, in a precise sense, maximum weight maximum size matchings and uniform maximum size matchings are fundamentally different.

Let $G_n$ be a sequence of \textbf{unweighted} random graphs that converge in Benjamini-Schramm sense to a unimodular Bienaymé-Galton-Watson tree $(\mathbb{T},o)$ with reproduction law $\pi$.
For any $n$, we can consider $M_{n,\mathcal{U}}$ to be a maximum size matching uniformly picked among all maximum size matchings of $G_n$.

From compactness arguments, the sequence $(G_n,M_{n,\mathcal{U}})$ gives rise to unimodular subsequent limits $(\mathbb{T},o,\mathbb{M}_{\mathcal{U}})$ where $\mathbb{M}_{\mathcal{U}}$ must be a unimodular matching with the same size as the ones we have constructed in this paper.

For $\omega$ with finite expectation, let us write $(\mathbb{T},o,\mathbb{M}^{\omega})$ to be the unique optimal matching distribution when one samples i.i.d. weights with distribution $\omega$ on the edges of $\mathbb{T}$ given by Theorem~\ref{maintheorem}.

A natural question to ask is the relationship between $\mathbb{M}_{\mathcal{U}}$ and the set $\{ \mathbb{M}^{\omega}\}_\omega$ with $\omega$ spanning distributions with finite expectation. The proposition below shows that they are well separated.

\begin{prop}\label{prop:pasuniforme}
    Consider $C_{w}(\pi)$ the convex hull of the distributions of variables of the form
    \[ \bigg\{ (\mathbb{T},o,\mathbb{M}^{\omega}), \omega \text{ with finite expectation} \bigg\}   .\]
    Let $C_{\mathcal{U}}(\pi)$ to be the convex hull of subsequent Benjamini-Schramm limits of $(G_n,M_{n,\mathcal{U}}(G_n))$ where $(G_n)$ is any finite graph sequence converging in Benjamini-Schramm sense to $(\mathbb{T},o)$.
    For two distributions $\mu$ of a variable $(\mathbb{T},o,\mathbb{M})$ and $\mu'$ of a variable $(\mathbb{T},o,\mathbb{M}')$, set the total variation distance \[d_{\mathrm{TV}}(\mu,\mu')=\sup_{ \text{ events }A} | \mu(A)-\mu'(A)| .\]
    Then for any $\pi$ with finite expectation, satisfying $\hat{\pi} \neq \delta_0$ and $\hat{\pi}_0>0$, we have 
    \[ d_{\mathrm{TV}}(C_w(\pi),C_{\mathcal{U}}(\pi))> 0 . \]
\end{prop}
As a consequence, we deduce that it is impossible to approach uniform maximum size matchings with maximum weight matchings on i.i.d. weights.
\begin{proof}
    Let $p= \min \{ j>0: \hat{\pi}_j>0 \}$. Define $A$ to be the event that $\deg(o)=p+1$, every neighbour $v_o,....,v_{p} \sim o$ has degree $p+1$ and every neighbour $v_{j,1},\dots,v_{j,p} \sim v_{j}, v_{j,j'}\neq o$ is a leaf. Set $B$ the event that $o$ is matched to one of its neighbours $v_j$. See Figure~\ref{fig:notuniform} for an illustration of this event.
    \begin{figure}
        \centering
        \includegraphics[width=0.5\linewidth]{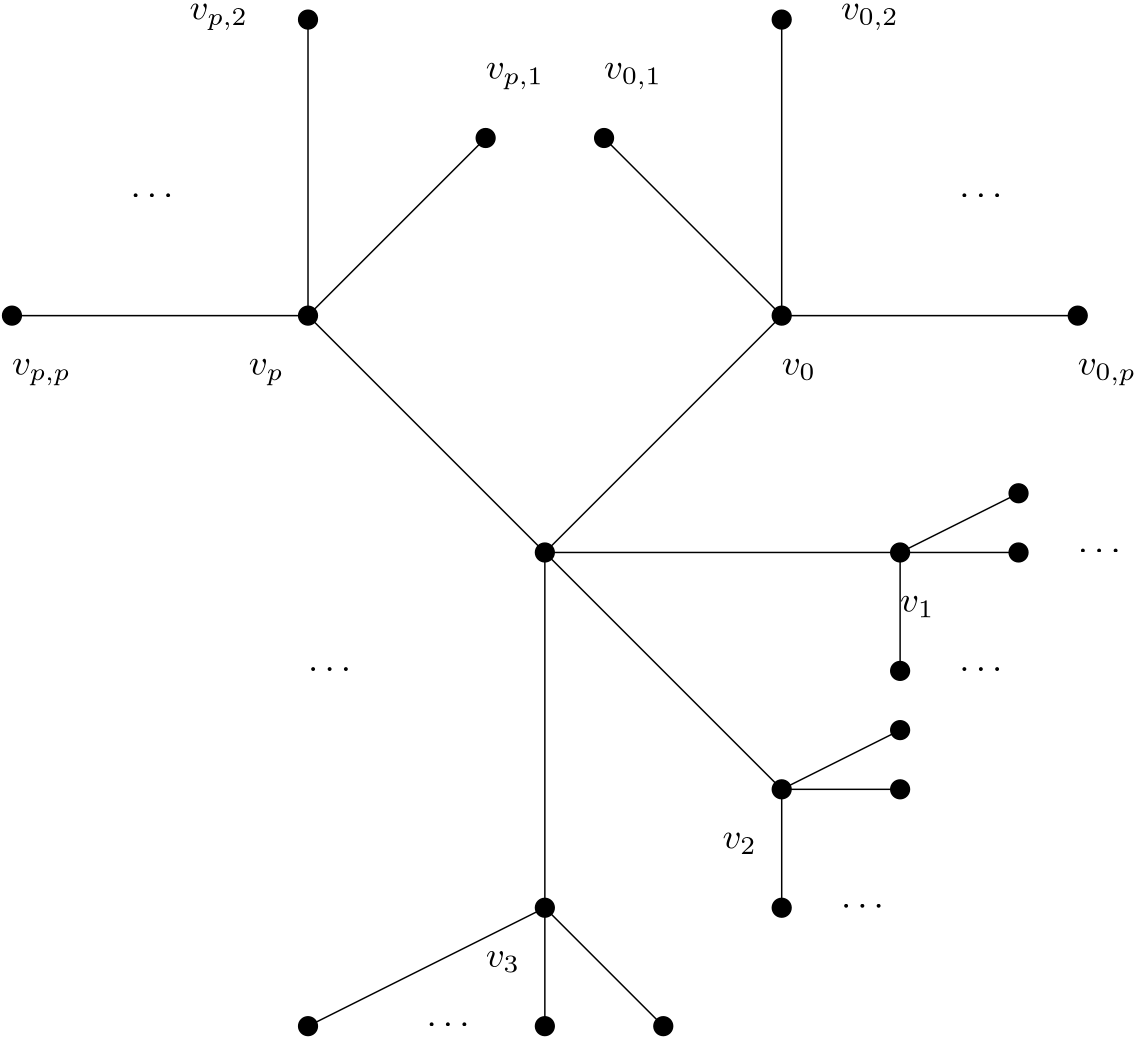}
        \caption{The event $A$}
        \label{fig:notuniform}
    \end{figure}
    
   On the event $A$, where $\mathbb{T}$ is finite, there are $p^{p+1}$ maximum size matchings that leave $o$ unmatched, and $(p+1)p^{p}$ maximum size matchings that matches $o$ with one of its neighbours. Hence,
    \[ \mathbb{P}_{M_{\mathcal{U}}}(B|A)=\frac{(p+1)p^p}{(p+1)p^p+p^{p+1}} = \frac{1}{1+ \frac{p}{p+1} }. \]
    For the weighted matching problem, pick any weight distribution $\omega$ and sample $w_j=w(o,v_j)$ and $w_{j,j'}=w(v_j,v_{j,j'})$ independently according to $\omega$.
    The event $B$ can then be written as the disjoint union
    \[ \bigcup_{j=0}^{p}  \bigg\{ (i,Z)(o,v_j)+(i,Z)(v_j,o) \overset{\lex}{<} (1,w_j) \bigg \} .  \]
    
    It is easy to see that every edge is neither mandatory nor blocking, so the first component sums to $1$.
    By using the recursion, the messages can be expressed as a function of the weights:
    \begin{align*}
        Z(o,v_j)&= \max_{1 \leq j'\leq p} w_{j,j'}  ,\\
        Z(v_j,o) &=  \max \left(0, \max_{j'' \neq j} (w_{j''}-\max_{1\leq j''' \leq p} w_{j'',j'''}) \right). 
    \end{align*}
    The event $B$ then rewrites into
    \[  \bigcup_{j=0}^{p} \bigg\{ \max_{1\leq j'\leq p}(w_{j,j'})+\max(0, \max_{j''\neq j} (w_{j''}-\max_{1\leq j'''\leq p} w_{j'',j'''} ))  < w_j   \bigg \}.  \]
    We thus obtain that
    \begin{align*}
        \mathbb{P}_{\mathbb{M}^{\omega}}(B|A)=(p+1)\mathbb{P}\left(w_0-  \max_{1\leq j'\leq p}(w_{0,j'}) >\max(0, \max_{0<j''<p} (w_{j''}-\max_{1\leq j'''\leq p} w_{j'',j'''} ))\right).
    \end{align*}
    Set 
    \begin{align*}
       \forall \, 0\leq j\leq p, X_j&:= (w_{j}-\max_{1\leq j'''\leq p} w_{j,j'''} ).
    \end{align*}
    The variables $(X_j)_{0 \leq j \leq p}$ are i.i.d. and $\mathbb{P}(X_j>0)$ is exactly the probability that $w_j$ is the largest among $p+1$ i.i.d. variables of distribution $\omega$, and is thus $\tfrac{1}{p+1}$. With this notation, $\mathbb{P}_{\mathbb{M}^{\omega}}(B|A)$ rewrites as
    \begin{align*}
     \mathbb{P}_{\mathbb{M}^{\omega}}(B|A)&=(p+1)\mathbb{P}\left(  X_0 > \max(0,\max_{1\leq j \leq p} X_j) \right) .
     \end{align*}
     We will now decompose with respect to which of the $X_j$ are positive:
     \begin{align*}
     \mathbb{P}_{\mathbb{M}^{\omega}}(B|A)&=(p+1)\mathbb{P}\left(\bigcup_{S \subseteq \{1,\dots,p\}} \left\{ \forall j' \in S, X_{j'}>0, \forall j'' \notin S, X_{j''} \leq 0 , X_0>0, X_0>\max_{j' \in S} X_{j'} \right\} \right) \\
     &=(p+1)\sum_{S \subseteq \{1,\dots,p\}} \mathbb{P}\left(X_0 > \max_{j' \in S} X_{j'} | X_0>0, \forall j' \in S, X_j'>0 \right) \mathbb{P}(X_0>0)^{|S|+1}\mathbb{P}(X_0\leq 0)^{p-|S|} \\
     &= (p+1) \sum_{S \subseteq \{1,\dots,p\}} \left( \frac{1}{|S|+1} \right) \frac{p^{p-|S|}}{(p+1)^{p+1}},
     \end{align*}
    where we used the fact that conditionally on $X_0>0$ and  $X_{j'}>0$ for $j' \in S$, the family $(X_0, (X_{j'})_{j' \in S})$ are i.i.d. so that the probability that $X_0$ is the largest is $\tfrac{1}{|S|+1}$. Remarkably, this quantity does not depend on the distribution $\omega$.
    Finally, decomposing with respect to the cardinal of $S$:
    \begin{align*}
        \mathbb{P}_{\mathbb{M}^{\omega}}(B|A)&= \frac{p^{p}}{(p+1)^{p}}  \sum_{l=0}^{p}  \binom{p}{l}  \frac{1}{l+1} \left(\frac{1}{p}\right)^{l} \\
        &= \frac{p^{p}}{(p+1)^{p}} \frac{1}{p+1}  \frac{(1+\frac{1}{p})^{p+1}-1}{\frac{1}{p}} \\
        &= \frac{(p+1)^{p+1}-p^{p+1}}{(p+1)^{p+1}} \\
        &= 1- \left(1-\frac{1}{p+1}\right)^{p+1}.
    \end{align*}
    Since $A$ is only a conditioning on $\mathbb{T}$, we get that for any $\omega$, \[ \mathbb{P}_{(\mathbb{T},o,\mathbb{M}^{\omega})}(B \cap A) = \left(1- \left(1-\frac{1}{p+1}\right)^{p+1}\right)\mathbb{P}_{\mathbb{T}}(A)\]
    whereas for any subsequent limit $(\mathbb{T},o,\mathbb{M}^{\mathcal{U}})$ of uniform maximum size matchings, we have that:
    \[ \mathbb{P}_{(\mathbb{T},o,\mathbb{M}^{\mathcal{U}})}(B \cap A)=\frac{1}{1+\frac{p}{p+1}}\mathbb{P}_{\mathbb{T}}(A).\]
    From this we deduce that \[ \forall \mu \in C_{w}(\pi), \mu(A\cap B)=\left( 1- \left(1-\frac{1}{p+1}\right)^{p+1}\right)  \mathbb{P}_{\mathbb{T}}(A) \]
    and
    \[ \forall \mu' \in C_{\mathcal{U}}(\pi), \mu'(A\cap B)= \frac{1}{1+\frac{p}{p+1}}\mathbb{P}_{\mathbb{T}}(A).\]
    As a conclusion, as $\mathbb{P}_{\mathbb{T}}(A)=\pi_{p+1}\hat{\pi}_p^{p+1}\hat{\pi}_0^{p(p+1)}$,
    \begin{align*}
        d_{\mathrm{TV}}(C_{w}(\pi),C_{\mathcal{U}}(\pi)) &\geq \inf_{\mu \in C_w(\pi),\mu'\in C_{\mathcal{U}}(\pi)} |\mu(A \cap B)-\mu'(A \cap B)| \\
        &=\left| \frac{1}{1+\frac{p}{p+1}}-1+ \left(1-\frac{1}{p+1}\right)^{p+1}\right| \pi_{p+1}\hat{\pi}_p^{p+1}\hat{\pi}_0^{p(p+1)} >0.
    \end{align*}
    \end{proof}
\begin{remark}
    We find an effective lower bound in this proof. By considering every possible "central symmetric" configuration of finite trees, one could, in theory, further improve this lower bound.
    To remove the condition $\hat{\pi}_0>0$, one would need to condition on infinite trees which would require studying Benjamini-Schramm limits of uniform maximum size matchings.
\end{remark}

\subsection{A conjecture on mandatory and blocking edges}
In the previous subsection, we have shown that uniform maximum size matchings and maximum weight matching maximum size are well separated. Nevertheless, in Theorem~\ref{thm:mandatory}, we gave a result about mandatory and blocking edges that are universal objects to any maximum size matching.
It is natural to ask whether these results hold outside of the subcritical regime.

For Erd\H{o}s-Renyí random graphs, Theorem~\ref{thm:mandatory} reads:
\begin{coro}
    Let $(\mathbb{T},o)$ be a UBGW with Poisson $\mathcal{P}(c)$ reproduction law with $c<e$.
    Let $(G_n,o_n)$ be a sequence of random weighted graphs converging locally in total variation to $(\mathbb{T},o)$. Define $\mathcal{M}_{\max}(G_n)$ to be the set of maximal matchings on $(G_n)$
    We then have the following convergence:
        \begin{equation}
            \begin{aligned}
                \frac{1}{|E(G_n)|} \sum_{e \in E(G_n)} \mathbbm{1}_{\forall \mathbb{M} \in \mathcal{M}_{\max}(G_n), e \in \mathbb{M}} &\overset{\mathbb{P}}{\longrightarrow} \gamma^2, \\
                 \frac{1}{|E(G_n)|} \sum_{e \in E(G_n)} \mathbbm{1}_{\forall \mathbb{M} \in \mathcal{M}_{\max}(G_n), e \notin \mathbb{M}} &\overset{\mathbb{P}}{\longrightarrow} (1-\gamma)^2.
            \end{aligned}
        \end{equation}
\end{coro}
When contraction does not occur, there are good reasons to believe that a law of large numbers is too much to ask for. Nonetheless, the convergence in expectation may still hold. This naturally leads to the following conjectures:
\begin{conj}\label{conj:ER}
    Let $G_n$ be Erdös-Reny\'i graphs with parameters $(n, \frac{c}{n})$.
    Define $\mathcal{M}_{\max}(G_n)$ to be the set of maximum size matchings of $G_n$. Writing $\underline{\gamma}$ to be the smallest solution to $x=e^{-ce^{-cx}}$ and $\overline{\gamma}=e^{-c\underline{\gamma}}$, we then have the following:
        \begin{equation}
            \begin{aligned}
                \lim_{n \rightarrow +\infty} \mathbb{E}\left[ \frac{1}{|E(G_n)|} \sum_{e \in E(G_n)} \mathbbm{1}_{\forall \mathbb{M} \in \mathcal{M}_{\max}(G_n), e \in \mathbb{M}}\right] &= \underline{\gamma}^2, \\
                \lim_{n \rightarrow +\infty} \mathbb{E}\left[ \frac{1}{|E(G_n)|} \sum_{e \in E(G_n)} \mathbbm{1}_{\forall \mathbb{M} \in \mathcal{M}_{\max}(G_n), e \notin \mathbb{M}}\right] &= (1-\overline{\gamma})^2.
            \end{aligned}
        \end{equation}
\end{conj}
We also state its companion conjecture in the general case.
    \begin{conj}\label{conj:desirables}
        Under the setting and notation of Theorem~\ref{maintheorem2}, let us write $\mathcal{M}_{\max}(G_n)$ the set of maximum size matchings in $G_n$, then:
        \begin{equation}
            \begin{aligned}
                \lim_{n \rightarrow +\infty} \mathbb{E}\left[ \frac{1}{|E(G_n)|} \sum_{e \in E(G_n)} \mathbbm{1}_{\forall M \in \mathcal{M}_{\max}(G_n), e \in M}\right] = \mathbb{P}_{((i,Z),(i',Z'))\sim \zeta'\otimes\zeta'}\left( i+i'<k \right), \\
                \lim_{n \rightarrow +\infty} \mathbb{E}\left[ \frac{1}{|E(G_n)|} \sum_{e \in E(G_n)} \mathbbm{1}_{\forall M \in \mathcal{M}_{\max}(G_n), e \notin M}\right] = \mathbb{P}_{((i,Z),(i',Z'))\sim \zeta'\otimes\zeta'}\left( i+i'>k \right).
            \end{aligned}
        \end{equation}
    \end{conj}

\bigskip

For the local geometry, we have proved in Theorem~\ref{uniqueness} that under some generic condition on $\pi$, regardless of the weight distribution $\omega$, the distribution of the macroscopic variables $i$ does not depend on the weight distribution. As a consequence, in the construction of the random variable $\left( \mathbb{T},o,(i,Z)(u,v)_{(u,v)\in \overset{\rightarrow}{E}(\mathbb{T})} \right)$ by Proposition~\ref{prop:Zconstruction}, the law of the marginal $\left(\mathbb{T},o,i(u,v)_{(u,v) \in \overset{\rightarrow}{E}(\mathbb{T})} \right)$ does not depend on the chosen weight distribution $\omega$. This indicates that $i$ might also represent an universal object that only depends on the unweighted tree and leads to the following conjecture:
\begin{conj}\label{conj:desirablesgeom}
    Under the setting and notation of Theorem~\ref{maintheorem2}, let us write $\mathcal{M}_{\max}(G_n)$ to be the set of maximum size matchings on $G_n$, then the marked random graphs
    \[ \left(G_n,\mathbbm{1}_{\bigcap_{M \in \mathcal{M}_{\max}(G_n)}M},\mathbbm{1}_{\bigcap_{M \in \mathcal{M}_{\max}(G_n)}M^{\complement}}  \right)  \]
    converges in Benjamini Schramm sense to
    \[ \left(\mathbb{T},o,\mathbbm{1}_{i(\overset{\rightarrow}{e})+i(\overset{\leftarrow}{e})<k},\mathbbm{1}_{i(\overset{\rightarrow}{e})+i(\overset{\leftarrow}{e})>k} \right) . \]
\end{conj}

\section{Proof of Proposition~\ref{prop:russe}}\label{sec:Appendix}

Recall \Cref{prop:russe}:
\proprusse*
We will first prove the equivalent proposition for unidimensional variables:
\begin{lemma}
    \label{prop:russeuni}
    Let $Y,Y'$ be two real-valued random variables such that $Y\overset{\mathcal{L}}{=}Y'$.
    Then \[ \mathbb{E}[(Y-Y')_+]=\mathbb{E}[(Y-Y')_-]. \]
\end{lemma}

\begin{proof}[Proof of Lemma~\ref{prop:russeuni}]
We will break down the proof into three steps, we will first prove it assuming $Y-Y'$ is integrable, then nonnegative, then without any assumptions on $Y-Y'$.

Let us assume $Y-Y'$ is integrable, we will show that $\mathbb{E}[Y-Y']=0$.
    Write $Y'+(Y-Y')=Y$, set $H=Y-Y'$, then:
    \begin{align*}
        \forall t \in \mathbb{R}, \mathbb{E}\left[ e^{it(Y+H)}\right]=\mathbb{E}\left[ e^{itY}\right] 
        &\Rightarrow \forall t \in \mathbb{R}, \mathbb{E}\left[ e^{itY} \left(e^{itH}-1\right) \right]=0, \\
        &\Rightarrow \forall t \in \mathbb{R}, \mathbb{E}\left[ e^{itY} \frac{\left(e^{itH}-1\right)}{t} \right]=0.
    \end{align*}
    The integrand is dominated by $|H|$ which is integrable, so by dominated convergence:
    \begin{align*}
        \mathbb{E}\left[ e^{itY} \frac{\left(e^{itH}-1\right)}{t} \right] \underset{t \rightarrow 0}{ \rightarrow} \mathbb{E}[iH].
    \end{align*}
    Hence $\mathbb{E}[H]=0$.

Now let us assume that $Y$ and $Y'$ are nonnegative. If the difference is integrable then we are already done. Assume $(Y-Y')$ is not integrable, we will proceed by contradiction.
Let us assume without loss of generality that $\mathbb{E}\left[ H_{+}\right]<+\infty$ and $\mathbb{E}\left[ H_{-}\right]=+\infty$. \\
Since the variables $Y$ and $Y'$ are nonnegative, we can define their Laplace transform on $\mathbb{R}_-$:
    \begin{align*}
        \forall a \in \mathbb{R}_+, \mathbb{E}\left[ e^{-a(Y+H)}\right]=\mathbb{E}\left[ e^{-aY}\right] 
        &\Rightarrow \forall a \in \mathbb{R}_+, \mathbb{E}\left[ e^{-aY} \left(e^{-aH}-1\right) \right]=0, \\
        &\Rightarrow \forall a \in \mathbb{R}_+^{*}, \mathbb{E}\left[ e^{-aY} \frac{\left(e^{-aH}-1\right)}{a} \right]=0.
    \end{align*}
We decompose the above expectation into positive and negative parts:
\begin{align*}
    \forall a \in \mathbb{R}_+^{*}, \mathbb{E}\left[ e^{-aY} \frac{\left(e^{-aH}-1\right)}{a} \mathbbm{1}_{H>0} \right]+\mathbb{E}\left[ e^{-aY} \frac{\left(e^{-aH}-1\right)}{a} \mathbbm{1}_{H<0} \right]:=I_+(a)+I_-(a)=0.
\end{align*}
The integrand in $I_+(a)$ is dominated by $H_+$ so by dominated convergence theorem:
\begin{align*}
    \lim_{a \rightarrow 0^+} I_+(a)= -\mathbb{E}[H_+].
\end{align*}
For $I_-(a)$, for every $M>0$, define $Z_M= \max(H_-,M)$, then:
\begin{align*}
    I_-(a) \geq \mathbb{E}\left[ e^{-aY} \frac{\left(e^{a Z_M}-1\right)}{a} \mathbbm{1}_{H_->0} \right].
\end{align*}
Now in the second expectation, the integrand is bounded by $Me^M$, so fixing $M$ and taking $a \rightarrow 0^+$ yields by dominated convergence theorem:
\begin{align*}
    \liminf_{a \rightarrow 0^+} I_-(a) \geq \mathbb{E}\left[ Z_M \mathbbm{1}_{H_->0} \right]= \mathbb{E}\left[H_- \mathbbm{1}_{H_-\leq M} \right]. 
\end{align*}
Taking $M \rightarrow \infty$ yields:
\begin{align*}
    \liminf_{a \rightarrow 0^+} I_-(a) \geq \mathbb{E}\left[ H_-\right]= +\infty.
\end{align*}
In conclusion we have shown that $I_-(a)+I_+(a)=0$ for every $a>0$, that $I_+(a)$ is bounded but that $I_-(a)$ diverges to $+\infty$ as $a \rightarrow 0^+$. This is clearly a contradiction.

Now we will lift the nonnegativity requirement by decomposing $Y$ and $Y'$ into positive and negative parts. We assume that $H_+$ is integrable and we will show that $H_-$ is also integrable.
Writing $Y=Y_+-Y_-$ and $Y'=Y_+'-Y_-'$ we have that:
\begin{align*}
    (Y_+-Y_+')_+ \leq (Y-Y')_+.
\end{align*}
Indeed 
\begin{enumerate}
\item If $Y\geq 0$ and $Y' \geq 0$, there is equality.
\item If $Y <0$ then the left-hand side is $(-Y_+')_+=0$ and the right-hand side is positive, so it holds.
\item If $Y\geq 0$ and $Y'<0$ then the left hand side is $Y_+$ and the right hand side is $Y_+-Y'>Y_+$ so it still holds.
\end{enumerate}
Since $H_+$ is integrable by hypothesis, the previous discussion thus implies that $(Y_+-Y_+')_+$ is also integrable.
By symmetry, we deduce that $(Y_-'-Y_-)_+$ must also be integrable.

Since $Y_+$ and $Y_+'$ have the same law, applying the nonnegative result to the couple $(Y_+,Y_+')$ yields that $(Y_+-Y_+')$ must be integrable.
By symmetry, $(Y_-'-Y_-)$ is also integrable.

Now writing that $H=Y-Y'=(Y_+-Y_+')-(Y_--Y_-')$, we have that $H$ is integrable, the first step yields the conclusion.
\end{proof}
As a corollary, we obtain the following unidimensional version of Proposition~\ref{prop:russe}:
\begin{prop}\label{coro:russeuni}
    If $X,X',Y,Y'$ are four real valued random variables such that $X$ is integrable, $Y\overset{\mathcal{L}}{=}Y'$ and $X' \leq 0$ almost surely.
    If almost surely, $X+X'\geq Y-Y'$, then:
    \begin{align*}
        \mathbb{E}[X]\geq 0
    \end{align*}
    with equality implying $X'=0$ almost surely.
\end{prop}

\begin{proof}[Proof of Proposition~\ref{coro:russeuni}]
    Since $X'\leq 0$, we have that $X\geq Y-Y'$ which in turn implies that $(Y-Y')_+$ must be integrable. 
    Applying Lemma~\ref{prop:russeuni}, we get that $(Y-Y')_-$ is integrable and hence $Y-Y'$ is integrable and of zero expectation.
    Taking expectation in $X \geq Y-Y'$ yields 
    \begin{align*}
        \mathbb{E}[X] \geq \mathbb{E}[Y-Y']=0.
    \end{align*}

    Now if $\mathbb{E}[X]=0$ then $X-(Y-Y') \geq -X' \geq 0$ implies that $X'$ is integrable and that $\mathbb{E}[-X']=0$. Therefore $X'=0$ almost surely.
\end{proof}
We are now ready to prove Proposition~\ref{prop:russe}:
\begin{proof}[Proof of Proposition~\ref{prop:russe}]
First let us look at the projection on the first marginal. The projection on the first marginal converts lexicographic order into the usual order on $\mathbb{R}$.
Let us write $X=(X_1,X_2), X'=(X_1',X_2'), Y=(Y_1,Y_2)$ and $Y'=(Y_1',Y_2')$ we then have:
$X_1+X_1' \geq Y_1-Y_1'$ with $X_1$ integrable, $X_1' \leq 0$ almost surely, and $Y_1 \overset{\mathcal{L}}{=} Y_1'$.
Applying Lemma~\ref{coro:russeuni} to $X_1,X_1',Y_1,Y_1'$, we obtain that $\mathbb{E}[X_1]\geq 0$ with equality if and only if $X_1'=0$ almost surely.

If $\mathbb{E}[X_1]>0$ then we must have $\mathbb{E}[X] \overset{\lex}{>}(0,0)$ and we are done. Else, $\mathbb{E}[X_1]=0$ in which case $X_1'=0$ almost surely.
Furthermore, by the inequality $X_1-Y_1+Y_1' \geq 0$ and $\mathbb{E}[X_1-Y_1+Y_1]=0$ we deduce that almost surely, $X_1=Y_1-Y_1'$.
In this case, $X+X' \overset{\lex}{\geq} Y-Y'$ becomes $X_2+X_2'\geq Y_2-Y_2'$ as there is always equality on the first marginal.
Once again, $X_2$ is integrable, $X_2'\leq 0$ almost surely since $X_1'=0$ and $X' \overset{\lex}{\leq }(0,0)$, and $Y_2$ and $Y_2'$ have the same law as the images of the second projection of $Y,Y'$.
Now we simply apply Proposition~\ref{coro:russeuni} to $X_2,X_2',Y_2,Y_2'$ to deduce that $\mathbb{E}[X_2]\geq 0$ with equality implying $\mathbb{E}[X_2']=0$.

Putting everything together, we have shown that $\mathbb{E}[X_1]>0$ or $\mathbb{E}[X_1]=0$ and $\mathbb{E}[X_2]\geq 0$ which is equivalent to $\mathbb{E}[X] \overset{\lex}{\geq }(0,0)$. Furthermore, if there is equality then $X_1'=X_2'=0$ almost surely which means that $X'=(0,0)$ almost surely.
\end{proof}

\addcontentsline{toc}{section}{References}
\printbibliography

\bigskip

{Nathanaël Enriquez: \href{mailto:nathanael:enriquez@universite-paris-saclay.fr}{nathanael.enriquez@universite-paris-saclay.fr}}\\
{Laboratoire de Mathématiques d’Orsay, CNRS, Université Paris-Saclay, 91405, Orsay, France
and DMA, \'Ecole Normale Supérieure – PSL, 45 rue d’Ulm, F-75230 Cedex 5 Paris, France}

\medskip

{Mike Liu: \href{mailto:mike.liu@ensae.fr}{mike.liu@ensae.fr}}\\
{ENSAE, Fairplay joint team, CREST France}

\medskip

{Laurent Ménard: \href{mailto:laurent.menard@normalesup.org}{laurent.menard@normalesup.org}}\\
{Modal'X, UPL, Université Paris-Nanterre, F92000 Nanterre, France}  

\medskip

{Vianney Perchet: \href{mailto:vianney.perchet@ensae.fr}{vianney.perchet@ensae.fr}}\\
{ENSAE \& Criteo AI Lab (Fairplay team), France}
\end{document}